\documentclass[12pt]{article}
\usepackage{amsmath}
\usepackage{amssymb}
\usepackage{amsthm}
\usepackage{mathrsfs}
\usepackage{mathtools}
\usepackage[top=1in, bottom=1in, left=0.8in, right=1in]{geometry}
\usepackage{multicol}
\usepackage{wrapfig}
\usepackage[retainorgcmds]{IEEEtrantools}
\usepackage{upgreek}
\usepackage[T1]{fontenc}
\usepackage{stmaryrd}
\usepackage{xcolor}
\usepackage[backend=biber]{biblatex}
\usepackage{tikz-cd}
\usepackage{xy}
\input xy
\xyoption{all}

\usepackage{url}

\usepackage{breakurl}
\usepackage[breaklinks]{hyperref}

\newcommand{\Z}{\mathbb Z}
\newcommand{\Q}{\mathbb Q}

\newcommand{\F}{\mathbb F}

\newcommand{\id}{\mathrm{id}}

\DeclareMathOperator{\Hom}{Hom}
\DeclareMathOperator{\im}{im}

\DeclareMathOperator{\Spec}{Spec}

\DeclareMathOperator{\Spf}{Spf}

\DeclareMathOperator{\Sat}{Sat}

\newcommand{\DC}{\mathbf{DC}}

\newcommand{\DA}{\mathbf{DA}}
\newcommand{\sat}{\mathrm{sat}}
\newcommand{\str}{\mathrm{str}}
\newcommand{\W}{\mathcal{W}}
\newcommand{\TD}{\mathbf{TD}}

\DeclareMathOperator{\Alg}{\mathrm{Alg}}
\DeclareMathOperator{\CAlg}{\mathrm{CAlg}}
\newcommand{\dR}{\mathrm{dR}}
\newcommand{\red}{\mathrm{red}}

\newcommand{\Ab}{\mathbf{Ab}}
\DeclareMathOperator{\Gr}{\mathbf{Gr}}

\newcommand{\cris}{\mathrm{cris}}
\newcommand{\Cris}{\mathrm{Cris}}
\newcommand{\zar}{\mathrm{zar}}

\newcommand{\AffEt}{\mathrm{Aff\acute{E}t}}

\newcommand{\dRWM}{\mathrm{dRWM}}
\newcommand{\dRWLM}{\mathrm{dRWLM}}
\newcommand{\Frob}{\mathrm{Frob}}

\newcommand{\op}{\mathrm{op}}

\newcommand{\brackets}{[ \ ]}

\newcommand{\lmodDC}{{\operatorname{\mathbf{-mod}_{\DC}}}}
\newcommand{\lmodstr}{{\operatorname{\mathbf{-mod}_{\str}}}}

\newcommand{\rmodDC}{{\operatorname{\mathbf{mod}_{\DC}-}}}
\newcommand{\rmodstr}{{\operatorname{\mathbf{mod}_{\str}-}}}
\newcommand{\ldgmod}{{\operatorname{\mathbf{-dgmod}}}}

\newcommand{\lalg}{{\operatorname{\mathbf{-alg}}}}

\makeatletter
\newcommand{\colim@}[2]{%
  \vtop{\m@th\ialign{##\cr
    \hfil$#1\operator@font colim$\hfil\cr
    \noalign{\nointerlineskip\kern1.5\ex@}#2\cr
    \noalign{\nointerlineskip\kern-\ex@}\cr}}
}
\newcommand{\colim}{%
  \mathop{\mathpalette\colim@{\rightarrowfill@\scriptscriptstyle}}\nmlimits@
}
\newcommand{\lim@}[2]{%
  \vtop{\m@th\ialign{##\cr
    \hfil$#1\operator@font lim$\hfil\cr
    \noalign{\nointerlineskip\kern1.5\ex@}#2\cr
    \noalign{\nointerlineskip\kern-\ex@}\cr}}
}
\newcommand{\limarrow}{%
  \mathop{\mathpalette\lim@{\leftarrowfill@\scriptscriptstyle}}\nmlimits@
}
\renewcommand{\varprojlim}{%
  \mathop{\mathpalette\varlim@{\leftarrowfill@\scriptscriptstyle}}\nmlimits@
}
\renewcommand{\varinjlim}{%
  \mathop{\mathpalette\varlim@{\rightarrowfill@\scriptscriptstyle}}\nmlimits@
}
\makeatother

\newcounter{ctr}[subsection]

\theoremstyle{plain}

\newtheorem{theorem}[ctr]{Theorem}
\newtheorem{proposition}[ctr]{Proposition}

\newtheorem{lemma}[ctr]{Lemma}
\newtheorem{corollary}[ctr]{Corollary}

\theoremstyle{definition}
\newtheorem{mydef}[ctr]{Definition}
\newtheorem{example}[ctr]{Example}
\newtheorem{remark}[ctr]{Remark}
\newtheorem{para}[ctr]{}

\newtheorem{construction}[ctr]{Construction}

\newtheorem{situation}[ctr]{Situation}
\newtheorem{altdef}[ctr]{Alternative Definition}

\numberwithin{ctr}{subsection}
\numberwithin{equation}{ctr}

\title{Saturated de Rham--Witt complexes with unit-root coefficients}
\author{Ravi Fernando -- \texttt{ravif@illinois.edu}}
\date{}
\begin{document}

\maketitle

\begin{abstract}
The saturated de Rham--Witt complex, introduced by Bhatt--Lurie--Mathew, is a variant of the classical de Rham--Witt complex which provides a conceptual simplification of the construction and which is expected to produce better results for non-smooth varieties.  In this paper, we introduce a generalization of the saturated de Rham--Witt complex which allows coefficients in a unit-root $F$-crystal.  We define our complex by a universal property in a category of so-called de Rham--Witt modules.  We prove a number of results about it, including existence, quasicoherence, and comparisons to the de Rham--Witt complex of Bhatt--Lurie--Mathew and (in the smooth case) to crystalline cohomology and the classical de Rham--Witt complex with coefficients.
\end{abstract}

\tableofcontents


\section{Introduction}
\label{ch:intro}

\subsection{The classical de Rham--Witt complex}
\label{sec:intro_classical_dRW}

Suppose $X$ is a variety over a perfect field $k$ of characteristic $p$.  A central object of study is the crystalline cohomology $H^*_{\cris}(X/W(k))$, conceived by Grothendieck and developed by Berthelot using the crystalline site.  Crystalline cohomology is a Weil cohomology theory; it serves as a characteristic--0 lift of algebraic de Rham cohomology in characteristic $p$, and it fills the gap among $\ell$-adic \'etale cohomology theories at $\ell = p$.

Following ideas of Bloch and Deligne, Illusie's classic paper \cite{illusie} associates to $X$ a pro-complex $(W_r \Omega^*_X)_r$, called the \emph{de Rham--Witt pro-complex} of $X$, whose inverse limit $W\Omega^*_X$ is called the \emph{de Rham--Witt complex} of $X$.  These are $p$-adic lifts of the de Rham complex $\Omega^*_{X/k} = W_1 \Omega^*_X$, and when $X/k$ is smooth, they compute crystalline cohomology in the same sense that the de Rham complex computes algebraic de Rham cohomology.  More precisely, Illusie shows that these complexes are representatives of the derived pushforward of the crystalline structure sheaf to the Zariski site, and therefore that their hypercohomology recovers crystalline cohomology:

\begin{theorem} \label{thm:illusie_thm}
(\cite[II, Th\'eor\`eme 1.4]{illusie})
Suppose $X/k$ is smooth.  Then for each $r > 0$, we have an isomorphism
$$
Ru_{X/W_r *} \mathcal O_{X/W_r} \overset{\sim}{\to} W_r\Omega^*_X
$$
in the derived category $D(X_{\zar}, W_r(k))$ of sheaves of $W_r(k)$-modules on $X$.
Passing to hypercohomology, this induces isomorphisms
\begin{align*}
H^*_{\cris}(X/W_r) & \overset{\sim}{\to} \mathbb H^*(X, W_r \Omega^*_X) := R^* \Gamma(X_{\zar}, W_r \Omega^*_X) \text{ and} \\
H^*_{\cris}(X/W) & \overset{\sim}{\to} \mathbb H^*(X, W \Omega^*_X) := R^* \Gamma(X_{\zar}, W \Omega^*_X).
\end{align*}
\end{theorem}

In the decades since the introduction of the de Rham--Witt complex, there has been a wealth of further work
to generalize its construction, understand its structure, and study its implications for crystalline cohomology and characteristic-$p$ algebraic geometry more broadly.
To give just a few examples:  we now have de Rham--Witt complexes with coefficients (\cite{etesse}) and over a nontrivial base (\cite{L-Z}); a good understanding (\cite{illusie}, \cite{I-R}) of the two spectral sequences computing crystalline cohomology in terms of de Rham--Witt, and proofs of the K\"unneth and duality formulas for crystalline cohomology via de Rham--Witt (\cite{ekedahl2}, \cite{ekedahl1}).


\subsection{The saturated de Rham--Witt complex}
\label{sec:intro_sat_dRW}

Recently, Bhatt--Lurie--Mathew (\cite{BLM}) introduced a variant called the \emph{saturated} de Rham--Witt complex.  This is denoted $\W\Omega^*_X$, with a calligraphic $\W$ to distinguish it from the classical de Rham--Witt complex $W\Omega^*_X$.
The construction of $\W\Omega^*_X$ provides a new perspective that simplifies the classical picture in several ways.
Moreover, although the saturated de Rham--Witt complex agrees with the classical one when $X$ is smooth over some $k$, it is expected to have better properties in non-smooth situations---for example, Ogus (\cite{ogus}) has studied $\W\Omega^*_X$ for $X$ with toroidal singularities, showing that it enjoys good finiteness properties, and that in certain situations the quotient $\W_1 \Omega^*_X$ recovers a previously studied subcomplex of the log de Rham complex of $X$.

Let us recall the general outline of Bhatt--Lurie--Mathew's approach.  The construction is affine-local, and most of the work happens in the category of Dieudonn\'e complexes, so we begin by recalling the main features of this category.

\begin{mydef}
(\cite[Definition 2.1.1]{BLM})
A \emph{Dieudonn\'e complex} is a complex $(X^*, d)$ of abelian groups, equipped with a graded group endomorphism $F$ such that we have $dF = pFd$ as morphisms $X^i \to X^{i+1}$ for each $i$.
Dieudonn\'e complexes form a category $\DC$, with the obvious definition of morphisms.
\end{mydef}

The category $\DC$ has two important full subcategories
$$
\DC_{\str} \hookrightarrow \DC_{\sat} \hookrightarrow \DC,
$$
called the categories of \emph{strict} and \emph{saturated} Dieudonn\'e complexes respectively.  Both inclusions have left adjoints, called the \emph{saturation} $\Sat \colon \DC \to \DC_{\sat}$ and the \emph{completion} $\W \colon \DC_{\sat} \to \DC_{\str}$.  (Thus the composition $\W\Sat \colon \DC \to \DC_{\str}$ is left-adjoint to the inclusion $\DC_{\str} \to \DC$; we will call $\W\Sat$ the \emph{strictification}.)  Roughly speaking, a Dieudonn\'e complex is saturated if it is $p$-torsionfree and $F$ is a bijection onto its ``expected image''; the saturation is a suitable colimit which forces this to be so.  If $X^*$ is saturated, it can then be endowed with a Verschiebung operator $V \colon X^i \to X^i$ for each $i$, satisfying $FV = VF = p$ among other identities.  We then call a Dieudonn\'e complex strict if it is saturated and is complete with respect to the filtration given by $(\im V^r + \im dV^r)_r$; $\W$ is the corresponding completion functor.

An especially pleasant feature of Dieudonn\'e complexes is that they provide a very clean dictionary between $p$-power torsion and torsion-free objects.  Namely, suppose $X^*$ is a saturated Dieudonn\'e complex, and let $\W_r X^* = X^*/(\im V^r + \im dV^r)$.  These objects form a pro-complex, where each $\W_r X^*$ is killed by $p^r$.  The Frobenius and Verschiebung operators on $X^*$ pass to the tower, giving it the structure of a \emph{strict Dieudonn\'e tower} (\cite[Definition 2.6.1]{BLM}).  One then shows that the category $\TD$ of strict Dieudonn\'e towers is equivalent to $\DC_{\str}$, where the forward equivalence is given by taking the limit and the inverse is $(\W_r (-))_r$.

Bhatt--Lurie--Mathew also introduces \emph{Dieudonn\'e algebras}, which are commutative algebra objects in $\DC$ satisfying a few extra hypotheses.  Dieudonn\'e algebras form a category $\DA$, and the subcategory of Dieudonn\'e algebras that are strict as Dieudonn\'e complexes is called $\DA_{\str}$.

With this setup, it is quite simple to define the saturated de Rham--Witt complex:
\begin{mydef} \label{BLM_definition}
(\cite[Definition 4.1.1]{BLM})  
If $R$ is an $\F_p$-algebra, a \emph{saturated de Rham--Witt complex associated to $R$} is a strict Dieudonn\'e algebra $A^*$ equipped with a ring map
$$
f \colon R \to \W_1 A^0 = A^0/VA^0
$$
that is initial among such collections of data; that is, such that every ring map $R \to \W_1 B^0$ for a strict Dieudonn\'e algebra $B^*$ factors as $\W_1(g) \circ f$ for a unique morphism $g \colon A^* \to B^*$ in $\DA_{\str}$.
\end{mydef}

Assuming existence, this is clearly unique up to unique isomorphism, and the functor
$$
\W\Omega^*_{(-)} \colon \F_p\lalg \to \DA_{\str}
$$
is left-adjoint to the functor
$$
\W_1 (-)^0 = (-)^0/V(-)^0 \colon \DA_{\str} \to \F_p\lalg.
$$
Bhatt--Lurie--Mathew provides two explicit constructions of $\W\Omega^*_R$:  one completely general construction built from Witt vectors, and a second construction under the hypothesis that $R$ admits a $p$-torsionfree lift with Frobenius.  Both constructions begin with the following observation.

\begin{proposition} \label{prop:BLM_dR_as_a_DC}
(\cite[Proposition 3.2.1]{BLM})
Suppose $A$ is a $p$-torsionfree ring equipped with a homomorphism $\phi \colon A \to A$ lifting the absolute Frobenius endomorphism of $A/pA$.  Then there is a unique graded ring homomorphism $F \colon \Omega^*_A \to \Omega^*_A$ which extends the Frobenius $\phi \colon A \to A$ in degree 0 and satisfies the identity
$$
F(dx) = x^{p-1} dx + d \left(\frac{\phi(x) - x^p}{p} \right)
$$
for all $x \in A$.  Moreover, this $F$ gives $\Omega^*_A$ the structure of a Dieudonn\'e algebra.
\end{proposition}

\begin{remark}
One can easily verify that the endomorphism $F$ of Proposition \ref{prop:BLM_dR_as_a_DC} satisfies the identity
$$
p^n F(x \cdot dx_1 \wedge \cdots \wedge dx_n) = \phi(x) \cdot d\phi(x_1) \wedge \cdots \wedge d\phi(x_n);
$$
that is, $p^n F$ agrees with the map $\phi^*$ which induced on $\Omega^n_A$ by the functoriality of $\phi$.  For this reason, we call $F$ a \emph{divided Frobenius}.  Note however that we cannot simply define $F$ as $\frac{\phi^*}{p^n}$ in degree $n$, as $\Omega^*_A$ may have $p$-torsion even when $A$ does not.
\end{remark}

The constructions are as follows:

\begin{construction} \label{con:BLM_construction_W}
(\cite[Proposition 4.1.4]{BLM}) Let $R$ be an $\F_p$-algebra, $R_{\red}$ its reduction, and $W(R_{\red})$ the ring of $p$-typical Witt vectors.  Then $W(R_{\red})$ is $p$-torsionfree and comes equipped with a Frobenius endomorphism $F$, so Proposition \ref{prop:BLM_dR_as_a_DC} gives the (naive) de Rham complex $\Omega^*_{W(R_{\red})} = \Omega^*_{W(R_{\red})/\Z}$ the structure of a Dieudonn\'e algebra.
Then the saturated de Rham--Witt complex functor is given by $R \mapsto \W\Sat(\Omega^*_{W(R_{\red})})$.
\end{construction}

\begin{construction} \label{con:BLM_lifted_construction}
(\cite[Variant 3.3.1 and Corollary 4.2.3]{BLM}) Suppose $A$ is a $p$-torsionfree ring equipped with a lift $\phi$ of the absolute Frobenius endomorphism of $R := A/pA$.  Then the $p$-adically completed de Rham complex $\widehat{\Omega}^*_A$ has the structure of a Dieudonn\'e algebra, and its strictification $\W\Sat(\widehat{\Omega}^*_A)$ is a saturated de Rham--Witt complex associated to $R$.
\end{construction}

\begin{remark} \label{rmk:BLM_nilpotent_invariance}
Implicit in Construction \ref{con:BLM_construction_W} is the fact that for any $\F_p$-algebra $R$, the natural morphism $\W\Omega^*_R \to \W\Omega^*_{R_{\red}}$ is an isomorphism.  This follows from the definition of $\W\Omega^*_R$ in light of \cite[Lemma 3.6.1]{BLM} and does not require the construction.
\end{remark}

Bhatt--Lurie--Mathew proves a number of results about the saturated de Rham--Witt complex, of which the following few are the most relevant to us.

\begin{theorem} \label{thm:BLM_comparison}
(\cite[Corollary 4.4.11, Theorem 4.4.12]{BLM})
For any $\F_p$-algebra $R$, we have canonical maps\footnote{These maps are called $\gamma_r$ and $\gamma$ in the original; we have changed the notation to avoid any potential confusion later with divided power structures which are denoted $\gamma$.}
\begin{align*}
\zeta_r \colon W_r \Omega^*_R & \overset{\sim}{\to} \W_r \Omega^*_R \text{ for all } r, \text{ and} \\
\zeta \colon W\Omega^*_R & \overset{\sim}{\to} \W\Omega^*_R,
\end{align*}
where $W\Omega^*_R = \limarrow_r W_r\Omega^*_R$ is the classical de Rham--Witt complex of \cite{illusie}.
If $R$ is a regular noetherian $\F_p$-algebra, then the maps $\zeta_r$ and $\zeta$ are isomorphisms.
\end{theorem}

\begin{theorem} \label{BLM_sheaf}
(\cite[Theorem 5.3.7, Remark 5.2.3]{BLM}) Let $X$ be an arbitrary $\F_p$-scheme.  The functors $\Spec R \mapsto \W\Omega^*_R$ and $\Spec R \mapsto \W_r \Omega^*_R$ define sheaves for the \'etale topology on $X$.  Moreover, the latter is a quasicoherent sheaf of $W_r \mathcal O_X$-modules.
\end{theorem}

We denote these sheaves $\W\Omega^*_X$ and $\W_r\Omega^*_X$, and call them the \emph{de Rham--Witt complex} and the \emph{de Rham--Witt pro-complex} of $X$ respectively.
Bhatt--Lurie--Mathew then shows that under reasonable hypotheses, the saturated de Rham--Witt complex of $X$ agrees with the classical de Rham--Witt complex and computes crystalline cohomology:

\begin{theorem} \label{thm:BLM_classical_comparison}
(\cite[Corollary 4.4.11, Theorem 4.4.12]{BLM})
Let $R$ be an $\F_p$-algebra.  Then there is a canonical morphism of differential graded algebras from the classical de Rham--Witt complex $W\Omega^*_R$
to the saturated de Rham--Witt complex $\W\Omega^*_R$.
If $R$ is regular Noetherian, then this map is an isomorphism.
\end{theorem}

\begin{theorem} \label{BLM_cohomology_comparison}
(\cite[Theorem 10.1.1 and surrounding remarks]{BLM})
Let $k$ be a perfect field of characteristic $p$, and $X$ a smooth $k$-scheme.  Then there is a canonical isomorphism of cohomology rings
$$
H^*_{\cris}(X/W(k)) \simeq \mathbb H^*(X, \W\Omega^*_X).
$$
\end{theorem}

Of course Theorem \ref{BLM_cohomology_comparison} follows from Theorems \ref{thm:BLM_classical_comparison} and \ref{thm:illusie_thm}, but \cite[\textsection 10]{BLM} proves it directly as well.  Ogus (\cite[Corollary 5.4]{ogus}) has given an alternative proof.


\subsection{Summary of main results}
\label{sec:intro_summary_of_results}

Our goal in this paper is to generalize the approach of Bhatt--Lurie--Mathew in order to define and construct saturated de Rham--Witt complexes with coefficients in a unit-root $F$-crystal $(\mathcal E, \phi_{\mathcal E})$.
As in \cite{BLM}, we work affine-locally; thus, suppose we are given a perfect field $k$ of characteristic $p$, a $k$-algebra $R$, and a unit-root $F$-crystal $\mathcal E$ on $\Cris(\Spec R/W(k))$.

To begin, we must say what category our saturated de Rham--Witt complexes with coefficients live in.  This is as follows (see Definition \ref{def:dRW_module} for details):

\begin{mydef} \label{def:dRW_module_intro}
A \emph{de Rham--Witt module over $(R, \mathcal E)$} is a collection of the following data:  a left $\W\Omega^*_R$-module $M^*$ in $\DC_{\str}$, equipped with $W_r(R)$-linear maps $\iota_r \colon \mathcal E(W_r(R)) \to \W_r M^0$ for each $r$, compatible with quotient and Frobenius maps, and compatible with connections in a suitable sense.
\end{mydef}

Within this category, our de Rham--Witt complexes are defined by a concise universal property:

\begin{mydef} \label{def:saturated_dRW_intro}
A \emph{saturated de Rham--Witt complex associated to $\mathcal E$ over $R$} is an initial object in the category of de Rham--Witt modules over $(R, \mathcal E)$.  Such an object is unique up to unique isomorphism if it exists; we will denote it by $\W\Omega^*_{R, \mathcal E}$.
\end{mydef}

Since we define our de Rham--Witt complexes by universal property, we had better prove that they exist:

\begin{theorem} \label{thm:existence}
(See Theorem \ref{thm:general_construction}.)
Suppose $R$ is a $k$-algebra, and $\mathcal E$ is a unit-root F-crystal on $\Cris(\Spec R/W)$.  Then there exists a saturated de Rham--Witt complex $\W\Omega^*_{R, \mathcal E}$ associated to $(R, \mathcal E)$.
\end{theorem}

The proof is much more complicated than in the case of trivial coefficients, since we have no analogue of the Witt vector construction \ref{con:BLM_construction_W}.  Roughly speaking, we will prove an analogue of the lifted construction \ref{con:BLM_lifted_construction}, and then reduce to the case where it applies.

We also prove an analogue of Theorem \ref{BLM_sheaf}:

\begin{proposition} \label{prop:sheaf}
(See Proposition \ref{prop:dRW_is_qcoh}.)
Let $X$ be a $k$-scheme, and $\mathcal E$ a unit-root $F$-crystal on $\Cris(X/W(k))$.  The functors
\begin{align*}
\Spec R & \mapsto \W\Omega^*_{R, \mathcal E} \text{ and} \\
\Spec R & \mapsto \W_r \Omega^*_{R, \mathcal E} \text{ for } r > 0
\end{align*}
define sheaves for the \'etale topology on $X$.  The latter is a quasicoherent sheaf of $W_r \mathcal O_X$-modules.
\end{proposition}

Let $\W\Omega^*_{X, \mathcal E}$ and $\W_r \Omega^*_{X, \mathcal E}$ denote the sheaves produced by the proposition.
Our next main result relates these sheaves to the cohomology of $\mathcal E$ provided that $X/k$ is smooth:

\begin{theorem} \label{thm:cohomology}
(See Corollary \ref{cor:cohomology_comparison}.)
If $X$ is a smooth $k$-scheme and $(\mathcal E, \phi_{\mathcal E})$ is a unit-root F-crystal on $\Cris(X/W(k))$, then there are canonical isomorphisms
$$
\mathbb H^i(X_{\zar}, \W_r \Omega^*_{X, \mathcal E}) \simeq H^i((X/W_r)_{\cris}, \mathcal E)
$$
for each $r > 0$, and
$$
\mathbb H^i(X_{\zar}, \W\Omega^*_{X, \mathcal E}) \simeq H^i((X/W)_{\cris}, \mathcal E)
$$
satisfying various compatibilities as described in Proposition \ref{prop:compatibilities_of_cohomology_comparison}.
\end{theorem}

Finally, again assuming smoothness, we compare our de Rham--Witt (pro-)complexes to the classical ones constructed in \cite{etesse}.  These have the form
\begin{align*}
W_r\Omega^*_{X, \mathcal E} & = \dR(\mathcal E_{W_r(X), \gamma}) \otimes_{\Omega^*_{W_r(\mathcal O_X), \gamma}} W_r \Omega^*_X \text{ and} \\
W \Omega^*_{X, \mathcal E} & = \limarrow_r W_r\Omega^*_{X, \mathcal E}.
\end{align*}
where $\dR(\mathcal E_{W_r(X), \gamma})$ is the PD-de Rham complex associated to $\mathcal E$ on $W_r(X)$ (\cite[\textsection II, Proposition 1.1.5]{etesse}), and $(W_r \Omega^*_X)_r$ is the classical de Rham--Witt pro-complex of $X$.  (As before, we use a calligraphic $\W$ for the saturated de Rham--Witt complex and a roman $W$ for the classical one.)

\begin{theorem} \label{thm:etesse_comparison_intro}
(See Theorem \ref{thm:etesse_comparison}.)
Suppose $X$ is a $k$-scheme and $(\mathcal E, \phi_{\mathcal E})$ is a unit-root $F$-crystal on $\Cris(X/W(k))$.  Then we have compatible maps
\begin{align*}
W_r \Omega^*_{X, \mathcal E} & \to \W_r \Omega^*_{X, \mathcal E} \text{ for each } r, \text{ and} \\
W\Omega^*_{X, \mathcal E} & \to \W\Omega^*_{X, \mathcal E},
\end{align*}
which are isomorphisms if $X/k$ is smooth.
\end{theorem}


\subsection{New ingredients}
\label{sec:new_ingredients}

The central idea introduced in this work is the category $\dRWM_{R, \mathcal E}$ of de Rham--Witt modules (Definition \ref{def:dRW_module}).  Whereas the classical treatment of de Rham--Witt complexes with coefficients, due to \'Etesse (\cite{etesse}), is more constructive in nature, de Rham--Witt modules provide a clean setting in which to define our saturated de Rham--Witt complexes $\W\Omega^*_{R, \mathcal E}$ by universal property:  namely, $\W\Omega^*_{R, \mathcal E}$ is defined to be an initial object of $\dRWM_{R, \mathcal E}$.  This is in the spirit of Illusie's original definition (\cite[I, num\'ero 1]{illusie})
as well as that of Bhatt--Lurie--Mathew (\cite[\textsection 4]{BLM}):  we build a category of objects that carry all the structure we want, and we ask for an initial object in this category.  The actual construction is done only later, first in the special case where $R$ admits a $p$-torsionfree lift with Frobenius, and later in general by reducing to this case.

A significant---and necessary---difference between our universal property and its predecessors is that while de Rham--Witt complexes with trivial coefficients carry an algebra structure (corresponding to the algebra structure of the trivial $F$-crystal $\mathcal O_{X/W}$), ours instead carries the structure of a module over $\W\Omega^*_R$ (corresponding to the structure of the coefficient crystal $\mathcal E$ as an $\mathcal O_{X/W}$-module).  Note that this has some content even when $\mathcal E$ is the trivial crystal:  $\W\Omega^*_R$ satisfies a universal property not only as an algebra but also as a module over itself.

Accordingly, we must make systematic use of modules in the categories $\DC$ and $\DC_{\str}$.  We develop this theory in chapter \ref{ch:DC_algebra}.  Although the symmetric monoidal structures of these categories appear already in \cite{BLM}, we study them in greater detail, including concrete descriptions of the resulting algebra and module objects (\textsection \ref{sec:modules_in_DC}), an account of tensor products over a base algebra (\textsection \ref{sec:tensor_over_a_base}), and a discussion of the structure carried by ${\W_r} M^*$ when $M^*$ is a module in $\DC_{\str}$ (\textsection \ref{sec:filtrations_on_modules}).


\subsection{Outline}
\label{sec:intro_outline}

We begin in chapter \ref{ch:preliminaries} with some preliminary material on de Rham and PD-de Rham complexes, which will be used throughout.  In chapter \ref{ch:DC_algebra}, we develop the theory of modules in the categories $\DC$ and $\DC_{\str}$, as discussed above, and also study how filtered colimits of strict Dieudonn\'e complexes behave under the equivalence of categories $\DC_{\str} \simeq \TD$ (\textsection \ref{sec:DCstr_and_colimits}).

In chapter \ref{ch:dRWM}, we define the category $\dRWM_{R, \mathcal E}$ of \emph{de Rham--Witt modules} associated to a unit-root $F$-crystal $\mathcal E = (\mathcal E, \phi_{\mathcal E})$ on $\Spec R$, which houses the universal property characterizing our saturated de Rham--Witt complex $\W\Omega^*_{R, \mathcal E}$.  At this stage, we will not be able to construct $\W\Omega^*_{R, \mathcal E}$ except in the case of the trivial crystal $\mathcal E = \mathcal O_{X/W}$ (\textsection \ref{sec:trivial_F-crystal}).  However, we will prove several useful formal properties about the behavior of $\dRWM_{R, \mathcal E}$ and $\W\Omega^*_{R, \mathcal E}$:  functoriality (\textsection \ref{sec:functoriality_of_dRWM}), insensitivity to nilpotent thickenings (\textsection \ref{sec:nilpotent_thickenings}), an \'etale sheaf property (\textsection \ref{sec:etale_localization}), and compatibility with colimits (\textsection \ref{sec:compatibility_with_colimits}).  Although we work on an affine $k$-scheme $\Spec R$ throughout, the \'etale sheaf property allows us to define a saturated de Rham--Witt complex $\W\Omega^*_{X, \mathcal E}$ when $X$ is not necessarily affine.

In chapter \ref{ch:dRWLM}, we introduce our main technique for constructing $\W\Omega^*_{R, \mathcal E}$.  Namely, assuming that $R$ admits a $p$-torsionfree lift $A$ with Frobenius, we introduce the category $\dRWLM_{A, \mathcal E}$ of \emph{de Rham--Witt lift modules}.  In this case, we show that the categories $\dRWM_{R, \mathcal E}$ and $\dRWLM_{A, \mathcal E}$ are equivalent.  By studying the latter category, we are able to construct the saturated de Rham--Witt complex $\W\Omega^*_{R, \mathcal E}$ in the Frobenius-lifted situation.

We prove our main theorems in chapter \ref{ch:main_results}.  Namely, in \textsection \ref{sec:general_construction}, we show by reducing to the Frobenius-lifted situation that $\W\Omega^*_{R, \mathcal E}$ always exists.  If $X/k$ is smooth, we show in \textsection \ref{sec:cohomology_comparison} that $\W\Omega^*_{X, \mathcal E}$ computes the cohomology of $\mathcal E$,
and in \textsection \ref{sec:comparison_to_etesse} we compare it to the classical de Rham--Witt complex $W\Omega^*_{X, \mathcal E}$.


\subsection{Future directions}
\label{sec:future_directions}

A significant technical nuisance in this work is our lack of a rich theory of Dieudonn\'e complexes valued in sheaves, including the tensor algebra discussed in chapter \ref{ch:DC_algebra}.  Given such a theory, we could define our saturated de Rham--Witt complexes by a universal property within a category of \emph{sheafy de Rham--Witt modules}, rather than working affine-locally and bootstrapping up to the non-affine case.

This would not only be a conceptual simplification; it would also allow for a simple alternate proof of Theorem \ref{thm:etesse_comparison_intro}.
Namely, suppose $X/k$ is smooth and $(\mathcal E, \phi_{\mathcal E})$ is a unit-root $F$-crystal on $X$.
It follows from a theorem of Katz that we have $(\mathcal E, \phi_{\mathcal E}) \simeq (\mathcal L \otimes \mathcal O_{X/W(k)}, \id \otimes \phi_{\mathcal O})$ for some \'etale $\Z_p$-local system $\mathcal L = (\mathcal L_r)_r$ on $X$.
Then Theorem \ref{thm:etesse_comparison_intro} would follow from an isomorphism of pro-complexes
\begin{equation} \label{eqn:katz_calculation}
(\W_r \Omega^*_{X, \mathcal E})_r \simeq (\W_r\Omega^*_X \otimes_{\Z/p^r\Z} \mathcal L_r)_r.
\end{equation}
Unfortunately, we cannot prove \ref{eqn:katz_calculation} by trivializing $\mathcal E$:  our definition of $\W_r \Omega^*_{X, \mathcal E}$ relies on the data of the entire tower $(\mathcal L_r)_r$, and there is generally no Zariski or \'etale cover of $X$ that trivializes the entire tower simultaneously.
Thus, we would like to embrace the idea of working with sheaves, and view both sides of the isomorphism not as sheaves valued in strict Dieudonn\'e towers but as \emph{strict Dieudonn\'e towers valued in the category of abelian sheaves}.

The notion of a strict Dieudonn\'e tower (or a Dieudonn\'e complex) valued in sheaves does not appear in \cite{BLM}, as all of their constructions are affine-local.  In joint work in progress with Joe Stahl (\cite{joint}), we develop a theory of Dieudonn\'e complexes in a general complete and cocomplete abelian category $\mathcal A$.  This theory includes categories $\DC_{\mathcal A}$, $\DC_{\sat, \mathcal A}$, $\DC_{\str, \mathcal A}$, and $\TD_{\mathcal A}$ of (saturated, strict) Dieudonn\'e complexes and strict Dieudonn\'e towers valued in $\mathcal A$, which recover the usual categories when $\mathcal A = \Ab$.

Generalizing the theory to this context introduces several new difficulties, both internal to the theory of Dieudonn\'e complexes in $\mathcal A$ and (when $\mathcal A$ is the category of abelian sheaves on a site $X$) involving comparisons between sheafy and presheafy constructions.
For example, we generally do not have an equivalence of categories
$$\DC_{\str, X} \simeq \TD_X$$
between strict Dieudonn\'e complexes and Dieudonn\'e towers of sheaves on $X$ (cf. \cite[Proposition 2.9.1, Corollary 2.9.4]{BLM}):  although we have functors in both directions, the composition
$$
\TD_X \to \DC_{\str, X} \to \TD_X
$$
is generally not isomorphic to the identity.
Similarly, given $M^*$ in $\DC_{\sat, X}$, we do not know whether the ``strictification'' $\W(M^*)$ is strict (cf. \cite[Proposition 2.7.5, Corollary 2.7.6]{BLM}).
Both of these issues can be resolved for our objects of interest by appealing to the quasicoherence property of Proposition \ref{prop:sheaf}; however, it would be interesting to know the most general statements.

A further challenge is to construct a symmetric monoidal structure on either $\DC_{\str, X}$ or $\TD_X$ generalizing the strict tensor product
$$
\otimes^{\str} \colon \DC_{\str} \times \DC_{\str} \to \DC_{\str}
$$
of \cite[Remark 7.6.4]{BLM}.  This is a necessary ingredient for building a category of sheafy de Rham--Witt modules; such a module should in particular be a module over $\W\Omega^*_X$ in $\DC_{\str, X}$.
It is not difficult to define a tensor product on $\DC_{\mathcal A}$ given one on $\mathcal A$; however, Bhatt--Lurie--Mathew's construction of $\otimes^{\str}$ seems to rely crucially on the fact that $\Z$ has cohomological dimension $1$, and it does not readily generalize to the case of sheaves.

Beyond these technical issues, we would like to give an account, in the spirit of \cite{BLM}, of de Rham--Witt complexes with coefficients in an $F$-crystal $(\mathcal E, \phi_{\mathcal E})$ that is not unit-root.  These objects cannot be strict Dieudonn\'e complexes, because the map
$$
\alpha_F \colon \W\Omega^*_{R, \mathcal E} \to \eta_p \W\Omega^*_{R, \mathcal E}
$$
given by $p^i F$ in degree $i$ will not be an isomorphism.  Instead, we would need a modification of the category $\DC_{\str}$ which carries some extra data allowing the Frobenius endomorphism to have smaller-than-expected image in a controlled way, in the spirit of Fontaine--Jannsen's theory of $\varphi$-gauges (\cite{F-J})---but which nonetheless has a universal construction analogous to strictification.  In this light, the current work can be viewed as a proof of concept, which addresses all cases that are accessible with the current tools.


\subsection{Notation}
\label{sec:notation}

\begin{para}
Unless otherwise specified, all rings are commutative with $1$.  We fix a prime number $p$ throughout.
We will always let $k$ denote a perfect field of characteristic $p$, which we will usually treat as fixed but will occasionally need to vary.
We let $W = W(k)$ be its ring of Witt vectors, and $W_r = W_r(k) = W/p^rW$ for each $r > 0$.  We will always\footnote{Occasionally we will also use the letter $R$ for certain quotient maps in towers, in accordance with \cite[Definition 2.6.1]{BLM}.} let $R$ denote a ring of characteristic $p$, which will usually be a $k$-algebra and will sometimes have additional restrictions.  We will usually let $A$ denote a $p$-torsionfree lift of $R$, although the exact hypotheses we impose will vary.
\end{para}

\begin{para} \label{para:notation_Frobenii}
Several kinds of Frobenius operators will make an appearance, and we will do our best to notate them all unambiguously with the limited set of suggestive names at our disposal.  The absolute Frobenius endomorphism of an $\F_p$-algebra $R$ will be denoted $F_R$ or $\Frob_R$.  The Frobenius endomorphism on $W(R)$ will be written as $F$, or as $\sigma$ if $R = k$ is a perfect field of characteristic $p$.  Other Frobenius lifts will typically be denoted as $\phi$; cf. Situations \ref{sit:lifted_situation} and \ref{sit:embedded_situation}.  If $A$ is a ring with Frobenius lift $\phi$, then the induced endomorphism of $\Omega^*_A$ will be called $\phi$ or $\phi^*$; in this situation, we will reserve the name $F$ for a divided Frobenius.  (This is consistent with the convention of denoting Dieudonn\'e complexes as $(M^*, d, F)$; cf. \cite[Remark 2.1.4]{BLM}.)

We will write $F$-crystals as $(\mathcal E, \phi_{\mathcal E})$, where $\mathcal E$ is a crystal on $X$ and $\phi_{\mathcal E} \colon (F_X)^*_{\cris} \mathcal E \to \mathcal E$ is $\mathcal O_{X/S}$-linear; we will often suppress the Frobenius $\phi_{\mathcal E}$ from the notation.  Finally, if $T = (U, T, \delta)$ and $T' = (U', T', \delta')$ are objects of $\Cris(X/S)$ and we have a PD-morphism $h \colon T' \to T$ over $F_X \colon X \to X$, then the Frobenius-semilinear map of sections
\begin{equation} \label{eqn:F_on_sections}
\mathcal E(T) \overset{h^*}{\to} (h^* \mathcal E)(T') \simeq ((F_X)^*_{\cris} \mathcal E)(T') \overset{\phi_{\mathcal E}}{\to} \mathcal E(T')
\end{equation}
will be written as $F$.
\end{para}

\begin{para} \label{para:notation_PD_structures}
We will also encounter a number of PD-structures.  We will always reserve the notation $\brackets$ for one particular type of PD-structure:  namely, the unique PD-structure on the ideal $(p)$ of a $\Z_{(p)}$-algebra $A$ (usually a $\Z/p^r\Z$-algebra for some $r > 0$) that is compatible with the PD-structure on $(p) \subset \Z_{(p)}$.  Explicitly:
$$
(px)^{[n]} = x^n p^{[n]} = \frac{p^n}{n!} x^n,
$$
where $p^n/n! \in p\Z_{(p)}$.  Other PD-structures---such as the standard PD-structure $\gamma_n(Vx) = p^{n-1} V(x^n)/n!$ on $VW(R) \subset W(R)$ or $VW_{r-1}(R) \subset W_r(R)$---will be called $\gamma$ or $\delta$.
\end{para}

\begin{para}
We will borrow a number of notational conventions from \cite{BLM}.  For example, if $M^*$ is a Dieudonn\'e complex, we will let $\rho = \rho_M \colon M^* \to \W\Sat(M^*)$ denote its strictification map; that is, the initial morphism from $M^*$ to a strict Dieudonn\'e complex.  Similarly, $\alpha_F$ will denote the map of complexes $M^* \to \eta_p M^*$ (or $M^* \to M^*$ if $M^*$ is concentrated in nonnegative degrees) which is defined by $p^i F$ in degree $i$.
\end{para}

For future reference, we collect here some setup and related observations which we will frequently invoke throughout the paper.

\begin{situation} \label{sit:lifted_situation}
\hspace{0in}
\begin{enumerate}
\item[(a)] (\emph{The lifted situation})
Let $R$ be a $k$-algebra and $A$ a $p$-torsion-free ring with $A/pA = R$.  We set $A_r = A/p^r A$ for each $r > 0$, and $\widehat{A} = \limarrow_r A_r$.
\item[(b)] (\emph{The Frobenius-lifted situation})
In addition to the above, let $\phi \colon A \to A$ be a lift of the Frobenius endomorphism of $R$; that is, a ring homomorphism such that $\phi(x) \equiv x^p \pmod p$ for all $x \in A$.  Note that $\phi$ induces endomorphisms of each $A_r$.
\end{enumerate}
\end{situation}

\begin{situation} \label{sit:embedded_situation}
\hspace{0in}
\begin{enumerate}
\item[(a)] (\emph{The embedded situation})
Let $R$ be a $k$-algebra and $A$ a smooth algebra over $W = W(k)$ equipped with a quotient map $A \twoheadrightarrow R$ of $W$-algebras, with kernel $I \subset A$ (necessarily containing $p$).  For each $r > 0$, we set $A_r = A/p^r A$ and $I_r = I/p^r A$, and let $(B_r, \overline I_r, \gamma)$ denote the PD-envelope of $(A_r, I_r)$ over $(W_r(k), (p), \brackets)$.  Let $(B, \overline I, \gamma)$ be the inverse limit of the PD-algebras $(B_r, \overline I_r, \gamma)$.  Note that we have $B_r = B_{r+1}/p^r B_{r+1}$ and $\overline I_r = \overline I_{r+1}/p^r B_{r+1}$ by \cite[Remark 3.20.8]{B-O}; it follows by \cite[tag \texttt{09B8}]{stacks} that $B_r = B/p^r B$ and thus also $\overline I_r = \overline I/p^r B$.

\item[(b)] (\emph{The Frobenius-embedded situation})
In addition to the above, let $\phi = \phi_A \colon A \to A$ be a lift of the Frobenius endomorphism of $A_1$.  Note that $\phi$ is then also compatible with the Frobenius endomorphism of $R$, and thus preserves the ideal $I$.  It follows that $\phi$ induces PD-endomorphisms $\phi_{B_r} \colon (B_r, \overline I_r, \gamma) \to (B_r, \overline I_r, \gamma)$ for each $r$ by the functoriality of the PD-envelope, and thus also $\phi_B \colon (B, \overline I, \gamma) \to (B, \overline I, \gamma)$ by passing to the limit.
\end{enumerate}
\end{situation}

\begin{remark} \label{rmk:Frob_embedded_situation_Frob_lift}
In the Frobenius-embedded situation, $B$ may or may not have $p$-torsion,
but the endomorphism $\phi_B$ is a lift of the absolute Frobenius of $B_1 = B/pB$.  To prove this, note that the set of $b \in B_1$ such that $\phi_{B_1}(b) = b^p$ is a subring of $B_1$ containing $A_1$, so it suffices to check that it also contains the algebra generators $\gamma_n(a)$ for all $a \in I_1$ and $n > 0$.  But recalling that $a^p = p! \gamma_p(a) = 0$ and similarly $\gamma_n(a)^p = 0$, we can calculate:
$$
\phi_{B_1}(\gamma_n(a)) = \gamma_n(\phi_{A_1}(a)) = \gamma_n(a^p) = \gamma_n(0) = 0 = (\gamma_n(a))^p.
$$
\end{remark}


\subsection{Acknowledgements}
This work is a much-abridged version of the author's Ph.D. thesis, \cite{thesis}, which was supported in significant part by NSF RTG grant DMS-1646385.

First and foremost, I thank my Ph.D. advisor, Martin Olsson, for his patient guidance throughout my time at Berkeley, without which this work would not have been possible.
I thank Arthur Ogus for several helpful conversations, and especially for a series of suggestions which inspired the strategy of chapters \ref{ch:dRWM} through \ref{ch:main_results}.
I thank Joe Stahl for numerous helpful conversations, particularly those related to our joint project \cite{joint}.
I am also grateful to Akhil Mathew, Minseon Shin, and Chris Dodd for helpful comments during various stages of this work, as well as to all the authors of \cite{BLM}, to which this work is heavily indebted.


\section{Preliminaries}
\label{ch:preliminaries}


\subsection{de Rham and PD-de Rham complexes of rings}
\label{sec:dR_and_PD-dR_complexes_of_rings}

\begin{para}
In this section, we recall some basic facts about de Rham and PD-de Rham complexes, which will serve as the most basic building blocks of all of our later constructions.
We refer the reader to \cite[\textsection 3]{B-O} and \cite[\textsection I]{berthelot} for background on PD-structures.
\end{para}

\begin{mydef} \label{def:PD_dR_cx}
Let $B \to A$ be a morphism of rings, and suppose $(A, I, \gamma)$ is a PD-algebra.
The \emph{PD-de Rham complex of $A/B$} is the initial object among $B$-linear commutative differential graded algebras $C^*$ which are equipped with a morphism $A \to C^0$ and satisfy the identity
$$
d\gamma_n(x) = \gamma_{n-1}(x) dx
$$
for all $n > 0$ and $x \in I$.
It can be constructed as the quotient of the usual de Rham complex $\Omega^*_{A/B} = {\bigwedge}^*_A \Omega^1_{A/B}$ by the dg-ideal generated by the elements $d\gamma_n(x) - \gamma_{n-1}(x) dx$.
\end{mydef}

\begin{remark} \label{rmk:PD_dR_dg_vs_graded_ideal}
Note that for $n > 1$, we have the following calculation in $\Omega^*_{A/B}$:
\begin{align*}
d \left( d\gamma_n(x) - \gamma_{n-1}(x) dx \right) & = 0 - d \gamma_{n-1}(x) \wedge dx \\
& = \gamma_{n-2}(x) dx \wedge dx - d \gamma_{n-1}(x) \wedge dx \\
& = dx \wedge \left( d \gamma_{n-1}(x) - \gamma_{n-2}(x) dx \right).
\end{align*}
This (combined with the fact that $d\gamma_1(x) - \gamma_0(x) dx$ vanishes already in $\Omega^1_{A/B}$) implies that the graded ideal generated by elements of the form $d\gamma_n(x) - \gamma_{n-1}(x) dx$ is closed under the differential, and thus it agrees with the dg-ideal generated by the same elements.
\end{remark}

\begin{remark} \label{rmk:base_ring_doesn't_matter}
Suppose $k$ is a perfect $\F_p$-algebra, and $A$ is a $W_r(k)$-algebra for some $r > 0$.  Then a standard argument shows that the de Rham complexes of $A$ over the base rings $\Z, \Z/p^r\Z, \Z_p, W(k),$ and $W_r(k)$ all agree.
Moreover, if $A$ comes equipped with a PD-structure $(I, \gamma)$, then the same is true for PD-de Rham complexes.
Accordingly, we will sometimes omit the base ring from the notation in this situation.
\end{remark}

Next we note how de Rham and PD-de Rham complexes behave under surjections of rings:
\begin{lemma} \label{lem:quotient_of_PD_dR_algebras}
\hspace{0in}
\begin{enumerate}
\item Let $A \to B$ be a surjection of rings.  Then the natural map $\Omega^*_A \to \Omega^*_B$ is the quotient by the dg-ideal $K$ generated by $\ker(A \to B)$ in degree 0.
\item Let $(A, I, \gamma) \to (B, J, \delta)$ be a surjection of PD-algebras with $J = \im(I)$.  Then the natural map $\Omega^*_{A,\gamma} \to \Omega^*_{B, \delta}$ is the quotient by the dg-ideal $K$ generated by $\ker(A \to B)$ in degree 0.
\end{enumerate}
\begin{proof}
We compare universal properties.  For any commutative differential graded algebra $C^*$, we have natural bijections:
\begin{align*}
\Hom_{cdga} (\Omega^*_A/K, C^*) & \simeq \{ f \in \Hom_{cdga}(\Omega^*_A, C^*): f(K) = 0 \} \\
& \simeq \{g \in \Hom(A, C^0):  g(\ker(A \to B)) = 0\} \\
& \simeq \Hom(B, C^0) \\
& \simeq \Hom_{cdga} (\Omega^*_B, C^*)
\end{align*}
and
\begin{align*}
\Hom_{cdga} (\Omega^*_{A, \gamma}/K, C^*) & \simeq \{ f \in \Hom_{cdga}(\Omega^*_{A,\gamma}, C^*): f(K) = 0 \} \\
& \simeq \{g \in \Hom(A, C^0):  dg(\gamma_n(x)) = g(\gamma_{n-1}(x)) dg(x), g(\ker(A \to B)) = 0\} \\
& \simeq \{h \in \Hom(B, C^0):  dh(\delta_n(y)) = h(\delta_{n-1}(y)) dh(y) \} \\
& \simeq \Hom_{cdga} (\Omega^*_{B, \delta}, C^*),
\end{align*}
where $x$ ranges over $I \subset A$ and $y$ ranges over $J \subset B$.  The result follows.
\end{proof}
\end{lemma}

In Definition \ref{def:PD_dR_cx}, it is sometimes useful to impose the divided power compatibility condition only for $p$-th divided powers, and to quotient by the ideal generated by the given relations instead of its closure under the differential.  This is sufficient in the case of $\Z_{(p)}$-algebras:

\begin{proposition} \label{prop:PD_dR_p_suffices}
Suppose $B \to A$ is a morphism of rings, and $(A, I, \gamma)$ is a PD-algebra.  Suppose $A$ is moreover a $\mathbb Z_{(p)}$-algebra.
Then the dg-ideal of $\Omega^*_{A/B}$ generated by the elements $d\gamma_n(x) - \gamma_{n-1}(x) dx$ with $x \in I$ and $n \geq 1$ is in fact generated as an ordinary (graded) ideal by the elements of this form with $n=p$.
\end{proposition}
\begin{proof}
First we note that the graded ideal generated by the elements $d\gamma_p(x) - \gamma_{p-1}(x) dx$ is already a dg-ideal; this follows from the calculation (cf. Remark \ref{rmk:PD_dR_dg_vs_graded_ideal}):
\begin{align*}
d(d\gamma_p(x) - \gamma_{p-1}(x) dx) & = 0 - d\gamma_{p-1}(x) \wedge dx \\
& = - d \left(\frac{x^{p-1}}{(p-1)!} \right) \wedge dx \\
& = - \frac{x^{p-2}}{(p-2)!} dx \wedge dx = 0,
\end{align*}
since $(p-1)!$ is invertible in $A$.
Next, it follows from \cite[Lemma 1.2]{L-Z} that the map
$$
A \overset{d}{\longrightarrow} \Omega^1_{A/B}/(d\gamma_p(x) - \gamma_{p-1}(x) dx)
$$
is a PD-derivation.  Since $\Omega^*_{A/B, \gamma}$ is the initial quotient of $\Omega^*_{A/B}$ with this property, it follows that the dg-ideals $\langle d\gamma_p(x) - \gamma_{p-1}(x) dx \rangle$ and $\langle d\gamma_n(x) - \gamma_{n-1}(x) dx \rangle$ agree.
\end{proof}

\begin{corollary} \label{cor:PD_dR_exact_p-torsion}
In the situation of Proposition \ref{prop:PD_dR_p_suffices}, the kernel of the quotient map $\Omega^*_{A/B} \to \Omega^*_{A/B, \gamma}$ is killed by $p$.
\end{corollary}
\begin{proof}
By the proposition, the kernel is generated by elements of the form $d\gamma_p(x) - \gamma_{p-1}(x) dx$, where $x \in I$.  But these elements are killed by $p!$, and indeed by $p$ since $(p-1)!$ is a unit.
\end{proof}
We finish this section by discussing how de Rham and PD-de Rham complexes of truncated Witt vectors behave under \'etale base change:

\begin{lemma} \label{lem:PD_dR_etale_base_change}
Let $R \to S$ be an \'etale map of $\F_p$-algebras, and endow the ideals $VW_{r-1}(R) \subset W_r(R)$ and $VW_{r-1}(S) \subset W_r(S)$ with their standard PD-structures, both denoted $\gamma$.  Then the graded $W_r(S)$-module maps
\begin{align*}
W_r(S) \otimes_{W_r(R)} \Omega^*_{W_r(R)} & \to \Omega^*_{W_r(S)} \text{ and} \\
W_r(S) \otimes_{W_r(R)} \Omega^*_{W_r(R), \gamma} & \to \Omega^*_{W_r(S), \gamma}
\end{align*}
induced by functoriality of the de Rham and PD-de Rham complexes are isomorphisms.
\end{lemma}
\begin{proof}
The statement for ordinary de Rham complexes holds because $W_r(R) \to W_r(S)$ is \'etale.
To prove the statement for PD-de Rham complexes, it suffices to work in degree $1$, since the kernels
\begin{align*}
K_R & := \ker(\Omega^*_{W_r(R)} \to \Omega^*_{W_r(R), \gamma}), \\
K_S & := \ker(\Omega^*_{W_r(S)} \to \Omega^*_{W_r(S), \gamma})
\end{align*}
are generated in degree $1$ by Remark \ref{rmk:PD_dR_dg_vs_graded_ideal}.
Thus, by the case of ordinary de Rham complexes, it suffices to show that the natural map
\begin{equation} \label{eqn:PD_dR_kernel_base_change}
W_r(R) \otimes_{W_r(S)} K_R \to K_S
\end{equation}
is surjective.
To this end, let $Vy$ be an arbitrary element of $VW_{r-1}(S)$, and let $K$ denote the image of $\eqref{eqn:PD_dR_kernel_base_change}$.  We must show that
$$
d\gamma_n(Vy) - \gamma_{n-1}(Vy) dVy \in K
$$
for all $n$.  By base-changing the short exact sequence
$$
0 \to VW_{r-1}(R) \hookrightarrow W_r(R) \to R \to 0
$$
along the \'etale map $W_r(R) \to W_r(S)$, we see that
$$
W_r(S) \otimes_{W_r(R)} VW_{r-1}(R) \to VW_{r-1}(S)
$$
is surjective, and so $Vy$ lies in the $W_r(S)$-span of elements $Vx$ with $x \in W_{r-1}(R)$.

Certainly the elements $d\gamma_n(Vx) - \gamma_{n-1}(Vx) dVx$ lie in $K$ (or even $K_R$) for all $n$.
We will finish the proof by showing that the set of elements $\alpha \in VW_{r-1}(S)$ such that $d\gamma_n(\alpha) - \gamma_{n-1}(\alpha) d\alpha$ lies in $K$ for all $n$ is a sub-ideal of $VW_{r-1}(S)$, and thus contains $Vy$.  Indeed, this set is closed under scalar multiplication by $s \in W_r(S)$ by the calculation
\begin{align*}
d\gamma_n(s\alpha) - \gamma_{n-1}(s\alpha) d(s\alpha) & = d(s^n \gamma_n(\alpha)) - s^{n-1} \gamma_{n-1}(\alpha) d(s\alpha) \\
& = s^n d\gamma_n(\alpha) + n s^{n-1} \gamma_n(\alpha) ds - s^{n-1} \gamma_{n-1}(\alpha) (s d\alpha + \alpha ds) \\
& = s^n \left( d\gamma_n(\alpha) - \gamma_{n-1}(\alpha) d\alpha \right) + s^{n-1} (n \gamma_n(\alpha) - \alpha \gamma_{n-1}(\alpha)) ds \\
& = s^n \left( d\gamma_n(\alpha) - \gamma_{n-1}(\alpha) d\alpha \right) \in K,
\end{align*}
and it is closed under addition by the calculation
\begin{align*}
& d\gamma_n(\alpha + \alpha') - \gamma_{n-1}(\alpha + \alpha') d(\alpha + \alpha') \\
& = \sum_{i+j=n} d(\gamma_i(\alpha) \gamma_j(\alpha')) - \left( \sum_{i+j=n-1} \gamma_i(\alpha) \gamma_j(\alpha') \right) (d\alpha + d\alpha') \\
& = \sum_{i+j=n} \left( \gamma_i(\alpha) d\gamma_j(\alpha') + \gamma_j(\alpha') d\gamma_i(\alpha) \right) - \sum_{\substack{i+j=n \\ i \neq 0}} \gamma_j(\alpha') \gamma_{i-1}(\alpha) d\alpha - \sum_{\substack{i+j=n \\ j \neq 0}} \gamma_i(\alpha) \gamma_{j-1}(\alpha') d\alpha' \\
& = \sum_{\substack{i+j=n \\ i \neq 0}} \gamma_j(\alpha') \left( d\gamma_i(\alpha) - \gamma_{i-1}(\alpha) d\alpha \right)
+ \sum_{\substack{i+j=n \\ j \neq 0}} \gamma_i(\alpha) \left( d\gamma_j(\alpha') - \gamma_{j-1}(\alpha') d\alpha' \right) \in K. \qedhere
\end{align*}
\end{proof}


\subsection{Completed de Rham and PD-de Rham complexes of rings}

\begin{para}
In this section, we will discuss the $p$-adically completed versions of the de Rham and PD-de Rham complexes of $p$-torsionfree rings.  It will be convenient to unify our discussions of the lifts $(A, (p))$ of the lifted situation \ref{sit:lifted_situation} and the completed PD-envelopes $(B, \overline I)$ of the embedded situation \ref{sit:embedded_situation} (under the extra hypothesis that $B$ is $p$-torsionfree).  Accordingly, we fix the following setup throughout this section.
\end{para}

\begin{situation} \label{sit:PD_lift}
\hspace{0in}
\begin{enumerate}
\item (The PD-embedded situation) Let $A$ be a $p$-torsionfree ring, and let $I$ be an ideal containing $p$ such that the quotients $I_r := I/p^r A \subset A/p^r A =: A_r$ are endowed with compatible PD-structures, all of which we (abusively) call $\gamma$.  Let $R$ be the $\F_p$-algebra $A/I$, which agrees with $A_r/I_r$ for all $r$.
\item (The PD-embedded situation with Frobenius) Suppose additionally that we are given an endomorphism $\phi \colon A \to A$, lifting the absolute Frobenius endomorphism of $A/pA$, such that each induced map $\phi_r \colon A_r \to A_r$ is a PD-morphism $(A_r, I_r, \gamma) \to (A_r, I_r, \gamma)$.  (In fact the condition $\phi(I_r) \subseteq I_r$ is automatic, as $\phi$ lifts the Frobenius endomorphism of $A_1$ and thus also that of its quotient $R$.)
\end{enumerate}
\end{situation}

\begin{example}
\hspace{0in}
\begin{enumerate}
\item In the lifted situation \ref{sit:lifted_situation}, the lift $(A, I = (p))$ satisfies the hypotheses of the PD-embedded situation, where $I_r \subset A_r$ carries the PD-structure $\brackets$ of \ref{para:notation_PD_structures}.

\item In the embedded situation \ref{sit:embedded_situation}, suppose the completed PD-envelope $B$ is $p$-torsionfree.  Then $(B, \overline{I})$ satisfies the hypotheses of the PD-embedded situation, where $\overline I_r \subset B_r$ carries its canonical PD-structure $\gamma$.
\end{enumerate}
Both of these examples carry through when adding Frobenius endomorphisms.
\end{example}

\begin{remark}
We do \emph{not} assume in Situation \ref{sit:PD_lift} that $I \subset A$ itself carries a PD-structure.  For example, we may have $(A, I) = (\Z, (p))$, which lacks divided powers because we cannot divide by integers that are prime to $p$.  Of course, this is not a problem in the $p^r$-torsion setting, as the quotients $A_r$ inherit PD-structures from the PD-structure on $(p) \subset \Z_p$.
On the other hand, the PD-structure can only be interpreted in terms of honest division by $p$ in the $p$-torsionfree ring $A$.  The need to switch between $p$-torsionfree and $p$-power torsion rings will make some proofs more complicated, so we invite the reader to imagine on first reading this section that $A$ is a $p$-torsionfree $\mathbb Z_{(p)}$-algebra, which would allow us to perform all necessary divisions in $A$ rather than passing to its quotients.
\end{remark}

\begin{mydef} \label{def:completed_dR}
Let $A$ be a $p$-torsionfree ring.
\begin{enumerate}
\item The \emph{completed de Rham complex} of $A$ is $\widehat{\Omega}^*_{A} := \limarrow_r \Omega^*_{A/p^r A}$.
\item If $I \subset A$ is an ideal containing $(p)$ and we have compatible PD-structures $\gamma$ on each of the ideals $I/p^r A \subset A/p^r A$, then the \emph{completed PD-de Rham complex} of $A$ is
\begin{align*}
\widehat{\Omega}^*_{A, \gamma} = \limarrow_r \Omega^*_{A/p^r A, \gamma},
\end{align*}
\end{enumerate}
\end{mydef}

\begin{remark}
Since the de Rham complex is compatible with base change, we have $\Omega^*_{A/p^r A} = \Omega^*_A/p^r \Omega^*_A$ for each $r$, and thus $\widehat{\Omega}^*_A$ agrees with the $p$-adic completion of the usual de Rham complex $\Omega^*_A$.
However, even when $A$ is $p$-adically complete, its completed de Rham complex need not agree with its naive de Rham complex.  For example, when $A = \Z_p$, the completed de Rham complex
$$
\widehat{\Omega}^*_{\Z_p} = \limarrow_r \Omega^*_{\Z/p^r\Z} = \Z_p
$$
is concentrated in degree $0$, but the naive de Rham complex $\Omega^*_{\Z_p}$ is unbounded.
\end{remark}

\begin{remark} \label{rmk:examples_of_completed_PD_dR}
The definition above will be useful later when we are given an $\F_p$-algebra $R$ which admits a lift; that is, a $p$-torsionfree ring $A$ such that $A/pA \simeq R$.  With no assumptions on $R$, we will also be interested in the tower $(\Omega^*_{W_r(R), \gamma})_r$, where the topology is coarser than the $p$-adic topology.  However, we will not work directly with its limit, as we are not able to give it the structure of a Dieudonn\'e complex.
\end{remark}

\begin{proposition} \label{prop:completed_dR_PD_only_kills_p-tors}
In the PD-embedded situation \ref{sit:PD_lift},
let $\pi_r$ be the quotient map $\Omega^*_{A_r} \to \Omega^*_{A_r, \gamma}$ for each $r$, and let $\pi = \limarrow(\pi_r) \colon \widehat{\Omega}^*_A \to \widehat{\Omega}^*_{A, \gamma}$ be their limit.  Then $\pi$ is surjective, and its kernel is contained in the exact $p$-torsion of $\widehat{\Omega}^*_A$.
\end{proposition}
\begin{proof}
The kernel of $\pi$ is contained in the exact $p$-torsion because the same is true of each $\pi_r$ by Proposition \ref{prop:PD_dR_p_suffices}.  As for surjectivity, we will apply the Mittag-Leffler criterion to the kernels $K_r = \ker(\pi_r)$.  Since each $K_r$ is generated by the elements $dx^{[n]} - x^{[n-1]} dx$ for all $x \in I_r$, it follows by lifting each $x$ to $I_{r+1}$ that the quotient maps send $K_{r+1}$ surjectively onto $K_r$.  So Mittag-Leffler implies that $R^1 \limarrow K_r = 0$, and thus $\pi$ is surjective.
\end{proof}

\begin{lemma} \label{lem:PD_compatibility_for_free_when_I_equals_p}
Continuing from Proposition \ref{prop:completed_dR_PD_only_kills_p-tors}, suppose that $I = (p)$, and the quotients $I_r$ are endowed with the PD-structures $\gamma = \brackets$ of \ref{para:notation_PD_structures}.  Then the maps $\pi_r \colon \Omega^*_{A_r} \to \Omega^*_{A_r, \gamma}$ and $\pi \colon \widehat{\Omega}^*_{A} \to \widehat{\Omega}^*_{A, \gamma}$ are isomorphisms.
\end{lemma}
\begin{proof}
To show that each $\pi_r$ is an isomorphism, we must show that for each $y \in (p) \subset A_r$, we have $d\gamma_n(y) = \gamma_{n-1}(y) dy$ already in $\Omega^*_{A_r}$.  Write $y$ as $px$ for some $x \in A_r$, and calculate:
\begin{align*}
d\gamma_n(px) & = d \left( \frac{p^n}{n!} x^n \right)
 = \frac{p^n}{n!} \cdot n x^{n-1} dx
 = \frac{p^{n-1} x^{n-1}}{(n-1)!} \cdot d(px)
 = \gamma_{n-1}(px) d(px).
\end{align*}
So each $\pi_r$ is an isomorphism, and their limit $\pi$ is as well.
\end{proof}

\begin{corollary} \label{cor:F_on_completed_PD_dR_when_I_equals_p}
Continuing from Lemma \ref{lem:PD_compatibility_for_free_when_I_equals_p},
suppose that we are also given a lift $\phi \colon A \to A$ of the absolute Frobenius of $A/pA$.  Then the divided Frobenius endomorphism $F$ on $\Omega^*_A$ (Proposition \ref{prop:BLM_dR_as_a_DC}) passes to each of the quotients $\Omega^*_{A_r, \gamma}$, and to their limit $\widehat{\Omega}^*_{A, \gamma}$.  This Frobenius makes $\widehat{\Omega}^*_{A, \gamma}$ a Dieudonn\'e complex, whose strictification is a saturated de Rham--Witt complex of the $\F_p$-algebra $A/pA$.
\end{corollary}
\begin{proof}
The PD-de Rham complexes agree with their non-PD analogues by Lemma \ref{lem:PD_compatibility_for_free_when_I_equals_p}.  The claims are proved for the latter by \cite[Variant 3.3.1, Corollary 4.2.3]{BLM}.
\end{proof}

\begin{para}
We will finish this section with Proposition \ref{prop:F_on_PD_dR_I_not_p}, a mild but essential generalization of Corollary \ref{cor:F_on_completed_PD_dR_when_I_equals_p} which relaxes the assumption that $I = (p)$.
Note that although this allows us to impose a stronger PD-compatibility on the de Rham complex of $A$ (and thus kill some of its $p$-torsion), we must still complete with respect to the $p$-adic topology.
In particular, the proposition does not apply to $\limarrow_r \Omega^*_{W_r(R), \gamma}$ for a typical (imperfect) $\F_p$-algebra $R$.

First we need a lemma:
\end{para}

\begin{lemma} \label{lem_for_prop:F_on_PD_dR_I_not_p}
Let $A$, $I$, $\gamma$, and $\phi$ be as in the PD-embedded situation with Frobenius, \ref{sit:PD_lift}.  Suppose $x \in I$ is a lift of $\overline{x} \in I_r$.  Then the elements $x^p$ and $\phi(x)$ both lie in $pA$, and 
$\frac{x^p}{p}$ reduces to the element $(p-1)! \gamma_p(\overline x) \in A_r$.

\begin{proof}
Since $\overline{x}^p = p \gamma_p(\overline{x}) \in p I_r$, we have $\overline{x}^p \in pA_r$ and thus $x^p \in pA$.  Since $\phi$ is a lift of Frobenius, it follows that $\phi(x) \equiv x^p \equiv 0 \pmod p$.  So it makes sense to speak of the elements $\frac{x^p}{p}$ and $\frac{\phi(x)}{p} \in A$.

Next we claim that $\frac{x^p}{p}$ is a representative of the divided power $(p-1)! \gamma_p(\overline{x})$.  Letting $x'$ denote the image of $x$ in $A_{r+1}$, it is clear that $x^p$ is a representative of $x'^p = p! \gamma_p(x')$, and dividing by $p$ implies that $\frac{x^p}{p}$ represents $(p-1)! \gamma_p(x')$ up to the $p$-torsion in $A_{r+1}$, namely $p^r A_{r+1}$.  This error term is killed when we pass to the quotient $A_r$.
\end{proof}
\end{lemma}

\begin{proposition} \label{prop:F_on_PD_dR_I_not_p}
Let $A$, $I$, $\gamma$, and $\phi$ be as in the PD-embedded situation with Frobenius, \ref{sit:PD_lift}.
Then the divided Frobenius $F \colon \Omega^*_A \to \Omega^*_A$ of Proposition \ref{prop:BLM_dR_as_a_DC} induces a (unique) map $F \colon \Omega^*_{A_r, \gamma} \to \Omega^*_{A_r, \gamma}$ for each $r$, and thus also $F \colon \widehat{\Omega}^*_{A, \gamma} \to \widehat{\Omega}^*_{A, \gamma}$ by passage to the limit.
\end{proposition}
\begin{proof}
It is clear (cf. \cite[Variant 3.3.1]{BLM}) that $F$ passes to an endomorphism of $\Omega^*_{A_r} = \Omega^*_A/p^r \Omega^*_A$, so we must only show that this preserves the dg-ideal generated by elements of the form $d\gamma_n(\overline x) - \gamma_{n-1}(\overline x) d \overline x$, with $\overline x \in I_r$ and $n \geq 1$.  Since $A_r$ is a $\Z/p^r\Z$-algebra, Proposition \ref{prop:PD_dR_p_suffices} allows us to restrict ourselves to $n = p$ and replace ``dg-ideal'' with ``ideal''.
Now fix $\overline x \in I_r$.  We must show that the elements $F(d\gamma_p(\overline x))$ and $F(\gamma_{p-1}(\overline x) d\overline x) \in \Omega^1_{A_r}$ map to the same element of $\Omega^1_{A_r, \gamma}$.

The proof will require some calculations.  Namely, letting $x \in I$ be a lift of $\overline x$, we claim that the elements $Fd(\frac{x^p}{p})$ and $F(x^{p-1} dx) \in \Omega^1_A$ reduce to the same element of $\Omega^1_{A_r, \gamma}$.  To see why this suffices, note that the former reduces to $(p-1)! Fd(\gamma_p(\overline{x}))$ by Lemma \ref{lem_for_prop:F_on_PD_dR_I_not_p}, the latter reduces to $(p-1)! F(\gamma_{p-1}(\overline x) d\overline x)$ since $\overline{x}^{p-1} = (p-1)! \gamma_{p-1}(\overline x)$, and we can divide by the common factor $(p-1)!$ in the $\Z/p^r\Z$-algebra $A_r$.  To prove our claim, we first calculate in $\Omega^1_A$ using the formula of Proposition \ref{prop:BLM_dR_as_a_DC}; all divisions in $A$ by powers of $p$ are justified by applying Lemma \ref{lem_for_prop:F_on_PD_dR_I_not_p} to either $x$ or $\frac{x^p}{p}$:
\begin{align*}
Fd\left( \frac{x^p}{p} \right) & = \left( \frac{x^p}{p} \right)^{p-1} d\left( \frac{x^p}{p} \right) + d \left( \frac{\phi\left( \frac{x^p}{p} \right) - \left( \frac{x^p}{p} \right)^p}{p} \right) \\
& = \left( \frac{x^{p}}{p} \right)^{p-1} d\left( \frac{x^p}{p} \right) + d \left( \frac{\phi(x)^p}{p^2} \right) - d \left( \frac{x^{p^2}}{p^{p+1}} \right) \\
& = \left( \frac{x^{p}}{p} \right)^{p-1} d\left( \frac{x^p}{p} \right) + p^{p-2} d \left( \frac{\phi(x)}{p} \right)^p - d \left( \frac{x^{p^2}}{p^{p+1}} \right) \\
& = \left( \frac{x^{p}}{p} \right)^{p-1} d\left( \frac{x^p}{p} \right) + p^{p-1} \left( \frac{\phi(x)}{p} \right)^{p-1} d \left( \frac{\phi(x)}{p} \right) - d \left( \frac{x^{p^2}}{p^{p+1}} \right) \\
& = \left( \frac{x^{p}}{p} \right)^{p-1} d\left( \frac{x^p}{p} \right) + \phi(x)^{p-1} d \left( \frac{\phi(x)}{p} \right) - d \left( \frac{x^{p^2}}{p^{p+1}} \right)
\end{align*}
and
\begin{align*}
F(x^{p-1} dx) & = \phi(x^{p-1}) \cdot F(dx) \\
& = \phi(x)^{p-1} \cdot \left( x^{p-1} dx + d \left( \frac{\phi(x) - x^p}{p} \right) \right) \\
& = \phi(x)^{p-1} \cdot \left( x^{p-1} dx + d \left( \frac{\phi(x)}{p} \right) - d \left( \frac{x^p}{p} \right) \right) \\
& = \phi(x)^{p-1} x^{p-1} dx + \phi(x)^{p-1} d \left( \frac{\phi(x)}{p} \right) - \phi(x)^{p-1} d \left( \frac{x^p}{p} \right)
\end{align*}
Note that both of these expressions contain the term $\phi(x)^{p-1} d(\phi(x)/p)$.  We claim that when we quotient down to $\Omega^1_{A_r, \gamma}$, all the remaining terms cancel.  Using Lemma \ref{lem_for_prop:F_on_PD_dR_I_not_p} again, we have
\begin{align*}
\frac{x^p}{p} & \mapsto (p-1)! \gamma_p(\overline x) \in A_r \text{\qquad and} \\
\frac{x^{p^2}}{p^{p+1}} = \frac{ \left( \frac{x^p}{p} \right)^p}{p} & \mapsto (p-1)! \gamma_p((p-1)! \gamma_p(\overline x)) = (p-1)!^{p+1} \gamma_p(\gamma_p(\overline x)) \in A_r,
\end{align*}
so the remaining terms of the former map to:
\begin{align*}
& ((p-1)! \gamma_p(\overline{x}))^{p-1} d\left( (p-1)! \gamma_p(\overline{x})) - d \left( (p-1)!^{p+1} \gamma_p(\gamma_p(\overline x) \right) \right) \\
& = (p-1)!^p \gamma_p(\overline{x})^{p-1} d\gamma_p(\overline{x}) - (p-1)!^{p+1} d\gamma_p(\gamma_p(\overline x)) \\
& = (p-1)!^{p+1} \left( \gamma_{p-1}(\gamma_p(\overline{x})) d\gamma_p(\overline{x}) -d\gamma_p(\gamma_p(\overline x)) \right) \\
& = 0 \in \Omega^*_{A_r, \gamma}
\end{align*}
by applying the PD-compatibility relation to the element $\gamma_p(\overline x)$.  The remaining terms of the latter map to
\begin{align*}
& \phi(\overline x)^{p-1} \overline{x}^{p-1} d\overline{x}  - \phi(\overline{x})^{p-1} d( (p-1)! \gamma_p(\overline{x})) \\
& = \phi(\overline x)^{p-1} (p-1)! \left( \gamma_{p-1}(\overline{x}) d\overline{x}  - d( \gamma_p(\overline{x})) \right) \\
& = 0 \in \Omega^*_{A_r, \gamma},
\end{align*}
by applying the same PD-compatibility to the element $\overline{x}$.
\end{proof}

\begin{corollary} \label{cor:PD_dR_computes_dRW}
In the situation of Proposition \ref{prop:F_on_PD_dR_I_not_p}, the object $(\widehat{\Omega}^*_{A, \gamma}, d, F)$ is a Dieudonn\'e algebra, and its strictification is a saturated de Rham--Witt complex associated to the $\F_p$-algebra $A/pA$.
\end{corollary}
\begin{proof}
All the Dieudonn\'e algebra properties are inherited via the quotient map $\widehat{\Omega}^*_A \to \widehat{\Omega}^*_{A, \gamma}$.  Moreover, by Proposition \ref{prop:completed_dR_PD_only_kills_p-tors}, this quotient map is surjective with $p$-torsion kernel, so it becomes an isomorphism after taking saturations, and in particular also after taking strictifications.  The result then follows from Construction \ref{con:BLM_lifted_construction}.
\end{proof}


\section{Algebra with Dieudonn\'e complexes}
\label{ch:DC_algebra}

The de Rham--Witt complex $\W\Omega^*_{R, \mathcal E}$ with coefficients in a unit-root $F$-crystal $\mathcal E$ will have the structure of a dg-module over the Dieudonn\'e algebra $\W\Omega^*_R$.  Accordingly, we begin by giving an account of modules in the categories $\DC$ and $\DC_{\str}$.


\subsection{Modules in $\DC$ and $\DC_{\str}$}
\label{sec:modules_in_DC}

\begin{para} \label{para:BLM_tensor}
In this section, we recall the symmetric monoidal structures with which Bhatt--Lurie--Mathew endows the categories $\DC$ and $\DC_{\str}$, denoted $\otimes$ and $\otimes^{\str}$,\footnote{We have promoted the subscript ``$\str$'' to a superscript in order to make room for a base ring later; cf. Definition \ref{def:tensor_over_a_base}.} and discuss the resulting categories of algebra and module objects.
\end{para}

\begin{mydef} \label{def:tensor_product}
(\cite[Remarks 2.1.5 and 7.6.4]{BLM})
If $M^*$ and $N^*$ are Dieudonn\'e complexes, then $M^* \otimes N^*$ is given by the tensor product of the underlying complexes of abelian groups, equipped with $F = F_M \otimes F_N$ and $d$ defined by the graded Leibniz rule.  The unit of this tensor product is $\Z = \Z[0]$, with Frobenius acting as the identity.  The symmetry $M^* \otimes N^* \stackrel{\sim}{\longrightarrow} N^* \otimes M^*$ is defined by
\begin{equation} \label{eqn:symmetry}
x \otimes y \mapsto (-1)^{|x| \cdot |y|} y \otimes x
\end{equation}
when $x$ and $y$ are homogeneous.

The strictified tensor product $\otimes^{\str}$ is
the unique symmetric monoidal structure on $\DC_{\str}$ making the strictification functor $\W\Sat \colon \DC \to \DC_{\str}$ symmetric monoidal.  In particular, we have a canonical isomorphism
$$
M^* \otimes^{\str} N^* \simeq \W\Sat(M^* \otimes N^*)
$$
for each $M^*$, $N^* \in \DC_{\str}$.  The unit of $\otimes^{\str}$ is $\W\Sat(\Z) = \Z_p$, again concentrated in degree 0 with trivial Frobenius.
\end{mydef}

\begin{remark} \label{rmk:internal_Hom}
In \cite[\textsection 3.4]{thesis}, we endow $\DC$ with an internal Hom functor, inspired by Ekedahl's $*$-product (\cite[I, \textsection 5]{ekedahl2}).  We do not know whether $\DC_{\str}$ admits an internal Hom functor.
\end{remark}

\begin{mydef} \label{def:alg_DC}
By an \emph{(associative) algebra in $\DC$}, we mean a monoid for the monoidal structure $\otimes$; that is, a Dieudonn\'e complex $A^*$ equipped with a multiplication map $m \colon A^* \otimes A^* \to A^*$ and a unit map $1 \colon \Z \to A^*$ in $\DC$, making the usual associativity and unit diagrams
commute.  We call $A^*$ \emph{commutative} if multiplication is compatible with the symmetry of the tensor product; this is the usual graded-commutative law.  We denote the categories of such objects, with the obvious morphisms, by $\Alg(\DC)$ and $\CAlg(\DC)$.

Algebra objects in $\DC_{\str}$ are defined analogously, using the symmetric monoidal structure $\otimes^{\str} = \W\Sat(- \otimes -)$ and its unit $\Z_p = \W\Sat(\Z)$.

\end{mydef}

\begin{mydef} \label{def:mod_DC}
For $A^*$ in $\Alg(\DC)$, we define a \emph{left $A^*$-module in $\DC$} to be a Dieudonn\'e complex $M^*$ equipped with a morphism $m_M \colon A^* \otimes M^* \to M^*$, such that the action diagrams
\begin{equation} \label{eqn:action_1}
\begin{tikzcd}
A^* \otimes A^* \otimes M^* \arrow[r, "m_A \otimes \id"] \arrow[d, "\id \otimes m_M"'] & A^* \otimes M^* \arrow[d, "m_M"] \\
A^* \otimes M^* \arrow[r, "m_M"]                                                       & M^*                             
\end{tikzcd}
\end{equation}
and
\begin{equation} \label{eqn:action_2}
\begin{tikzcd}
\Z \otimes M^* \arrow[rr, "1 \otimes \id"] \arrow[rd, equal] &     & A^* \otimes M^* \arrow[ld, "m_M"] \\
                                                                           & M^* &                                  
\end{tikzcd}
\end{equation}
commute.  For $A^* \in \Alg(\DC_{\str})$, we define a \emph{left $A^*$-module in $\DC_{\str}$} analogously, where the tensor products are replaced by $\otimes^{\str}$.

We denote the categories of such modules, with the obvious morphisms, as $A^*-\mathrm{mod}_{\DC}$ and $A^*-\mathrm{mod}_{\str}$ respectively.  We can similarly define the categories of right modules $\mathrm{mod}_{\DC}-A^*$ and $\mathrm{mod}_{\str}-A^*$, and bimodules $A^*-\mathrm{mod}_{\DC}-B^*$ and $A^*-\mathrm{mod}_{\str}-B^*$; as always, the two actions on a bimodule are required to commute.
\end{mydef}

\begin{remark} \label{DC_vs_CAlg}
As noted in \cite[Remark 3.1.5]{BLM}, a Dieudonn\'e algebra is just a commutative algebra object in $\DC$ satisfying a few extra conditions.  We will have no need for these conditions except inasmuch as they are built into the universal property of $\W\Omega^*_R$ for a $k$-algebra $R$.
\end{remark}

We have the following concrete interpretations of algebra and module objects in $\DC$ and $\DC_{\str}$:

\begin{lemma} \label{lem:description_of_algebras_and_modules_in_DC_and_DCstr}
\hspace{0in}
\begin{enumerate}
\item[(a)]
An algebra object in $\DC$ is a graded algebra $A^*$ equipped with graded group maps $F \colon A^* \to A^*$ and $d \colon A^* \to A^{*+1}$ satisfying the rules
\begin{align*}
d^2 & = 0, \\
dF & = pFd, \\ 
d(ab) & = da \cdot b + (-1)^{|a|} a \cdot db, \\
F(ab) & = Fa \cdot Fb, \\ 
d(1) & = 0, \text{ and} \\
F(1) & = 1. 
\end{align*}
(In fact the identity $d(1) = 0$ is redundant, as it follows from the graded Leibniz rule.)
\item[(b)]
Given an algebra object $A^*$ in $\DC$, an object $M^*$ in $A^* \lmodDC$ is a graded left $A^*$-module equipped with graded group maps $F \colon M^* \to M^*$ and $d \colon M^* \to M^{*+1}$ such that
\begin{align*}
d^2 & = 0, \\
dF & = pFd, \\ 
d(am) & = da \cdot m + (-1)^{|a|} a \cdot dm, \text{ and} \\
F(am) & = Fa \cdot Fm. 
\end{align*}
\item[(c)]
Morphisms of algebra objects (resp. module objects) are simply morphisms of graded algebras (resp. graded modules) that are compatible with $F$ and $d$.
\item[(d)]
The category of algebra objects (resp. modules over a fixed strict algebra $A^*$) in $\DC_{\str}$ is isomorphic to the full subcategory of the category of algebra objects (resp. modules over $A^*$) in $\DC$ spanned by the objects that are strict as Dieudonn\'e complexes.
\end{enumerate}
\end{lemma}
\begin{proof}
Parts (a-c) follow immediately from unraveling the definitions.
We will sketch a proof of the algebra case of part (d); the module case is completely analogous.  Let $A^*$ be a strict Dieudonn\'e complex.  First note that the maps $m \colon A^* \otimes A^* \to A^*$ correspond bijectively to the maps $m^{\str} \colon A^* \otimes^{\str} A^* \to A^*$, by factoring $m$ as
$$
A^* \otimes A^* \overset{\rho}{\to} \W\Sat(A^* \otimes A^*) = A^* \otimes^{\str} A^* \overset{m^{\str}}{\to} A^*;
$$
in particular, we have $m^{\str} = \W\Sat(m)$.  The same reasoning applies to the unit maps $1 \colon \Z \to A^*$ and $1^{\str} \colon \Z_p \to A^*$.  Then the associativity diagram for $m^{\str}$ is the strictification of the associativity diagram for $m$, so the universal property of the strictification map
$$
\rho \colon A^* \otimes A^* \otimes A^* \to A^* \otimes^{\str} A^* \otimes^{\str} A^*
$$
implies that one commutes if and only if the other does.
The same argument applies to the unit diagram.  Thus the algebra structures of $A^*$ as an object of $\DC$ are in bijection with those of $A^*$ as an object of $\DC_{\str}$.

If $A^*$ and $B^*$ are algebra objects in $\DC_{\str}$ (which we can now identify with algebra objects in $\DC$ that are strict as Dieudonn\'e complexes), then a morphism $f \colon A^* \to B^*$ in $\DC_{\str}$ is the same as a morphism $f \colon A^* \to B^*$ in $\DC$; the same argument as above shows that $f$ is a morphism of algebra objects in $\DC_{\str}$ if and only if it is a morphism of algebra objects in $\DC$.  So we have given bijections of objects and of corresponding hom-sets in the two specified categories; these define an isomorphism of categories.
\end{proof}

\begin{remark}
Note that the identities above which only involve $d$ say exactly that $A^*$ is a differential graded algebra and $M^*$ is a differential graded module over it.  Thus we can also describe algebra and module objects in $\DC$ respectively as dg-algebras and dg-modules equipped with a graded group endomorphism $F$ satisfying some properties.
\end{remark}

\begin{para}
Suppose $\mathcal E = (\mathcal E, \phi_{\mathcal E})$ is a unit-root $F$-crystal on $\Cris(\Spec R/W)$, where $R$ is a $k$-algebra.  Our ultimate goal is to construct the de Rham--Witt complex $\W\Omega^*_{R, \mathcal E}$ as a module over $\W\Omega^*_R$ in $\DC_{\str}$, but we record two important precursors here, given by two kinds of de Rham complexes associated to $\mathcal E$.  We refer the interested reader to \cite[\textsection \textsection 2.6-2.7]{thesis} for more details and proofs.
\end{para}

\begin{example} \label{ex:classical_dR_DC_mod}
Suppose $R$, $A$, $B$, $\phi_A$, and $\phi_B$ are as in the Frobenius-embedded situation \ref{sit:embedded_situation}, and write $X = \Spec R$, $Y_r = \Spec A_r$.  For each $r > 0$, we get a complex
$$
\dR(\mathcal E(X \hookrightarrow Y_r)) := (\mathcal E(B_r) \otimes_{A_r} \Omega^*_{A_r}, \nabla),
$$
e.g. by \cite[Theorem 6.6, (i) $\to$ (iii)]{B-O}.  We call this the \emph{de Rham complex of $\mathcal E$ associated to the smooth embedding $X \hookrightarrow Y_r$}.
This comes with a functorial Frobenius endomorphism
$$
\phi = \phi_{\mathcal E}^* \otimes \phi_{A_r}^* \colon \mathcal E(B_r) \otimes_{A_r} \Omega^*_{A_r} \to \mathcal E(B_r) \otimes_{A_r} \Omega^*_{A_r},
$$
as well as a divided Frobenius
$$
F = \phi_{\mathcal E} \otimes F \colon \mathcal E(B_r) \otimes_{A_r} \Omega^i_{A_r} \to \mathcal E(B_r) \otimes_{A_r} \Omega^i_{A_r},
$$
where the latter $F$ is given by taking the map of Proposition \ref{prop:BLM_dR_as_a_DC} modulo $p^r$.  Note that while $\phi$ is an endomorphism of complexes, $F$ is only a graded group endomorphism satisfying $\nabla F = pF \nabla$, and we have $p^i F = \phi$ as endomorphisms of the degree-$i$ piece of the complex.
Thus, passing to the limit, we get a Dieudonn\'e complex $(\widehat{\dR}(\mathcal E(X \hookrightarrow Y_{\bullet})), \nabla, F)$.

Moreover, $\dR(\mathcal E(X \hookrightarrow Y_r))$ is a dg-module over the dg-algebra $\dR(\mathcal O(X \hookrightarrow Y_r))$, and the Frobenius endomorphisms $\phi_{\mathcal E}^* \otimes \phi_{A_r}^*$ and $\phi_{\mathcal E} \otimes F$ are respectively semilinear over their trivial-coefficient analogues.  Thus one can check using Lemma \ref{lem:description_of_algebras_and_modules_in_DC_and_DCstr} that $\widehat{\dR}(\mathcal O(X \hookrightarrow Y_{\bullet}))$ is an algebra object in $\DC$ (in fact it is a Dieudonn\'e algebra), and in general, $\widehat{\dR}(\mathcal E(X \hookrightarrow Y_{\bullet}))$ is a module over it in $\DC$.
\end{example}

\begin{example} \label{ex:PD_dR_DC_mod}
Suppose $R$ admits a $p$-torsionfree PD-embedding with Frobenius $(A, I, \gamma, \phi)$ as in Situation \ref{sit:PD_lift}, where $\gamma$ denotes the PD-structure on $I_r \subset A_r$ for each $r$.  Then $(T_r = \Spec A_r, \gamma)$ is a PD-thickening of $\Spec R$, equipped with an endomorphism $\phi_{A_r}$ lying over $F_R$.
Thus, for each $r > 0$, we get a complex
$$
\dR(\mathcal E(A_r, \gamma)) := (\mathcal E(A_r) \otimes_{A_r} \Omega^*_{A_r, \gamma}, \nabla),
$$
where the connection $\nabla$ is given by \cite[\textsection II, Proposition 1.1.5]{etesse}.  We call this the \emph{PD-de Rham complex associated to $\mathcal E$ over the PD-thickening $T_r$}.
As before, this comes with a functorial Frobenius endomorphism
$$
\phi = \phi_{\mathcal E} \otimes \phi^*_{A_r} \colon \mathcal E(A_r) \otimes_{A_r} \Omega^*_{A_r, \gamma} \to \mathcal E(A_r) \otimes_{A_r} \Omega^*_{A_r, \gamma}
$$
and a divided Frobenius
$$
F = \phi_{\mathcal E} \otimes F \colon \mathcal E(A_r) \otimes_{A_r} \Omega^*_{A_r, \gamma} \to \mathcal E(A_r) \otimes_{A_r} \Omega^*_{A_r, \gamma},
$$
where the latter $F$ is the divided Frobenius of Proposition \ref{prop:F_on_PD_dR_I_not_p}.
We have the same compatibilities as before; thus, passing to the limit, we get a Dieudonn\'e algebra $\widehat{\dR}(\mathcal O(A, \gamma)) = \widehat{\Omega}^*_{A, \gamma}$ and a module $\widehat{\dR}(\mathcal E(A, \gamma))$ over it in $\DC$.
\end{example}

\begin{remark} \label{rmk:PD_dR_functorialities}
The de Rham complexes of Examples \ref{ex:classical_dR_DC_mod} and \ref{ex:PD_dR_DC_mod} are functorial in $(\mathcal E, \phi_{\mathcal E})$, and also in the tuples $(X \hookrightarrow Y_r/W)$ and $(X/W(k), T_r)$ respectively.  We will often use the following property of the latter functoriality (which we state in geometric language):  given compatible maps $g \colon X' \to X$, $f \colon \Spec k' \to \Spec k$, and $h \colon T_r' \to T_r$, the induced pullback map
$$
h^*_{\mathcal E} \otimes h^*_{\Omega} \colon h^{-1}(\mathcal E_T \otimes_{\mathcal O_T} \Omega^*_{T, \gamma}) \to (g^*_{\cris} \mathcal E)_{T'} \otimes_{\mathcal O_{T'}} \Omega^*_{T', \gamma'}
$$
expresses its target $\dR((g_{\cris}^* \mathcal E)_{T', \gamma'})$ as the dg-module base change $\Omega^*_{T', \gamma'} \otimes_{h^{-1}(\Omega^*_{T, \gamma})} h^{-1}(\dR(\mathcal E_{T, \gamma}))$.
\end{remark}

\begin{remark} \label{rmk:PD_dR_over_Witt_vectors}
It will later be useful to work with de Rham complexes of $\mathcal E$ on the tower of truncated Witt vectors, $T_r = \Spec W_r(R)$.  Unfortunately, even when $W(R)$ is $p$-torsionfree (i.e. $R$ is reduced), this does not generally fit within Situation \ref{sit:PD_lift}, as we do not typically have $W_r(R) = W(R)/p^r$.
In this case, we can still write down a PD-de Rham complex
$$
\dR(\mathcal E(W_r(R), \gamma)) = \mathcal E(W_r(R)) \otimes_{W_r(R)} \Omega^*_{W_r(R), \gamma},
$$
equipped with the undivided Frobenius $\phi$ and with the structure of a dg-module over the dg-algebra $\dR(\mathcal O(W_r(R), \gamma)) = \Omega^*_{W_r(R), \gamma}$.  However, we no longer have a divided Frobenius (even on $\Omega^*_{W_r(R), \gamma}$), and so we are not able to give the limit
$$
\limarrow_r \mathcal E(W_r(R)) \otimes_{W_r(R)} \Omega^*_{W_r(R), \gamma}
$$
the structure of a Dieudonn\'e complex.
\end{remark}

Finally, we record the following lemma for later use.
\begin{lemma} \label{lem:A_to_M}
If $A^*$ is an algebra object in $\DC$ and $M^*$ is a left $A^*$-module in $\DC$, then we have a bijection
$$
\Hom_{A^* \lmodDC}(A^*, M^*) \leftrightarrow \{ x \in M^0:  Fx = x, dx = 0\}
$$
given by $f \leftrightarrow f(1)$.
\begin{proof}
An $A^*$-module homomorphism $f \colon A^* \to M^*$ is a map of graded modules commuting with $F$ and $d$.  In particular, $f$ must be given by the rule $f(a) = ax$, where $x = f(1) \in M^0$.  We claim that this $f$ is a map in $A^* \lmodDC$ if and only if $Fx = x$ and $dx = 0$.  These conditions are necessary because we always have $F(1) = 1$ and $d(1) = 0$ in $A^*$.  Conversely, given any $x \in M^0$ such that $Fx = x$ and $dx = 0$, $f(a) = ax$ does define a map in $A^* \lmodDC$, since it is a map of graded $A^*$-modules and we have the compatibilities
\[
f(Fa) = Fa \cdot x = Fa \cdot Fx = F(ax) = F(f(a))
\]
and (taking $a$ homogeneous)
\[
f(da) = da \cdot x = da \cdot x + (-1)^{|a|} a \cdot dx = d(ax) = d(f(a)). \qedhere
\]
\end{proof}
\end{lemma}

\begin{remark} \label{rmk:A_to_M_strict_case}
In the situation of Lemma \ref{lem:A_to_M}, if the module $M^*$ is strict, then in fact the condition $Fx = x$ implies the condition $dx = 0$, since for each $r$ we have
$$
dx = dF^r x = p^r F^r dx \equiv 0 \in \W_r M^*.
$$
\end{remark}


\subsection{More on modules in $\DC$ and $\DC_{\str}$}
\label{sec:tensor_over_a_base}

\begin{para}
In this entirely formal section, we describe a few natural operations on module objects in $\DC$ and $\DC_{\str}$:  tensor product over a base algebra (including the special case of base change along a morphism of algebras), and strictification of a module in $\DC$.
\end{para}

\begin{mydef} \label{def:tensor_over_a_base}
Suppose $A^*$ is an algebra object in $\DC$, for example a Dieudonn\'e algebra.  If $M^*$ and $N^*$ are respectively left and right $A^*$-modules, we define $M^* \otimes_{A^*} N^*$ as the coequalizer of the diagram
\[
\begin{tikzcd}
M^* \otimes A^* \otimes N^* \arrow[r, "m_M \otimes \id_N", shift left] \arrow[r, "\id_M \otimes m_N"', shift right] & M^* \otimes N^*
\end{tikzcd}
\]
in $\DC$.
If $A^*$, $M^*$, and $N^*$ are strict, then we define the strictified tensor product $M^* \otimes_{A^*}^{\str} N^*$ as the coequalizer of the analogous diagram
\[
\begin{tikzcd}
M^* \otimes^{\str} A^* \otimes^{\str} N^* \arrow[r, "m_M \otimes \id_N", shift left] \arrow[r, "\id_M \otimes m_N"', shift right] & M^* \otimes^{\str} N^*
\end{tikzcd}
\]
in $\DC_{\str}$.  These respectively define bifunctors
\begin{align*}
\otimes_{A^*} & \colon (\rmodDC A^*) \times (A^* \lmodDC) \to \DC \\
\otimes_{A^*}^{\str} & \colon (\rmodstr A^*) \times (A^* \lmodstr) \to \DC_{\str}.
\end{align*}
\end{mydef}

\begin{remark} \label{rmk:forgetful_functor_commutes_with_limits_and_colimits}
By interpreting $\DC$ as the category of graded left modules over the noncommutative graded ring $R^* = \Z \langle d, F \rangle / (d^2, dF - pFd)$ (with $d$ in degree $1$ and $F$ in degree $0$), one can show that $\DC$ admits all limits and colimits, and they are compatible with the forgetful functor $\DC \to \Gr(\Ab)$.  In particular, the coequalizer defining $M^* \otimes_{A^*} N^*$ may be computed at the level of the underlying graded groups.  Moreover, we can equivalently describe $M^* \otimes_{A^*}^{\str} N^*$ as $\W\Sat(M^* \otimes_{A^*} N^*)$, as $\W\Sat$ is a left adjoint and therefore commutes with colimits.
\end{remark}

\begin{remark} \label{rmk:tensor_module_on_other_side}
Suppose we have algebra objects $A^*$ and $B^*$ in $\DC$.  If $M^*$, $N^*$, and $P^*$ have the appropriate module structures in $\DC$, then a diagram chase allows us to identify
$$
M^* \otimes_{A^*} (N^* \otimes_{B^*} P^*) \text{ and } (M^* \otimes_{A^*} N^*) \otimes_{B^*} P^*
$$
in $\DC$, together with their quotient maps from $M^* \otimes N^* \otimes P^*$.
This allows us to perform various standard module constructions in $\DC$.  For example, if $M^*$ is a $(B^*, A^*)$-bimodule, and $N^*$ is a left $A^*$-module, then $M^* \otimes_{A^*} N^*$ has the structure of a left $B^*$-module; its multiplication map is the composition
$$
B^* \otimes (M^* \otimes_{A^*} N^*) \simeq (B^* \otimes M^*) \otimes_{A^*} N^* \overset{m_M \otimes \id}{\longrightarrow} M^* \otimes_{A^*} N^*.
$$
(Note that the absolute tensor product here coincides with the tensor product over the unit object $\Z$.)  An important special case is where $M^* = B^*$ is an $A^*$-algebra.
All of this carries over to the strict setting as well.
\end{remark}

\begin{para} \label{para:WSatOmega}
We finish this section by showing that for $A^*$ an algebra object in $\DC$, the strictification of an $A^*$-module in $\DC$ is a $\W\Sat(A^*)$-module in $\DC_{\str}$.
A typical case of interest is $A^* = \widehat{\Omega}^*_A$, where $A$ is a $p$-torsionfree ring equipped with a lift of the Frobenius endomorphism of $R = A/pA$.  By Bhatt--Lurie--Mathew's Construction \ref{con:BLM_lifted_construction}, we have $\W\Sat(\widehat{\Omega}^*_A) = \W\Omega^*_R$, so this lemma says that the strictification of an $\widehat{\Omega}^*_A$-module in $\DC$ is a $\W\Omega^*_R$-module in $\DC_{\str}$.  The same is true of $A^* = \Omega^*_{W(R_{\red})}$, in view of Construction \ref{con:BLM_construction_W}.
\end{para}

\begin{lemma} \label{lem:WSat_module}
Let $A^*$ be an algebra object in $\DC$ (for example, a Dieudonn\'e algebra), and let $M^*$ be a left $A^*$-module in $\DC$.
\begin{enumerate}
\item The strictification $\W\Sat(M^*)$ carries a unique strict left $\W\Sat(A^*)$-module structure such that $\rho \colon M^* \to \W\Sat(M^*)$ is a morphism in $A^* \lmodDC$.
\item Sending $f \colon M^* \to N^*$ to $\W\Sat(f) \colon \W\Sat(M^*) \to \W\Sat(N^*)$ makes $\W\Sat$ a functor from $A^* \lmodDC$ to $\W\Sat(A^*) \lmodstr$.
\item The functor of part (2) is left-adjoint to the forgetful functor $\W\Sat(A^*) \lmodstr \to A^* \lmodDC$.
\end{enumerate}
\end{lemma}
\begin{proof}
First recall from \cite[Remark 7.6.4]{BLM} that for any $M^*$ and $N^*$ in $\DC$, the map
$$
M^* \otimes N^* \overset{\rho \otimes \rho}{\to} \W\Sat(M^*) \otimes \W\Sat(N^*) \overset{\rho}{\to} \W\Sat(M^*) \otimes^{\str} \W\Sat(N^*)
$$
is a strictification map.
For part (1), we are given a map $m_M \colon A^* \otimes M^* \to M^*$ in $\DC$ satisfying the action properties; i.e. such that the associativity and unit diagrams \ref{eqn:action_1} and \ref{eqn:action_2} commute.  We are looking for a map $m' \colon \W\Sat(A^*) \otimes^{\str} \W\Sat(M^*) \to \W\Sat(M^*)$ satisfying the analogous action properties, and which makes the outer region of the diagram
\[
\begin{tikzcd}
                                                      & A^* \otimes M^* \arrow[r, "m_M"] \arrow[d, "\rho \otimes \rho"] \arrow[ld, "\id \otimes \rho"'] & M^* \arrow[d, "\rho"] \\
A^* \otimes \W\Sat(M)^* \arrow[r, "\rho \otimes \id"] & \W\Sat(A^*) \otimes^{\str} \W\Sat(M^*) \arrow[r, "m'"]                               & \W\Sat(M^*)          
\end{tikzcd}
\]
commute.  But since the triangle on the left commutes, the outer region commutes if and only if the square does.  Since the left vertical map is a strictification map, there is a unique choice of $m'$ which makes the diagram commute, namely $m' = \W\Sat(m_M)$.  This also satisfies the associativity and unit properties, since the respective diagrams still commute after applying $\W\Sat$.

For (2), functoriality is clear; we must only show that for each $f \colon M^* \to N^*$ in $A^* \lmodDC$, $\W\Sat(f)$ is actually a morphism in $\W\Sat(A^*) \lmodstr$.  This follows by strictifying the commutative square
\[
\begin{tikzcd}
A^* \otimes M^* \arrow[r, "m_M"] \arrow[d, "\id \otimes f"'] & M^* \arrow[d, "f"] \\
A^* \otimes N^* \arrow[r, "m_N"]                             & N^*.               
\end{tikzcd}
\]
For (3), we must show that if $M^*$ is in $A^* \lmodDC$ and $N^*$ is in $\W\Sat(A^*) \lmodstr$, then any morphism $f \colon M^* \to N^*$ in $A^* \lmodDC$ factors uniquely through $\W\Sat(M^*)$ in $\W\Sat(A^*) \lmodstr$.  Uniqueness is clear, as there is a unique factoring map even in $\DC$, namely
$$
\W\Sat(f) \colon \W\Sat(M^*) \to \W\Sat(N^*) = N^*.
$$
This is in fact a morphism in $\W\Sat(A^*) \lmodstr$ by part (2).
\end{proof}


\subsection{Filtrations on modules in $\DC_{\str}$}
\label{sec:filtrations_on_modules}

\begin{para}
In this section, we examine the following question:  if $A^*$ is an algebra object in $\DC$ (for example a Dieudonn\'e algebra), and $M^*$ is a strict $A^*$-module in $\DC$, then what kind of structure does $\W_r M^*$ have?  It is only a complex of abelian groups a priori; recall that Frobenius maps $\W_r M^*$ to its quotient $\W_{r-1} M^*$ and not to itself.  We will begin with generalities, and then specialize to the case $A^* = \W\Omega^*_R$ that is of greatest interest to us.
\end{para}

\begin{lemma} \label{lem:Vr+dVr_submodule}
Let $A^*$ be an algebra object in $\DC$, and $M^*$ a saturated left $A^*$-module in $\DC$.  Then for each $r$, we have containments:
\begin{enumerate}
\item $A^* \cdot (V^r M^* + dV^r M^*) \subset V^r M^* + dV^r M^*$, and
\item $(V^r A^* + dV^r A^*) \cdot M^* \subset V^r M^* + dV^r M^*$ if $A^*$ is saturated.
\end{enumerate}
In particular, part (1) makes $\W_r M^*$ a dg-module over $A^*$, and part (2) makes it a dg-module over $\W_r A^*$ when $A^*$ is saturated.
\begin{proof}
Fix elements $a$ of $A^*$ and $x$ of $M^*$, which for convenience we take to be homogeneous.  We must show that all of the elements $a \cdot V^r x$, $(V^r a) \cdot x$, $a \cdot dV^r x$, and $(dV^r a) \cdot x$ lie in $V^r M^* + dV^r M^*$.  The strategy for all four will be the same:  apply $F^r$ to the given element, simplify, and eventually cancel $F^r$'s from both sides of the equation, using the fact that $F$ is injective on $M^*$.

For $a \cdot V^r x$, we have:
\begin{align*}
F^r(a \cdot V^r x) & = F^r(a) \cdot F^r V^r x \\
& = F^r(a) \cdot p^r x \\
& = p^r (F^r(a) \cdot x) \\
& = F^r V^r(F^r a \cdot x).
\end{align*}
Canceling $F^r$'s gives $a \cdot V^r x = V^r(F^r a \cdot x) \in V^r M^*$.

For $(V^r a) \cdot x$, we have:
\begin{align*}
F^r((V^r a) \cdot x) & = (F^r V^r a) \cdot F^r x \\
& = p^r a \cdot F^r x \\
& = F^r V^r (a \cdot F^r x),
\end{align*}
so $(V^r a) \cdot x = V^r(a \cdot F^r x) \in V^r M^*$.

For $a \cdot dV^r x$, we need the graded Leibniz rule:
\begin{align*}
F^r(a \cdot dV^r x) & = (F^r a) \cdot (F^r dV^r x) \\
& = (F^r a) \cdot dx \\
& = (-1)^{|a|} \left( d((F^r a) \cdot x) - (d F^r a) \cdot x \right) \\
& = (-1)^{|a|} \left( F^r d V^r((F^r a) \cdot x) - (p^r F^r d a) \cdot x \right) \\
& = (-1)^{|a|} \left( F^r d V^r((F^r a) \cdot x) - F^r V^r((F^r d a) \cdot x) \right) \\
& = (-1)^{|a|} F^r \left( d V^r((F^r a) \cdot x) - V^r((F^r d a) \cdot x) \right).
\end{align*}
Again we cancel $F^r$'s and see that the right-hand side lies in $V^r M^* + dV^r M^*$.

Finally, for $(dV^r a) \cdot x$, we have:  
\begin{align*}
F^r((dV^r a) \cdot x) & = (F^r d V^r a) \cdot (F^r x) \\
& = da \cdot F^r x \\
& = d(a \cdot F^r x) - (-1)^{|a|} a \cdot dF^r x \\
& = F^r d V^r(a \cdot F^r x) - (-1)^{|a|} a \cdot p^r F^r dx \\
& = F^r d V^r(a \cdot F^r x) - (-1)^{|a|} F^r V^r(a \cdot F^r dx) \\
& = F^r \left(d V^r(a \cdot F^r x) - (-1)^{|a|} V^r(a \cdot F^r dx) \right).
\end{align*}
Canceling $F^r$'s, we again see that $(dV^r a) \cdot x \in V^r M^* + dV^r M^*$.
\end{proof}
\end{lemma}

\begin{para}
Now let $R$ be an $\F_p$-algebra and $A^* = \W\Omega^*_R$.  If $M^*$ is a strict module over $A^*$, then the previous lemma tells us that $\W_r M^*$ is a dg-module over $\W_r A^*$.  Our next goal is to make this more concrete by specifying some de Rham complexes which can serve as convenient substitutes for the cdga $\W_r \Omega^*_R$.
\end{para}

\begin{lemma} \label{lem:PD_dR_to_WrOmega}
Let $R$ be an $\F_p$-algebra.  The dg-algebra map
$$
\Omega^*_{W(R)} \to \Omega^*_{W(R_{\red})} \overset{\rho}{\to} \W\Sat(\Omega^*_{W(R_{\red})}) = \W\Omega^*_R \twoheadrightarrow \W_r \Omega^*_R
$$
factors uniquely through the quotient map $\Omega^*_{W(R)} \to \Omega^*_{W_r(R), \gamma}$.
\begin{proof}
Uniqueness is clear, because $\Omega^*_{W(R)} \to \Omega^*_{W_r(R), \gamma}$ is surjective.  For existence, recall by Definition \ref{def:PD_dR_cx} and Lemma \ref{lem:quotient_of_PD_dR_algebras} that
the kernel of $\Omega^*_{W(R)} \to \Omega^*_{W_r(R), \gamma}$ is generated as a dg-ideal by elements of the forms
\begin{align*}
\omega_1 & = d\gamma_n(x) - \gamma_{n-1}(x) dx \text{ for } x \in VW(R), \text{ and} \\
\omega_2 & = V^r y \text{ for } y \in W(R).
\end{align*}
We must show that each of these maps to zero in $\W_r \Omega^*_R$.  But $\omega_1$ maps to zero even in $\W \Omega^*_R$, since it is killed by $n!$ and $\W\Omega^*_R$ is $\Z$-torsionfree.  As for $\omega_2$, since the map $$W(R) = \Omega^0_{W(R)} \to \W\Omega^0_R$$ is compatible with Verschiebung operators, we have $\rho(V^r x) = V^r \rho(x) \in \W\Omega^*_R$, which vanishes in $\W_r \Omega^*_R$.
\end{proof}
\end{lemma}

\begin{corollary} \label{cor:PD_dR_module}
Let $M^*$ be a strict $\W\Omega^*_R$-module.
\begin{enumerate}
\item The dg-$\Omega^*_{W(R)}$-module structure of $M^*$ coming from the strictification map $\Omega^*_{W(R)} \to \W\Omega^*_R$ induces a (unique) dg-$\Omega^*_{W_r(R), \gamma}$-module structure on $\W_r M^*$ for each $r$.
\item If $f \colon M^* \to N^*$ is a morphism of strict $\W\Omega^*_R$-modules, then $\W_r(f) \colon \W_r M^* \to \W_r N^*$ is a morphism of dg-$\Omega^*_{W_r(R), \gamma}$-modules for each $r$.  Thus $\W_r$ defines a functor
$$
\W\Omega^*_R \lmodstr \to \Omega^*_{W_r(R), \gamma} \ldgmod.
$$
\end{enumerate}
\begin{proof}
For part (a), Lemma \ref{lem:Vr+dVr_submodule} shows that the dg-$\Omega^*_{W(R)}$-module structure of $\W_r M^*$ factors through $\Omega^*_{W(R)} \to \W\Omega^*_R \to \W_r \Omega^*_R$.  In view of the map $\Omega^*_{W_r(R), \gamma} \to \W_r \Omega^*_R$ of Lemma \ref{lem:PD_dR_to_WrOmega}, it also factors through $\Omega^*_{W(R)} \to \Omega^*_{W_r(R), \gamma}$.

For part (b), we simply observe that $f$ is a morphism of dg-$\W\Omega^*_R$-modules and thus also of dg-$\Omega^*_{W(R)}$-modules, so $\W_r(f)$ is a morphism of dg-$\Omega^*_{W_r(R), \gamma}$-modules.
\end{proof}
\end{corollary}

The following two results play the same role for the lifted construction \ref{con:BLM_lifted_construction} of $\W\Omega^*_R$ that the previous two play for the Witt vector construction \ref{con:BLM_construction_W}.

\begin{lemma} \label{lem:dR_lift_to_WrOmega}
Let $R$, $(A, \phi)$, and $A_r$ be as in the Frobenius-lifted situation \ref{sit:lifted_situation}.
\begin{enumerate}
\item The dg-algebra map
$$
\widehat{\Omega}^*_A \overset{\rho}{\to} \W\Sat(\widehat{\Omega}^*_A) = \W\Omega^*_R \twoheadrightarrow \W_r \Omega^*_R
$$
factors uniquely through the PD-de Rham complex $\Omega^*_{A_r}$.
\item The morphism $\Omega^*_{A_r} \to \W_r \Omega^*_R$ of part (1) factors as
$$
\Omega^*_{A_r} \overset{h_r}{\to} \Omega^*_{W_r(R), \gamma} \to \W_r \Omega^*_R,
$$
where the first map is induced (via quotients and de Rham complex functoriality) by the canonical $\delta$-ring map $h \colon A \to W(R)$, and the second was constructed in Lemma \ref{lem:PD_dR_to_WrOmega}.
\end{enumerate}
\begin{proof}
Part (1) is clear, since the target is killed by $p^r$ and the natural map $\widehat{\Omega}^*_A \to \Omega^*_{A_r}$ is the quotient by $p^r$.  For part (2), we must check that the bottom triangle of the diagram below commutes; this will follow from the commutativity of the upper pentagon.
\[
\begin{tikzcd}
\Omega^*_A \arrow[rd, "h"'] \arrow[dd, two heads] \arrow[rr] &                                          & \widehat{\Omega}^*_A \arrow[rr, "\rho"] &                                           & \W\Omega^*_R \arrow[dd, two heads] \\
                                                  & \Omega^*_{W(R)} \arrow[dd, two heads] \arrow[rr]    &                                         & \Omega^*_{W(R_{\red})} \arrow[ru, "\rho"] &                         \\
\Omega^*_{A_r} \arrow[rd, "h_r"'] \arrow[rrrr]    &                                          &                                         &                                           & \W_r \Omega^*_R         \\
                                                  & {\Omega^*_{W_r(R), \gamma}} \arrow[rrru] &                                         &                                           &                        
\end{tikzcd}
\]
To prove that the upper pentagon commutes, it suffices by the universal property of $\Omega^*_A$ to check commutativity in degree $0$.
Recall from \cite[Proposition 3.6.2]{BLM} that we have $\W\Omega^0_R = W(\W_1 \Omega^0_R)$.  All of the maps above are compatible with the various Frobenius endomorphisms, and the target $\W\Omega^0_R$ is $p$-torsionfree, so the two maps $A \rightrightarrows \W\Omega^0_R$ are both $\delta$-ring maps.  But by construction (\cite[Proposition 4.1.4, Corollary 4.2.3]{BLM}), both maps pass to the same map $A/pA = R \overset{e}{\to} \W_1 \Omega^0_R$ when we take quotients.  It follows by the universal property of Witt vectors that they agree.
\end{proof}
\end{lemma}

\begin{corollary} \label{cor:PD_dR_module_lift}
Let $R$, $(A, \phi)$, and $A_r$ be as in the Frobenius-lifted situation \ref{sit:lifted_situation}, and let $M^*$ be a strict $\W\Omega^*_R$-module.
\begin{enumerate}
\item The dg-$\widehat{\Omega}^*_A$-module structure of $M^*$ coming from the strictification map $\widehat{\Omega}^*_A \to \W\Omega^*_R$ induces a (unique) dg-$\Omega^*_{A_r}$-module structure on $\W_r M^*$ for each $r$.
\item If $f \colon M^* \to N^*$ is a morphism of strict $\W\Omega^*_R$-modules, then $\W_r(f) \colon \W_r M^* \to \W_r N^*$ is a morphism of dg-$\Omega^*_{A_r}$-modules for each $r$.  Thus $\W_r$ defines a functor
$$
\W\Omega^*_R \lmodstr \to \Omega^*_{A_r} \ldgmod.
$$
\item The functor of part (2) factors as 
$$
\W\Omega^*_R \lmodstr \to \Omega^*_{W_r(R), \gamma} \ldgmod \to \Omega^*_{A_r} \ldgmod,
$$
where the first step is the functor of Corollary \ref{cor:PD_dR_module} and the second is restriction of scalars along $h_r \colon \Omega^*_{A_r} \to \Omega^*_{W_r(R), \gamma}$.
\end{enumerate}
\begin{proof}
For part (1), simply recall that the natural map $\widehat{\Omega}^*_A \to \Omega^*_{A_r}$ is the quotient by $p^r$.
For part (2), note that $f$ and therefore also $\W_r(f)$ are maps of dg-$\widehat{\Omega}^*_A$-modules, so the latter is a map of dg-$\Omega^*_{A_r}$-modules.
Part (3) follows from part (2) of Lemma \ref{lem:dR_lift_to_WrOmega}.
\end{proof}
\end{corollary}

\begin{remark} \label{rmk:PD_dR_dRW_etale_base_change}
If $R \to S$ is an \'etale map of $\F_p$-algebras, then Lemma \ref{lem:PD_dR_etale_base_change} and \cite[Corollary 5.3.5]{BLM} respectively tell us that the natural maps
$$
\Omega^*_{W_r(R), \gamma} \to \Omega^*_{W_r(S), \gamma}
\text{\quad and \quad}
\W_r \Omega^*_R \to \W_r \Omega^*_S
$$
are base-change maps along $W_r(R) \to W_r(S)$.
It follows that the natural dg-algebra map
$$
\Omega^*_{W_r(S), \gamma} \otimes_{\Omega^*_{W_r(R), \gamma}} \W_r \Omega^*_R \to \W_r \Omega^*_S
$$
is an isomorphism.
\end{remark}


\subsection{Strict Dieudonn\'e complexes and colimits}
\label{sec:DCstr_and_colimits}

\begin{para}
Given a diagram $(M_i^*)_i$ of strict Dieudonn\'e complexes, there are three things we might reasonably mean by the ``colimit'' of $M_i^*$.  Namely, we could mean their colimit
$$
\colim_{i, \str} M_i^*
$$
in the category $\DC_{\str}$, their colimit
$$
\colim_{i, \DC} M_i^*
$$
in the category $\DC$ (which by Remark \ref{rmk:forgetful_functor_commutes_with_limits_and_colimits} may be computed at the level of graded abelian groups), or the colimit
$$
(\colim_i \W_r M_i^*)_r
$$
of the associated strict Dieudonn\'e towers.  In this largely formal section, we will compare these notions of colimits.  This will later play a crucial role, via Proposition \ref{prop:colimit_of_dRW_complexes_is_dRW_complex}, in the proof of Theorem \ref{thm:etesse_comparison}.  We begin with what holds for arbitrary colimits, and then specialize to filtered colimits.
\end{para}

\begin{lemma} \label{lem:WSat_commutes_with_colimits}
The category $\DC_{\str}$ has all colimits, and they can be computed as the strictifications of the corresponding colimits in $\DC$.
\end{lemma}
\begin{proof}
The strictification functor is a left adjoint, so it commutes with all colimits.
\end{proof}

\begin{lemma} \label{lem:filtered_colimits_preserve_DCsat}
(cf. \cite[Proposition 4.3.3 and Remark 4.3.4]{BLM}) Filtered colimits (in $\DC$) of saturated Dieudonn\'e complexes are saturated.  In particular, $\DC_{\sat}$ admits filtered colimits, and they coincide with the corresponding colimits in $\DC$.
\end{lemma}
\begin{proof}
Let $(M^*_i)_{i \in I}$ be a filtered system in $\DC$ with each $M^*_i$ saturated.  Then $M^* := \colim_i M^*_i$ is $p$- and $F$-torsionfree because $p$- and $F$-torsion are given by finite limits of graded groups, and filtered colimits commute with these.
Thus we must only show that the map $$F \colon M^* \to \{x \in M^*:  dx \in pM^{*+1}\}$$ is surjective.
An element of the target can be represented by an element $x \in M^*_i$ such that $dx \in p M^*_j$ for some $i, j \in I$.  By choosing an element $\ell \in I$ with maps $i \to \ell$, $j \to \ell$, we may instead take $x \in M^*_{\ell}$ and $dx \in pM^*_{\ell}$.  Then $x$ lies in the image of $F$ on $M^*_{\ell}$, and thus also on $M^*$.
\end{proof}

\begin{remark}
Strict Dieudonn\'e complexes are \emph{not} closed under filtered colimits in $\DC$.  For example, the group $\Z_p$ concentrated in degree $0$ with $F = \id$ forms a strict Dieudonn\'e complex, but its sequential colimit along the multiplication-by-$p$ map is isomorphic to $\Q_p$, which is not strict.  Of course, the colimit of this system in $\DC_{\str}$ is $\W\Sat({\Q_p}) = 0$ by Lemma \ref{lem:WSat_commutes_with_colimits}.
\end{remark}

\begin{lemma} \label{lem:Wr_commutes_with_colimits}
If $(M_i^*)_i$ is a filtered diagram of saturated Dieudonn\'e complexes, then the canonical map
$$
\W_r(\colim_{i, \DC} M^*_i) \to \colim_i (\W_r M^*_i)
$$
is an isomorphism.
\end{lemma}
\begin{proof}
Since $\W_r$ is a cokernel (in the category of graded abelian groups), this follows from the fact that colimits commute with colimits.
\end{proof}

\begin{para}
We finish this section by showing that filtered colimits of Dieudonn\'e towers---or equivalently, the Dieudonn\'e towers corresponding to filtered limits in $\DC_{\str}$ under the equivalence of categories $\DC_{\str} \simeq \TD$ of \cite[Corollary 2.9.4]{BLM}---can be computed at the level of projective systems of graded abelian groups:
\end{para}

\begin{proposition} \label{prop:TD_equivalence_compatible_with_filtered_colimits}
If $(M_i^*)_i$ is a filtered diagram of strict Dieudonn\'e complexes, then for each $r$, we have a canonical isomorphism of graded abelian groups
$$
(\W_r (\colim_{i, \str} M^*_i))_r \simeq (\colim_i (\W_r M^*_i))_r,
$$
compatible with their respective quotient, Frobenius, and Verschiebung maps as $r$ varies.
\end{proposition}
\begin{proof}
Combining Lemmas \ref{lem:WSat_commutes_with_colimits}, \ref{lem:filtered_colimits_preserve_DCsat}, and \ref{lem:Wr_commutes_with_colimits}, we have isomorphisms of complexes
\begin{align*}
\W_r (\colim_{i, \str} M^*_i) & \simeq \W_r \Sat(\colim_{i, \DC}(M^*_i)) \\
& = \W_r (\colim_{i, \DC}(M^*_i)) \\
& \simeq \colim_i (\W_r (M^*_i)).
\end{align*}
These isomorphisms are clearly compatible with the quotient, Frobenius, and Verschiebung maps, as all of these maps are induced by the respective endomorphisms of the various strict Dieudonn\'e complexes.
\end{proof}

\begin{remark}
Of course, Lemma \ref{lem:WSat_commutes_with_colimits} and the equivalence of categories $\DC_{\str} \simeq \TD$ guarantee that $\TD$ admits \emph{arbitrary} colimits, computed by the left-hand side of Proposition \ref{prop:TD_equivalence_compatible_with_filtered_colimits}.  But in general these need not agree with the colimits of the underlying projective systems of graded groups.  For example, the coequalizer of the diagram
\[
\begin{tikzcd}
\Z_p \arrow[r, "p", shift left] \arrow[r, "0"', shift right] & \Z_p
\end{tikzcd}
\]
in $\DC_{\str}$ is zero, and thus so is the coequalizer of the corresponding diagram
\[
\begin{tikzcd}
\Big( (\Z/p^r\Z) \arrow[r, "p", shift left] \arrow[r, "0"', shift right] & (\Z/p^r\Z) \Big)_r
\end{tikzcd}
\]
in $\TD$.  But this latter coequalizer does not vanish when computed at the level of towers of graded groups.
\end{remark}


\section{Definition and first properties}
\label{ch:dRWM}


\subsection{de Rham--Witt modules over a $k$-algebra}
\label{sec:dRWM}

\begin{para} \label{para:intro_dRWM}
Let $k$ be a perfect field of characteristic $p$, $R$ a $k$-algebra, and $\mathcal E$ a unit-root $F$-crystal on $(\Spec R/W)_{\cris}$.  In this section, we will describe a certain category of modules over the saturated de Rham--Witt complex $\W\Omega^*_R$ which are endowed with some data coming from $\mathcal E$.  Our de Rham--Witt complex $\W\Omega^*_{R, \mathcal E}$ will be defined as the initial object of this category.

As in the classical story, the starting point of our definition is the PD-de Rham complex
$$ \dR(\mathcal E(W_r(R), \gamma)) := \mathcal E(W_r(R)) \otimes_{W_r(R)} \Omega^*_{W_r(R), \gamma} $$
of Remark \ref{rmk:PD_dR_over_Witt_vectors}.
We begin with the following observation.  If $M^*$ is a strict $\W\Omega^*_R$-module, then Corollary \ref{cor:PD_dR_module} gives $\W_r M^*$ the structure of a dg-$\Omega^*_{W_r(R), \gamma}$-module, and in particular also a graded $W_r(R)$-module.  We then have the following easy lemma:
\end{para}

\begin{lemma} \label{lem:construct_graded_not_dg_map_non-sheafy}
Suppose we are given a $\W\Omega^*_R$-module $M^*$ in $\DC_{\str}$, equipped with a $W_r(R)$-linear map $\iota \colon \mathcal E(W_r(R)) \to \W_r M^0$.  Then:
\begin{enumerate}
\item[(a)] The map $\iota$ extends uniquely to a map of graded left $\Omega^*_{W_r(R), \gamma}$-modules (not necessarily compatible with differentials)
$$
\iota^* \colon \mathcal E(W_r(R)) \otimes_{W_r(R)} \Omega^*_{W_r(R), \gamma} \to \W_r M^*.
$$
\item[(b)] If we are instead given a compatible family of maps $\iota_r \colon \mathcal E(W_r(R)) \to \W_r M^0$ for all $r$, then the resulting maps $\iota_r^*$ are also compatible.
\end{enumerate}
\begin{proof}
Part (a) is the tensor-hom adjunction for the homomorphism $W_r(R) \to \Omega^*_{W_r(R), \gamma}$ of graded rings; explicitly, the map is
$$
\iota^* \colon \mathcal E(W_r(R)) \otimes_{W_r(R)} \Omega^*_{W_r(R), \gamma} \stackrel{\iota \otimes \id}{\longrightarrow} \W_r M^0 \otimes_{W_r(R)} \Omega^*_{W_r(R), \gamma} \stackrel{m}{\to} \W_r M^*,
$$
where $m$ is the multiplication map.  Part (b) is then clear from this formula.
\end{proof}
\end{lemma}

\begin{mydef} \label{def:dRW_module}
By a \emph{de Rham--Witt module over $(R, \mathcal E)$} we will mean a collection of the following data:  a left $\W\Omega^*_R$-module $M^*$ in $\DC_{\str}$, equipped with $W_r(R)$-linear maps
$$
\iota_r \colon \mathcal E(W_r(R)) \to \W_r M^0
$$
for each $r$,
such that:
\begin{enumerate}
\item For each $r > 0$, the diagram
\begin{equation} \label{eqn:dRWM_R_compatibility}
\begin{tikzcd}
\mathcal E(W_r(R)) \arrow[d, two heads] \arrow[r, "\iota_r"] & \W_r M^0 \arrow[d, two heads] \\
\mathcal E(W_{r-1}(R)) \arrow[r, "\iota_{r-1}"]         & \W_{r-1} M^0           
\end{tikzcd}
\end{equation}
commutes, where the two vertical maps are the quotient maps.

\item For each $r > 0$, the diagram
\begin{equation} \label{eqn:dRWM_F_compatibility}
\begin{tikzcd}
\mathcal E(W_r(R)) \arrow[d, "F"'] \arrow[r, "\iota_r"] & \W_r M^0 \arrow[d, "F"] \\
\mathcal E(W_{r-1}(R)) \arrow[r, "\iota_{r-1}"]         & \W_{r-1} M^0           
\end{tikzcd}
\end{equation}
commutes, where the left vertical map is given by applying \eqref{eqn:F_on_sections} to the Frobenius map $\Spec W_{r-1}(R) \to \Spec W_r(R)$.

\item The maps
$$
\iota_r^* \colon \dR(\mathcal E(W_r(R), \gamma)) \to \W_r M^*
$$
of Lemma \ref{lem:construct_graded_not_dg_map_non-sheafy} are compatible with differentials.
\end{enumerate}

A morphism of de Rham--Witt modules over $(R, \mathcal E)$ is a morphism $f \colon M^* \to N^*$ of strict $\W\Omega^*_R$-modules such that $\iota_{r, N} = \W_r(f^0) \circ \iota_{r, M}$ for each $r$.  We call the resulting category $\dRWM_{R, \mathcal E}$.
\end{mydef}

\begin{remark}
The motivation for the maps $\iota_r$ is classical:  in the case of the trivial $F$-crystal $\mathcal E = \mathcal O_{X/W}$, they correspond to the requirement $M_n^0 = W_n(M_1^0)$ of \cite[I, D\'efinition 1.1, V1]{illusie}, which subsequent authors (e.g. \cite[Definition 4.4.1(3)]{BLM}) have usually replaced with the data of a map $W(R) \to M^0$.  The three subsequent conditions simply state that the maps $\iota_r$ must be compatible with quotient maps, Frobenius, and connections.
\end{remark}

\begin{remark} \label{rmk:redefine_dRWM}
Of course we can equivalently demand the data of the extension $\iota_r^*$, rather than only $\iota_r = \iota^0_r$.  In this case, we demand compatibility with $F$ only in degree 0, as the source does not generally have a divided Frobenius operator in positive degrees (cf. Remark \ref{rmk:PD_dR_over_Witt_vectors})---it only has an undivided Frobenius.   As for the other compatibilities, by Lemma \ref{lem:construct_graded_not_dg_map_non-sheafy}, the graded maps $\iota^*_r$ are $\Omega^*_{W_r(R), \gamma}$-linear if and only if the $\iota_r$ are $W_r(R)$-linear, and the $\iota^*_r$ are compatible with quotient maps (in all degrees) if and only if the $\iota_r$ are.  We have a similar statement for morphisms $f \colon M^* \to N^*$:  if $\iota_{r, N} = \W_r(f) \circ \iota_{r, M}$, then the compatibility $\iota^*_{r, N} = \W_r(f) \circ \iota^*_{r, M}$ in all degrees follows, because a graded $\Omega^*_{W_r(R), \gamma}$-module morphism from $\mathcal E(W_r(R)) \otimes_{W_r(R)} \Omega^*_{W_r(R), \gamma}$ is determined by its degree-0 component.

This allows us to formulate an alternative to Definition \ref{def:dRW_module} which focuses on the maps of complexes $\iota^*_r$ rather than their degree-0 components $\iota_r = \iota_r^0$.  We will allow ourselves to pass freely between the two definitions, but will use this latter definition more often in practice.
\end{remark}

\begin{altdef} \label{def:dRW_module_again}
A \emph{de Rham--Witt module over $(R, \mathcal E)$} is a collection of the following data:  a left $\W\Omega^*_R$-module $M^*$ in $\DC_{\str}$, equipped with maps
$$
\iota^*_r \colon \dR(\mathcal E(W_r(R), \gamma)) \to \W_r M^*
$$
of dg-$\Omega^*_{W_r(R), \gamma}$-modules for each $r$, such that the following diagrams commute for all $r$:
\[
\begin{tikzcd}
\dR(\mathcal E(W_r(R), \gamma)) \arrow[d, two heads] \arrow[r, "\iota^*_r"] & \W_r M^* \arrow[d, two heads] \\
\dR(\mathcal E(W_{r-1}(R), \gamma)) \arrow[r, "\iota^*_{r-1}"]              & \W_{r-1} M^*                 
\end{tikzcd}
\]
and
\[
\begin{tikzcd}
\mathcal E(W_r(R)) \arrow[d, "F"'] \arrow[r, "\iota^0_r"] & \W_r M^0 \arrow[d, "F"] \\
\mathcal E(W_{r-1}(R)) \arrow[r, "\iota^0_{r-1}"]         & \W_{r-1} M^0           
\end{tikzcd}
\]
A morphism of de Rham--Witt modules over $(R, \mathcal E)$ is a morphism $f \colon M^* \to N^*$ of strict $\W\Omega^*_R$-modules such that $\iota^*_{r, N} = \W_r(f) \circ \iota^*_{r, M}$ for each $r$.
\end{altdef}

\begin{mydef} \label{def:saturated_dRW}
We call $M^*$ a \emph{saturated de Rham--Witt complex associated to $\mathcal E$ over $R$} if it is initial among de Rham--Witt modules over $(R, \mathcal E)$.  Such an object is unique up to unique isomorphism if it exists; we will denote it $\W\Omega^*_{R, \mathcal E}$.
\end{mydef}

\begin{remark} \label{rmk:no_obvious_initial_dRWM}
Unfortunately, there is in general no ``obvious'' construction of an initial de Rham--Witt module.  A natural attempt would be to give $\limarrow_r \dR(\mathcal E(W_r(R), \gamma))$ the structure of a Dieudonn\'e complex and take its strictification; we are unable to carry this out in light of Remark \ref{rmk:PD_dR_over_Witt_vectors}.  In lieu of such a direct construction, we will first give a construction in the case of the trivial $F$-crystal $\mathcal E = \mathcal O$, and later construct $\W\Omega^*_{R, \mathcal E}$ in general by reducing to the lifted situation, which we study in chapter \ref{ch:dRWLM}.
\end{remark}

\begin{remark} \label{rmk:dRWM_in_nonnegative_degrees}
All of the de Rham--Witt modules we are interested in will be concentrated in nonnegative degrees.  This property (or more generally the vanishing of $M^{-1}$ for a strict Dieudonn\'e complex $M^*$) allows us to formulate another slight variant of Definition \ref{def:dRW_module}.  Namely, if $M^{-1} = 0$, then we have $\W_r M^0 = M^0/\im V^r$, so $F \colon M^0 \to M^0$ induces an endomorphism of $\W_r M^0$ rather than just a map $\W_r M^0 \to \W_{r-1} M^0$.  Thus it makes sense to replace condition (2) of Definition \ref{def:dRW_module} with the commutativity of the diagram
\begin{equation} \label{eqn:dRWM_alternative_F_compatibility}
\begin{tikzcd}
\mathcal E(W_r(R)) \arrow[r, "\iota_r"] \arrow[d, "F"'] & \W_r M^0 \arrow[d, "F"] \\
\mathcal E(W_r(R)) \arrow[r, "\iota_r"]                & \W_r M^0,
\end{tikzcd}
\end{equation}
where the left vertical map is again given by \eqref{eqn:F_on_sections}.
One can easily check that for any pair $(M^*, (\iota_r^*)_r)$ such that $M^{-1} = 0$ and the squares \eqref{eqn:dRWM_R_compatibility} commute, the commutativity of \eqref{eqn:dRWM_F_compatibility} for all $r$ is equivalent to the commutativity of \eqref{eqn:dRWM_alternative_F_compatibility} for all $r$.
\end{remark}


\subsection{Example:  the trivial $F$-crystal}
\label{sec:trivial_F-crystal}

\begin{para}
We will first construct $\W\Omega^*_{R, \mathcal E}$ in the case of the trivial $F$-crystal $(\mathcal O_{X/S}, F)$
on an arbitrary affine $k$-scheme $X = \Spec R$.  This will of course recover the saturated de Rham--Witt complex $\W\Omega^*_R$ of Bhatt--Lurie--Mathew.  The point is that we have replaced the algebra structure in the universal property of \cite[Definition 4.1.1]{BLM} with a module structure.
\end{para}

\begin{construction} \label{con:dRW_with_trivial_coefficients}
The ``trivial'' de Rham--Witt module over $(R, \mathcal O_{\Spec R/W})$ is as follows.  View $\W\Omega^*_R$ as a left module over itself,
and let $\iota^*_r \colon \Omega^*_{W_r(R), \gamma} \to \W_r \Omega^0_R$ be the dg-algebra map of Lemma \ref{lem:PD_dR_to_WrOmega}, which is in particular also a dg-$\Omega^*_{W_r(R), \gamma}$-module map.
\end{construction}

\begin{proposition} \label{prop:dRW_with_trivial_coefficients}
The object $(\W\Omega^*_R, (\iota_r)_r)$ of Construction \ref{con:dRW_with_trivial_coefficients} is a saturated de Rham--Witt complex associated to the trivial $F$-crystal on $\Spec R$.

\begin{proof}
First, we check that it is a de Rham--Witt module over $(R, \mathcal O_{\Spec R/W})$, using Alternative Definition \ref{def:dRW_module_again}.  By construction, the maps $\iota^*_r$ and $\iota^*_{r-1}$ are compatible via the quotient maps.  To see that $\iota^0_r$ and $\iota^0_{r-1}$ are compatible via Frobenius, consider the diagram:

\[
\xymatrix{
W(R) \ar[rr] \ar[dr]^F \ar @{->>}[dd] & & \W\Sat(\Omega^*_{W(R_{\red})})^0 \ar @{->>}[dd] \ar[dr]^F \\
& W(R) \ar[rr] \ar @{->>}[dd] & & \W\Sat(\Omega^*_{W(R_{\red})})^0 \ar @{->>}[dd] \\
W_r(R) \ar[dr]^F \ar[rr]^{\iota^0_r} & & \W_r\Omega^0_R \ar[dr]^F &  \\
& W_{r-1} (R) \ar[rr]^{\iota^0_{r-1}} & & \W_{r-1}\Omega^0_R
}
\]
The top, left, and right faces of the cube commute, as do the front and back faces by the construction of $\iota^*_r$ and $\iota^*_{r-1}$.  Since the vertical maps are surjective, it follows that the bottom face commutes.

So $(\W\Omega^*_R, (\iota^*_r)_r)$ is indeed a de Rham--Witt module over $(R, \mathcal O_{\Spec R/W})$.  Now let $M^* = (M^*, (\iota^*_{r, M})_r)$ be another such de Rham--Witt module, we must show there is a unique map $\W\Omega^*_R \to M^*$ of de Rham--Witt modules.  Such a map is in particular a $\W\Omega^*_R$-module morphism; by Lemma \ref{lem:A_to_M} and Remark \ref{rmk:A_to_M_strict_case}, these are determined by $f(1)$, which may be any $x \in M^0$ such that $Fx = x$.

But in order for $f$ to be a morphism of de Rham--Witt modules, each $\W_r(f)$ must intertwine $\iota^*_{r, M}$ with $\iota^*_{r, \W\Omega}$.  Equivalently, $\W_r (f)$ must send $\iota^0_{r,\W\Omega}(1) = 1 \in \W_r \Omega^*_R$ to $\iota^0_{r, M}(1) \in \W_r M^0$.  So $f(1)$ must be the element $x := (\iota^0_{r, M}(1))_r \in \limarrow_r (\W_r M^0) = M^0$.  This element is indeed fixed by Frobenius, since we have $F(1) = 1$ in $W_r(R)$, and $\iota^0_{r, M}$ commutes with Frobenius.  Thus this $f$ is the unique morphism of de Rham--Witt modules $\W\Omega^*_R \to M^*$, as desired.
\end{proof}
\end{proposition}

\subsection{Functoriality of $\dRWM_{R, \mathcal E}$}
\label{sec:functoriality_of_dRWM}

\begin{para}
Before showing that our saturated de Rham--Witt complexes exist in general, we will discuss the functoriality of the category $\dRWM_{R, \mathcal E}$, first in the $F$-crystal $\mathcal E$ and then in the pair $R/k$.  Both of these functorialities will behave as the identity on the underlying strict Dieudonn\'e complexes.
\end{para}

\begin{construction} \label{construction:dRWM_functoriality_in_E}
Let $f \colon \mathcal E \to \mathcal E'$ be a morphism of unit-root $F$-crystals on $\Spec R$.  We have a functor $f^* \colon \dRWM_{R, \mathcal E'} \to \dRWM_{R, \mathcal E}$, defined as follows.

Given an object $(M^*, (\iota'^*_r)_r)$ in $\dRWM_{R, \mathcal E'}$, first notice that $M^*$ already has the structure of a $\W\Omega^*_R$-module in $\DC_{\str}$.  We endow $M^*$ with maps $\iota^*_r \colon \dR(\mathcal E(W_r(R), \gamma)) \to \W_r M^*$ defined as the composition
\[
\begin{tikzcd}
\dR(\mathcal E(W_r(R)), \gamma) \arrow[rr, "f(W_r(R)) \otimes \id"] & & \dR(\mathcal E'(W_r(R)), \gamma) \arrow[r, "\iota'^*_r"] & \W_r M^*.
\end{tikzcd}
\]
We first check that $f^*(M^*) := (M^*, (\iota^*_r)_r)$ is a de Rham--Witt module over $(R, \mathcal E)$.  For each $r$, $\iota^*_r$ is a map of dg-$\Omega^*_{W_r(R), \gamma}$-modules, because the same is true of $f(W_r(R)) \otimes \id$ (by the functoriality of $\dR(\mathcal E)$) and $\iota'^*_r$ (by assumption).  The compatibilities with quotient and Frobenius maps likewise follow from those of $f(W_r(R)) \otimes \id$ and $\iota'^*_r$.  So $(M^*, (\iota^*_r)_r)$ is a de Rham--Witt module.

Now let $M^* = (M^*, (\iota'^*_r)_r)$ and $N^* = (N^*, (\eta'^*_r)_r)$ be two de Rham--Witt modules over $(R, \mathcal E')$, and let $\iota^*_r$ and $\eta^*_r$ be the maps constructed above, so that
\begin{align*}
f^*(M^*) & = (M^*, (\iota^*_r)_r) \text{ and} \\
f^*(N^*) & = (N^*, (\eta^*_r)_r).
\end{align*}
Suppose we are given a morphism $g \colon M^* \to N^*$ in $\dRWM_{R, \mathcal E'}$.  This is by definition a map of $\W\Omega^*_R$-modules which makes the bottom triangle of the following diagram commute for each $r$:
\[
\begin{tikzcd}
                              & \dR(\mathcal E(W_r(R), \gamma)) \arrow[ldd, "\iota^*_r"'] \arrow[rdd, "\eta^*_r"] \arrow[d] &          \\
                              & \dR(\mathcal E'(W_r(R), \gamma)) \arrow[ld, "\iota'^*_r"] \arrow[rd, "\eta'^*_r"']          &          \\
\W_r M^* \arrow[rr, "\W_r g"] &                                                                                    & \W_r N^*
\end{tikzcd}
\]
This diagram lives in the category of dg-$\Omega^*_{W_r(R), \gamma}$-modules.  Since the two other small triangles commute by the construction of $\iota^*_r$ and $\eta^*_r$, it follows that the large triangle commutes, which makes the map $f^*(g) := g$ a morphism from $f^*(M^*) = (M^*, (\iota^*_r)_r)$ to $f^*(N^*) = (N^*, (\eta^*_r)_r)$ in $\dRWM_{R, \mathcal E}$.  We clearly have $f^*(\id) = \id$ and $f^*(g \circ h) = f^*(g) \circ f^*(h)$ for all $g, h$, so $f^*$ is a functor $\dRWM_{R, \mathcal E'} \to \dRWM_{R, \mathcal E}$.
\end{construction}

\begin{lemma} \label{lem:functoriality_in_E}
The construction above is functorial in $f$:  given maps $\mathcal E \overset{f}{\to} \mathcal E' \overset{g}{\to} \mathcal E''$, we have equalities of functors
\begin{align*}
(gf)^* = f^* \circ g^* \colon & \dRWM_{R, \mathcal E''} \to \dRWM_{R, \mathcal E} \text{ and} \\
\id^* = \id \colon & \dRWM_{R, \mathcal E} \to \dRWM_{R, \mathcal E}
\end{align*}
\begin{proof}
Immediate from the construction.
\end{proof}
\end{lemma}

\begin{para} \label{para:intro_functoriality_in_R/k}
Next we move on to functoriality in $R$.  In order to apply this to the Frobenius endomorphism of $R$ later (Proposition \ref{prop:alpha_F_on_WOmega_R,E}), it will be useful to give an account of functoriality in the pair $R/k$ rather than only for $k$-algebra maps $R \to R'$.  Thus we introduce the following temporary notation:  $\dRWM_{R/k, \mathcal E}$ is the category of de Rham--Witt modules over $(R, \mathcal E)$ with ground field $k$, and $\W\Omega^*_{R/k, \mathcal E}$ is its initial object (provided that this exists).
\end{para}

\begin{remark}
Note that although the ground field $k$ makes no direct appearance in Definition \ref{def:dRW_module} (thanks to our implicit use of Remark \ref{rmk:base_ring_doesn't_matter}), the choices of $R$ and $(\mathcal E, \phi_{\mathcal E})$ constrain $k$ in opposite directions:  $k$ must be small enough so that $R$ is a $k$-algebra, but large enough so that $(\mathcal E, \phi_{\mathcal E})$ is defined over $W = W(k)$.  Thus, in contrast to the situation in \cite{BLM}, choosing $k = \F_p$ uniformly would cause a loss of generality.
\end{remark}

\begin{construction} \label{construction:dRWM_functoriality_in_R}
Suppose we are given a commutative diagram
\[
\begin{tikzcd}
\Spec R' \arrow[r, "f"] \arrow[d] & \Spec R \arrow[d] \\
\Spec k' \arrow[r]                & \Spec k,
\end{tikzcd}
\]
corresponding to a ring map $f^{\sharp} \colon R \to R'$ over $k \to k'$, and let $\mathcal E$ be a unit-root $F$-crystal on $\Cris(\Spec R/W)$.  We have a functor $f_* \colon \dRWM_{R'/k', f_{\cris}^* \mathcal E} \to \dRWM_{R/k, \mathcal E}$, defined as follows.

Suppose we are given an object $M^*$ in $\dRWM_{R'/k', f_{\cris}^* \mathcal E}$.  This is by definition a strict $\W\Omega^*_{R'}$-module equipped with morphisms $\iota'^*_r \colon \dR(f_{\cris}^* \mathcal E(W_r(R'), \gamma)) \to \W_r M^*$ for each $r$, satisfying various compatibilities.  We make the Dieudonn\'e complex $M^*$ into a module over $\W\Omega^*_R$ by restricting scalars along the map $\W\Omega^*_{f^{\sharp}} \colon \W\Omega^*_R \to \W\Omega^*_{R'}$.
By Remark \ref{rmk:PD_dR_functorialities}, we have
$$
\dR(\mathcal E(W_r(R), \gamma)) \otimes_{\Omega^*_{W_r(R), \gamma}} \Omega^*_{W_r(R'), \gamma} \overset{\sim}{\to} \dR(f_{\cris}^* \mathcal E(W_r(R'), \gamma))
$$
as a dg-module.
We then define $\iota^*_r \colon \dR(\mathcal E(W_r(R), \gamma)) \to \W_r M^*$ as the composition
$$
\dR(\mathcal E(W_r(R), \gamma)) \overset{\pi^*_r}{\to} \dR(f_{\cris}^* \mathcal E(W_r(R'), \gamma)) \overset{\iota'^*_r}{\to} \W_r M^*,
$$
where $\pi^*_r$ is the base change map.  All of the compatibilities of $\iota^*_r$ follow from the corresponding compatibilities of $\iota'^*_r$ and $\pi^*_r$, so $(M^*, \iota^*_r)$ is a de Rham--Witt module over $(R, \mathcal E)$.

Finally, suppose we are given a morphism $g \colon (M^*, (\iota'^*_r)_r) \to (N^*, (\eta'^*_r)_r)$ in $\dRWM_{R'/k', f_{\cris}^* \mathcal E}$.  This is by definition a map of strict $\W\Omega^*_{R'}$-modules making the lower triangle of the diagram
\[
\begin{tikzcd}
                           &   & \dR(\mathcal E(W_r(R), \gamma)) \arrow[lldd, "\iota^*_r"'] \arrow[d, "\pi^*_r"] \arrow[rrdd, "\eta^*_r"] &  &        \\
                           &   & \dR(f_{\cris}^* \mathcal E(W_r(R'), \gamma)) \arrow[lld, "\iota'^*_r"] \arrow[rrd, "\eta'^*_r"']                     &   &       \\
\W_r M^* \arrow[rrrr, "\W_r g"] &   &                                                                                    &         & \W_r N^*
\end{tikzcd}
\]
commute.  But as in Construction \ref{construction:dRWM_functoriality_in_E}, commutativity of the lower triangle implies that of the outer triangle, and a map of strict $\W\Omega^*_{R'}$-modules is in particular a map of strict $\W\Omega^*_R$-modules via $\W\Omega^*_R \to \W\Omega^*_{R'}$.  This makes $f_*(g) := g$ a morphism in $\dRWM_{R/k, \mathcal E}$.  As before, the identities $f_*(\id) = \id$ and $f_*(g \circ h) = f_*(g) \circ f_*(h)$ are immediate, making $f_* \colon \dRWM_{R'/k', f_{\cris}^* \mathcal E} \to \dRWM_{R/k, \mathcal E}$ a functor.
\end{construction}

\begin{remark}
Analogously to Lemma \ref{lem:functoriality_in_E}, one can easily show that the functors of Construction \ref{construction:dRWM_functoriality_in_R} behave as expected when we are given a composition of maps $\Spec R'' \overset{f}{\to} \Spec R' \overset{g}{\to} \Spec R$ of affine schemes over $\Spec k'' \to \Spec k' \to \Spec k$.
A similarly routine verification shows that the functorialities of $\dRWM_{R/k, \mathcal E}$ in $\mathcal E$ and $R/k$ commute with each other.
\end{remark}

\begin{remark}
Examining Construction \ref{construction:dRWM_functoriality_in_R} reveals that the category $\dRWM_{R/k, \mathcal E}$ is independent of $k$ in the following sense:  whenever $f^{\sharp} \colon R \to R'$ is an isomorphism,
the functor $f_* \colon \dRWM_{R'/k', f_{\cris}^* \mathcal E} \to \dRWM_{R/k, \mathcal E}$ is an isomorphism of categories, regardless of whether $k \to k'$ is an isomorphism.
Thus, at this point we will resume our previous convention of omitting $k$ from the notation $\dRWM_{R, \mathcal E}$.
\end{remark}

\begin{remark} \label{rmk:functorialities_of_WOmega}
The functorialities of the category $\dRWM_{R, \mathcal E}$ naturally lead to functorialities of its initial object, provided that the initial object exists.
Namely, suppose we are given a map $f \colon \Spec R' \to \Spec R$ of affine schemes over $\Spec k' \to \Spec k$ and a map $g \colon \mathcal E \to \mathcal E'$ of unit-root $F$-crystals on $\Cris(\Spec R/W)$.  Then assuming the relevant saturated de Rham--Witt complexes exist, there are unique maps in $\dRWM_{R, \mathcal E}$
$$
\W\Omega^*_{R, \mathcal E} \to g^*(\W\Omega^*_{R, \mathcal E'})
\text{\quad and \quad}
\W\Omega^*_{R, \mathcal E} \to f_*(\W\Omega^*_{R', f^*_{\cris} \mathcal E}),
$$
which compose as expected if we are instead given maps
$$
\mathcal E \overset{g}{\to} \mathcal E' \overset{g'}{\to} \mathcal E''
\text{\quad or \quad}
\Spec R'' \overset{f'}{\to} \Spec R' \overset{f}{\to} \Spec R.
$$
Since $g^*$ and $f_*$ act as the identity on the underlying strict Dieudonn\'e complexes, we may regard these as maps
$$
\W\Omega^*_{R, \mathcal E} \to \W\Omega^*_{R, \mathcal E'}
\text{\quad and \quad}
\W\Omega^*_{R, \mathcal E} \to \W\Omega^*_{R', f^*_{\cris} \mathcal E}
$$
in $\DC_{\str}$.
\end{remark}

\begin{remark} \label{rmk:alpha_F_is_a_DC_map}
We will finish this section by computing the functoriality map of Remark \ref{rmk:functorialities_of_WOmega} explicitly in the case of the Frobenius morphism.  Before stating this, recall that for any Dieudonn\'e complex $M^*$, the map
$$
\alpha_F \colon M^* \to M^*
$$
defined by $p^i F$ in degree $i$ is a morphism of Dieudonn\'e complexes, thanks to the calculations
\begin{align*}
F(\alpha_F(x)) & = p^i F^2(x) = \alpha_F(Fx) \text{ and} \\
d(\alpha_F(x)) & = p^i dF(x) = p^{i+1} Fdx = \alpha_F(dx)
\end{align*}
for $x$ homogeneous of degree $i$.  In particular, if $M^*$ is saturated, it makes sense to write down the maps $\W_r(\alpha_F) \colon \W_r M^* \to \W_r M^*$.

We give the result first in the case of trivial coefficients, and then for a general $\mathcal E$:
\end{remark}

\begin{lemma} \label{lem:alpha_F_on_WOmega_R}
For any $\F_p$-algebra $R$, the endomorphism
$$
\W\Omega^*_{F_R} \colon \W\Omega^*_R \to \W\Omega^*_R
$$
induced by the absolute Frobenius $F_R \colon R \to R$ coincides with $\alpha_F$.
\end{lemma}
\begin{proof}
First note that $\alpha_F$ is a map of Dieudonn\'e algebras, by Remark \ref{rmk:alpha_F_is_a_DC_map} and multiplicativity.
Now, unwinding the universal property of \cite[Definition 4.1.1]{BLM}, the claim says that the diagram
\[
\begin{tikzcd}
R \arrow[r] \arrow[d, "F_R"'] & \W_1 \Omega^0_R \arrow[d, "\W_1(\alpha_F)"]         \\
R \arrow[r] & \W_1 \Omega^0_R
\end{tikzcd}
\]
commutes.  But in fact it follows from the definition of Dieudonn\'e algebras that $\W_1(\alpha_F) = F \colon \W_1 \Omega^0_R \to \W_1 \Omega^0_R$ is the Frobenius endomorphism of this $\F_p$-algebra, so the diagram indeed commutes.
\end{proof}

\begin{proposition} \label{prop:alpha_F_on_WOmega_R,E}
Let $R$ be a $k$-algebra and $\mathcal E$ a unit-root $F$-crystal on $\Cris(\Spec R/W)$.  Consider the object $(F_{\Spec R})_* \phi_{\mathcal E}^* \W\Omega^*_{R, \mathcal E}$ of $\dRWM_{R, \mathcal E}$ obtained from Constructions \ref{construction:dRWM_functoriality_in_E} and \ref{construction:dRWM_functoriality_in_R}, and recall that we have
$$
(F_{\Spec R})_* \phi_{\mathcal E}^* \W\Omega^*_{R, \mathcal E} = \W\Omega^*_{R, \mathcal E}
$$
as a strict Dieudonn\'e complex.  Then the unique map
$$
\W\Omega^*_{R, \mathcal E} \to (F_{\Spec R})_* \phi_{\mathcal E}^* \W\Omega^*_{R, \mathcal E}
$$
in $\dRWM_{R, \mathcal E}$ coincides with the map $\alpha_F \colon \W\Omega^*_{R, \mathcal E} \to \W\Omega^*_{R, \mathcal E}$ of Dieudonn\'e complexes.
\end{proposition}
\begin{proof}
For convenience, we will abbreviate the absolute Frobenius morphism $F_{\Spec R}$ as $f$.  To show that $\alpha_F$ is a map of de Rham--Witt modules, we must show that it is compatible with module structures and $\iota$ maps.  Unraveling the given functors
$$
\dRWM_{R, \mathcal E} \overset{\phi_{\mathcal E}^*}{\to} \dRWM_{R, f^*_{\cris} \mathcal E} \overset{f_*}{\to} \dRWM_{R, \mathcal E},
$$
the necessary compatibility with module structures is that for all $x \in \W\Omega^*_R$ and $m \in \W\Omega^*_{R, \mathcal E}$, we have
$$
\alpha_F(x \cdot m) = \W\Omega^*_{F_R}(x) \cdot \alpha_F(m).
$$
Taking $x$ and $m$ homogeneous of degree $i$ and $j$ respectively, this equation simplifies by Lemma \ref{lem:alpha_F_on_WOmega_R} to the known identity
$$
p^{i+j} F(x \cdot m) = p^i F(x) \cdot p^j F(m).
$$
As for compatibility with $\iota$ maps, the de Rham--Witt module $f_* \phi_{\mathcal E}^* \W\Omega^*_{R, \mathcal E}$ is equipped by construction with the maps
$$
\iota_r \colon \mathcal E(W_r(R)) \overset{W_r(f)^*}{\longrightarrow} f^*_{\cris} \mathcal E(W_r(R)) \overset{\phi_{\mathcal E}}{\to} \mathcal E(W_r(R)) \overset{\iota_{r, \W\Omega}}{\longrightarrow} \W_r \Omega^0_{R, \mathcal E},
$$
where $\iota_{r, \W\Omega}$ is the map which the de Rham--Witt module $\W\Omega^*_{R, \mathcal E}$ comes equipped with, and the composition of the first two maps is the semilinear endomorphism
$$
F \colon \mathcal E(W_r(R)) \to \mathcal E(W_r(R))
$$
of \eqref{eqn:dRWM_F_compatibility}.  (Note in particular that $W_r(f)$ agrees with the Witt vector Frobenius $F \colon W_r(R) \to W_r(R)$.)
Thus, to prove that $\alpha_F$ ($= F$ in degree $0$) is compatible with the $\iota$ maps, it suffices to prove that the square
\[
\begin{tikzcd}
\mathcal E(W_r(R)) \arrow[r, "\iota_{r, \W\Omega}"] \arrow[d, "F"] & \W_r \Omega^0_{R, \mathcal E} \arrow[d, "F"] \\
\mathcal E(W_r(R)) \arrow[r, "\iota_{r, \W\Omega}"]                & \W_r \Omega^0_{R, \mathcal E}
\end{tikzcd}
\]
commutes.  But this square is exactly \eqref{eqn:dRWM_alternative_F_compatibility}, and its commutativity follows from Remark \ref{rmk:dRWM_in_nonnegative_degrees} since we know that $\W\Omega^*_{R, \mathcal E}$ is a de Rham--Witt module concentrated in nonnegative degrees.
\end{proof}


\subsection{Insensitivity to nilpotent thickenings}
\label{sec:nilpotent_thickenings}

Throughout this section, we fix the following setup:

\begin{situation} \label{sit:reduction}
Let $R$ be a $k$-algebra and $\mathcal E$ a unit-root $F$-crystal on $\Cris(\Spec R/W)$.  Let $f \colon \Spec R_{\red} \hookrightarrow \Spec R$ be the closed embedding over $k$ corresponding to the reduction map $f^{\sharp} \colon R \to R_{\red}$.
\end{situation}

\begin{para}
Recall from Remark \ref{rmk:BLM_nilpotent_invariance} that the saturated de Rham--Witt complex $\W\Omega^*_R$ of Bhatt--Lurie--Mathew is insensitive to nilpotent thickenings.  Our present goal is to generalize this statement to the case of unit-root coefficients.  The main result is as follows.
\end{para}

\begin{proposition} \label{prop:dRWM_equiv_reduction}
The functor $f_* \colon \dRWM_{R_{\red}, f^*_{\cris} \mathcal E} \to \dRWM_{R, \mathcal E}$ is an equivalence of categories.
\end{proposition}

Before we prove this, we will need a few lemmas on the behavior of the pullback map
$$
\mathcal E_{W_r(R)} \to W_r(f)^*(\mathcal E_{W_r(R)}) = (f^*_{\cris} \mathcal E)_{W_r(R_{\red})}.
$$

\begin{lemma} \label{lem:p^N-torsion_presheafy}
Fix $r > 0$.  For every $x \in \ker(\mathcal E(W_r(R)) \to (f_{\cris}^* \mathcal E)(W_r(R_{\red})))$, there exists an $N \gg 0$ such that $x$ lies in the image of the $p^N$-torsion in $\mathcal E(W_n(R))$ for all $n \geq r$.
\end{lemma}

The proof requires a technical digression on locally free crystals, so we split it off into \textsection \ref{sec:proof_of_lemma}.

\begin{lemma} \label{lem:reduction_iota_factors}
In Situation \ref{sit:reduction}, suppose $(M^*, (\iota^*_r)_r)$ is a de Rham--Witt module over $(R, \mathcal E)$.  Then:
\begin{enumerate}
\item Each map $\iota_r = \iota^0_r \colon \mathcal E(W_r(R)) \to \W_r M^0$ factors through the surjection
$$\mathcal E(W_r(R)) \twoheadrightarrow \mathcal E(W_r(R)) \otimes_{W_r(R)} W_r(R_{\red}) \simeq (f_{\cris}^* \mathcal E)(W_r(R_{\red})).$$
\item Each map $\iota^*_r \colon \dR(\mathcal E(W_r(R), \gamma)) \to \W_r M^*$ factors through the surjection
$$\dR(\mathcal E(W_r(R), \gamma)) \twoheadrightarrow \dR(\mathcal E(W_r(R), \gamma)) \otimes_{\Omega^*_{W_r(R), \gamma}} \Omega^*_{W_r(R_{\red}), \gamma} \simeq \dR(f^*_{\cris} \mathcal E(W_r(R_{\red}), \gamma)).$$
\end{enumerate}
\end{lemma}
\begin{proof}
For (1), let $x$ lie in the kernel of $\mathcal E(W_r(R)) \twoheadrightarrow (f_{\cris}^* \mathcal E)(W_r(R_{\red}))$.  We must show that $\iota_r(x) = 0$.  By Lemma \ref{lem:p^N-torsion_presheafy}, we can lift $x$ to a $p^N$-torsion element $\widetilde{x} \in \mathcal E(W_{r+N}(R))$.  Then we have $\iota_{r+N}(\widetilde{x}) \in \W_{r+N} M^0[p^N]$.  By axiom 7 of Dieudonn\'e towers (\cite[Definition 2.6.1]{BLM}), this must reduce to $0$ in $\W_r M^0$.  So we have $\iota_r(x) = 0$.

For (2), first observe by part (1) that $\iota_r^0$ annihilates
$$\ker(W_r(R) \to W_r(R_{\red})) \cdot \mathcal E(W_r(R)).$$
It follows that $\iota_r^*$ annihilates
$$\ker(W_r(R) \to W_r(R_{\red})) \cdot \dR(\mathcal E(W_r(R), \gamma)),$$
since $\mathcal E(W_r(R))$ generates its PD-de Rham complex as a module over $\Omega^*_{W_r(R), \gamma}$.  Therefore $\iota_r^*$ factors through
$$\dR(\mathcal E(W_r(R), \gamma)) \otimes_{\Omega^*_{W_r(R), \gamma}} \Omega^*_{W_r(R), \gamma}/K,$$
where $K$ is the dg-ideal generated by $\ker(W_r(R) \to W_r(R_{\red}))$ in degree 0.  But in view of Lemma \ref{lem:quotient_of_PD_dR_algebras} and Remark \ref{rmk:PD_dR_functorialities}, this quotient agrees with $\dR(f^*_{\cris} \mathcal E(W_r(R_{\red}), \gamma))$.
\end{proof}

\begin{proof}[Proof of Proposition \ref{prop:dRWM_equiv_reduction}]
We must show that $f_* \colon \dRWM_{R_{\red}, f^*_{\cris} \mathcal E} \to \dRWM_{R, \mathcal E}$ is essentially surjective and fully faithful.  To prove it is essentially surjective, let $(M^*, (\iota^*_r)_r)$ be a de Rham--Witt module over $(R, \mathcal E)$; we must show this is isomorphic to an object in the image of $f_*$.  
To construct such an object, we keep the same underlying strict Dieudonn\'e complex $M^*$ (viewed as a $\W\Omega^*_{R_{\red}}$-module via the isomorphism $\W\Omega^*_R \to \W\Omega^*_{R_{\red}}$ of Remark \ref{rmk:BLM_nilpotent_invariance}) and factor the maps $\iota^*_r$ through maps $\iota'^*_r \colon \dR(\mathcal E(W_r(R), \gamma)) \to \W_r M^*$ using Lemma \ref{lem:reduction_iota_factors}.

Each $\iota'^*_r$ is a map of dg-$\Omega^*_{W_r(R_{\red}), \gamma}$-modules by construction, and the necessary compatibilities with Frobenius (in degree 0) and quotient maps follow from the corresponding compatibilities of $\iota^*_r$ and $\dR(\mathcal E(W_r(R), \gamma)) \twoheadrightarrow \dR(f^*_{\cris} \mathcal E(W_r(R_{\red}), \gamma))$.  So the strict $\W\Omega^*_{R_{\red}}$-module $M^*$, equipped with the maps $\iota'^*_r$, is a de Rham--Witt module over $(R_{\red}, f_{\cris}^* \mathcal E)$.  We have $f_*(M^*, (\iota'^*_r)_r) = (M^*, (\iota^*_r)_r)$ by construction.  This proves essential surjectivity.

To prove that $f_*$ is fully faithful, let $M^* = (M^*, (\iota'^*_r)_r)$ and $N^* = (N^*, (\eta'^*_r)_r)$ be two objects in $\dRWM_{R_{\red}, f_{\cris}^* \mathcal E}$.  Let $\iota^*_r$ and $\eta^*_r$ be the maps of Construction \ref{construction:dRWM_functoriality_in_R}; that is,
\begin{align*}
f_* M^* & = (M^*, (\iota^*_r)_r) \text{ and} \\
f_* N^* & = (N^*, (\eta^*_r)_r)
\end{align*}
in $\dRWM_{R, \mathcal E}$.  Recall that a homomorphism $g \colon M^* \to N^*$ in $\dRWM_{R_{\red}, f_{\cris}^* \mathcal E}$ is a map of $\W\Omega^*_{R_{\red}}$-modules such that bottom triangle in the diagram
\[
\begin{tikzcd}
                              & \dR(\mathcal E(W_r(R), \gamma)) \arrow[ldd, "\iota^*_r"', bend right] \arrow[d, "\pi^*_r"] \arrow[rdd, "\eta^*_r", bend left] &          \\
                              & \dR(f^*_{\cris} \mathcal E(W_r(R_{\red}), \gamma)) \arrow[ld, "\iota'^*_r"] \arrow[rd, "\eta'^*_r"']               &          \\
\W_r M^* \arrow[rr, "\W_r g"] &  & \W_r N^*
\end{tikzcd}
\]
commutes for all $r$.  A homomorphism $g \colon f_* M^* \to f_* N^*$ in $\dRWM_{R, \mathcal E}$ is a map of $\W\Omega^*_R$-modules such that the outer region commutes.  But $\pi^*_r$ is surjective and the other two inner regions commute, so the inner triangle commutes if and only if the outer one does.  Since we also know that $\W\Omega^*_R \to \W\Omega^*_{R_{\red}}$ is an isomorphism, it follows that
$$
\Hom_{\dRWM_{R_{\red}, f_{\cris}^* \mathcal E}}(M^*, N^*) \overset{\sim}{\to} \Hom_{\dRWM_{R, \mathcal E}}(f_* M^*, f_* N^*),
$$
as desired.  This completes the proof.
\end{proof}

\begin{corollary} \label{cor:R_Rred_same_dRW}
In Situation \ref{sit:reduction}, suppose that either $(R_{\red}, f_{\cris}^* \mathcal E)$ or $(R, \mathcal E)$ admits a saturated de Rham--Witt complex.  Then both pairs admit saturated de Rham--Witt complexes, and the equivalence $f_* \colon \dRWM_{R_{\red}, f_{\cris}^* \mathcal E} \to \dRWM_{R, \mathcal E}$ sends one to the other.
\begin{proof}
This follows immediately from Proposition \ref{prop:dRWM_equiv_reduction} and the definition of saturated de Rham--Witt complexes as initial de Rham--Witt modules.
\end{proof}
\end{corollary}

\begin{remark}
Since the functor $f_*$ preserves the underlying Dieudonn\'e complex and its module structure over $\W\Omega^*_R \simeq \W\Omega^*_{R_{\red}}$, and the maps $\iota^*_r$ and $\iota'^*_r$ are compatible, we can regard the two saturated de Rham--Witt complexes in Corollary \ref{cor:R_Rred_same_dRW} as being ``the same''.
\end{remark}


\subsection{Digression:  the proof of Lemma \ref{lem:p^N-torsion_presheafy}}
\label{sec:proof_of_lemma}

To prove Lemma \ref{lem:p^N-torsion_presheafy}, we begin with the following technical lemma on locally free crystals:

\begin{lemma} \label{lem:locally_free_tower_nakayama}
Let $X$ be a $k$-scheme and $\mathcal E$ a finite locally free crystal of $\mathcal O_{X/W}$-modules.  Then the system $(\mathcal E_{W_r(X)})_r$ of Zariski sheaves on $W_r(X)$ is finite locally free, in the following sense:  there exist an affine open cover $(\Spec R_i)_i$ of $X$ (which we identify with $\Spec W_r R_i \subset W_r(X)$ as topological spaces), integers $n_i$, and isomorphisms $\mathcal E_{W_r(R_i)} \simeq \mathcal O^{\oplus n_i}$ for all $r$ and $i$, compatible with the quotient maps from $W_r(R_i)$ to $W_{r-1}(R_i)$.
\begin{proof}
Choose an affine open cover $(\Spec R_i)_i \subset X$ trivializing the Zariski sheaf $\mathcal E_X$.  For each $i$, choose generators $x_1, \dots, x_{n_i}$ for the finite free $R_i$-module $\mathcal E(R_i)$.  Lift each $x_j$ arbitrarily to a compatible family of sections $\widetilde x_j \in \mathcal E(W_r R_i)$ for each $r$; this is possible because $\mathcal E(W_r R_i) \simeq \mathcal E(W_{r+1} R_i) \otimes_{W_{r+1} R_i} W_r R_i$ and the map $W_{r+1} R_i \to W_r R_i$ is surjective.

Since the kernel $VW_{r-1}(R_i)$ of $W_r(R_i) \to R_i$ is nilpotent, it is contained in the Jacobson radical of $W_r(R_i)$.  Thus Nakayama's lemma implies that for each $r$, the elements $\widetilde x_j$ generate $\mathcal E(W_r R_i)$.  It follows that the $\widetilde x_j$ are linearly independent, as $\mathcal E(W_r R_i)$ is a locally free $W_r(R_i)$-module of rank $n_i$ and thus cannot be generated by fewer than $n_i$ elements.  Thus the elements $\widetilde x_j$ form a basis of the $W_r(R_i)$-module $\mathcal E(W_r R_i)$.  It follows by quasicoherence that the $\widetilde x_j$ identify $\mathcal E_{\Spec W_r R_i}$ with $\mathcal O^{\oplus n_i}$.
\end{proof}
\end{lemma}

The next lemma may be viewed as a ``sheafy version'' of our goal.

\begin{lemma} \label{lem:p^N-torsion_sheafy}
Fix some $r > 0$.  For every local section $x$ of $\ker(\mathcal E_{W_r(R)} \to (f_{\cris}^* \mathcal E)_{W_r(R_{\red})})$, there exists an $N \gg 0$ and an open covering $\{U_i\}_i$ of $\Spec R$ such that $x|_{W_r(U_i)}$ lies in the image of the $p^N$-torsion in $\mathcal E_{W_n(U_i)}$ for all $n \geq r$.
\end{lemma}
\begin{proof}
As the statement is local on $\Spec W_r(R)$, we may assume by Lemma \ref{lem:locally_free_tower_nakayama} that the entire tower of Zariski sheaves $\mathcal E_{W_{\bullet}(R)}$ is isomorphic to $\mathcal O_{\Spec W_{\bullet}(R)}^{\oplus m}$, and that $x$ is a globally defined section of $\mathcal E_{W_r(R)}$; that is, an element of $W_r(R)^{\oplus m}$.  We will moreover assume that $m=1$; the general case follows by repeating the same construction in each coordinate.

With this setup, all we must prove is the following:  for every $x \in \ker(W_r(R) \to W_r(R_{\red}))$, there exists an $N \gg 0$ such that $x$ lies in the image of the $p^N$-torsion in $W_n(R)$ for all $n \geq r$.
To prove this, write $x$ as $(x_0, x_1, \dots, x_{r-1})$ in Witt coordinates, where each $x_i \in R$ is nilpotent.  Choose $N$ so that each $x_i^{p^N} = 0$, and lift $x$ to $(x_0, x_1, \dots, x_{r-1}, 0, 0, \dots) \in W_n(R)$ for each $n \geq r$.  This is $p^N$-torsion by the formula for multiplication by $p$ on Witt coordinates:  $p(x_0, x_1, \dots) = FV(x_0, x_1, \dots) = (0, x_0^p, x_1^p, \dots)$.
\end{proof}

\begin{proof}[Proof of Lemma \ref{lem:p^N-torsion_presheafy}]
We proved above that there exists an $N \gg 0$ such that for all $n \geq r$, the sheaf-theoretic image of the reduction map 
$$
\mathcal E_{W_n(R)}[p^N] \to \mathcal E_{W_r(R)}
$$
(both interpreted as sheaves on $\Spec W_n(R)$, say) contains the kernel of
$$
\mathcal E_{W_r(R)} \to W_r(f)^*(\mathcal E_{W_r(R)}) = (f^*_{\cris} \mathcal E)_{W_r(R_{\red})}
$$
We must now prove the analogous statement for global sections instead of sheaves, which would ordinarily be a more delicate statement.  But since we are dealing with maps of quasicoherent sheaves on an affine scheme, taking global sections is an exact functor, and thus the global sections of the sheaf-theoretic image coincide with the image of the global sections.  So the two statements are equivalent.
\end{proof}


\subsection{\'Etale localization}
\label{sec:etale_localization}

We now address the question of how our saturated de Rham--Witt modules localize along \'etale morphisms of affine schemes.  We begin with a few lemmas:

\begin{lemma} \label{lem:Wr_tensor}
Suppose $A^*$ is an algebra object in $\DC$, $M^*$ and $N^*$ are strict right and left $A^*$-modules respectively, and $P^*$ is any strict Dieudonn\'e complex equipped with a map $f \colon M^* \otimes_{A^*} N^* \to P^*$ in $\DC$.  Then the map for each $r \geq 0$, the composition
$$
f_r \colon M^* \otimes_{A^*} N^* \to P^* \twoheadrightarrow \W_r(P^*)
$$
factors uniquely through a map of complexes
$$
\W_r(M^*) \otimes_{A^*} \W_r(N^*) \to \W_r(P^*).
$$
\begin{proof}
This follows from the identities
\begin{align*}
V^r m \otimes n & = V^r(m \otimes F^r n), \\
dV^r m \otimes n & = d(V^r m \otimes n) - (-1)^{|m|} V^r b \otimes dn \\
& = dV^r(m \otimes F^r n) - (-1)^{|m|} V^r(m \otimes F^r d n),
\end{align*}
and similarly for $m \otimes V^r n$ and $m \otimes dV^r n$.  (Recall that one can verify the first identity by applying the injective map $F^r$ to both sides.)
\end{proof}
\end{lemma}

\begin{remark}
If $A^*$ is also strict, then we have
$$
\W_r(M^*) \otimes_{A^*} \W_r(N^*) = \W_r(M^*) \otimes_{\W_r(A^*)} \W_r(N^*)
$$
in light of Lemma \ref{lem:Vr+dVr_submodule} and the surjectivity of $A^* \to \W_r A^*$, so we may view Lemma \ref{lem:Wr_tensor} as providing a map
$$
\W_r(M^*) \otimes_{\W_r(A^*)} \W_r(N^*) \to \W_r(P^*).
$$
\end{remark}

The main result of this section is as follows:

\begin{proposition} \label{prop:dRW_is_qcoh}
Let $X = \Spec R$ be an affine $k$-scheme, $f \colon \Spec S \to X$ an \'etale neighborhood, and $\mathcal E$ a unit-root $F$-crystal on $\Cris(X/W)$.  Assume $(R, \mathcal E)$ admits a saturated de Rham--Witt complex $\W\Omega^*_{R, \mathcal E}$.  Then:
\begin{enumerate}
\item The pair $(S, f^*_{\cris} \mathcal E)$ also admits a saturated de Rham--Witt complex $\W\Omega^*_{S, f^*_{\cris} \mathcal E}$.
\item For each $r > 0$, $S \mapsto \W_r \Omega^*_{S, f^*_{\cris} \mathcal E}$ defines a quasicoherent sheaf of $W_r \mathcal O$-modules on the affine \'etale site $\AffEt(X)$.
\end{enumerate}
\end{proposition}

\begin{remark} \label{rmk:etale_base_change_of_etesse}
The analogue of Proposition \ref{prop:dRW_is_qcoh} for the classical de Rham--Witt complex with coefficients (\cite[\textsection II, D\'efinition 1.1.7]{etesse}) is much easier to prove:  the natural morphisms
\begin{align*}
\Omega^*_{W_r(R), \gamma} & \to \Omega^*_{W_r(S), \gamma}, \\
\dR(\mathcal E(W_r(R), \gamma)) & \to \dR(f^*_{\cris}(\mathcal E)(W_r(S), \gamma)), \text{ and} \\
W_r \Omega^*_R & \to W_r \Omega^*_S
\end{align*}
are all base-change maps along $W_r(R) \to W_r(S)$ by Lemma \ref{lem:PD_dR_etale_base_change} and \cite[I, Proposition 1.14]{illusie}, thus so is
$$
\dR(\mathcal E(W_r(R), \gamma)) \otimes_{\Omega^*_{W_r(R), \gamma}} W_r \Omega^*_R \to \dR(f^*_{\cris}(\mathcal E)(W_r(S), \gamma)) \otimes_{\Omega^*_{W_r(S), \gamma}} W_r \Omega^*_S.
$$
\end{remark}

\begin{remark}
The idea of the proof of Proposition \ref{prop:dRW_is_qcoh} will be that we have isomorphisms
\begin{align*}
\W \Omega^*_{S, f^*_{\cris} \mathcal E} & \simeq \W\Omega^*_S \otimes^{\str}_{\W\Omega^*_R} \W\Omega^*_{R, \mathcal E}, \text{ and} \\
\W_r \Omega^*_{S, f^*_{\cris} \mathcal E} & \simeq W_r (S) \otimes_{W_r (R)} \W_r\Omega^*_{R, \mathcal E} \text{ for each } r,
\end{align*}
where the latter visibly defines the quasicoherent sheaf $\widetilde{\W_r\Omega^*_{R, \mathcal E}}$.

In order to pass between the two isomorphisms above, we must show that the strictified tensor product above can be computed at finite levels.  In fact, since writing down a de Rham--Witt module requires an understanding of both its underlying strict Dieudonn\'e complex $M^*$ and the corresponding strict Dieudonn\'e tower $(\W_r M^*)_r$, we will need this comparison (Proposition \ref{V_etale}) before we can prove either isomorphism above.
The proof of this proposition will take up much of this section; it takes place in the setting of a $V$-adically \'etale morphism of Dieudonn\'e algebras (cf. \cite[\textsection 5.3]{BLM})---of which the prototypical example is $\W\Omega^*_R \to \W\Omega^*_S$ where $R \to S$ is \'etale.
\end{remark}

\begin{para}
Suppose $A^* \to B^*$ is a $V$-adically \'etale morphism of strict Dieudonn\'e algebras,
suppose $M^*$ is a strict $A^*$-module, and set $N_r^* = \W_r B^* \otimes_{\W_r A^*} \W_r M^*$.
In the next few results, we will give this the structure of a Dieudonn\'e tower, and then show that its limit computes the strictified tensor product $B^* \otimes^{\str}_{A^*} M^*$.
Note first that we have natural isomorphisms of graded modules
\begin{align*}
N_r^* = \W_r B^* \otimes_{\W_r A^*} \W_r M^* & \simeq \W_r B^0 \otimes_{\W_r A^0} \W_r M^* \\
& \simeq W_r(B^0/VB^0) \otimes_{W_r(A^0/VA^0)} \W_r M^* \\
& \simeq W(B^0/VB^0) \otimes_{W(A^0/VA^0)} \W_r M^*,
\end{align*}
where the first step is the statement that $A^* \to B^*$ is $V$-adically \'etale, and the second step follows by \cite[Proposition 3.6.2]{BLM}.
\end{para}

\begin{construction} \label{con:strict_tower}
Suppose $A^* \to B^*$ is a $V$-adically \'etale morphism of strict Dieudonn\'e algebras,
and suppose $M^*$ is a strict $A^*$-module.  Then we can endow the family of complexes $(\W_r B^* \otimes_{\W_r A^*} \W_r M^*)_r$ with maps $R, F$, and $V$ making it into a strict Dieudonn\'e tower.
For convenience, we set $N_r^* = \W_r B^* \otimes_{\W_r A^*} \W_r M^*$.  First observe that we have natural isomorphisms of graded modules
\begin{align*}
N_r^* = \W_r B^* \otimes_{\W_r A^*} \W_r M^* & \simeq \W_r B^0 \otimes_{\W_r A^0} \W_r M^* \\
& \simeq W_r(B^0/VB^0) \otimes_{W_r(A^0/VA^0)} \W_r M^* \\
& \simeq W(B^0/VB^0) \otimes_{W(A^0/VA^0)} \W_r M^*,
\end{align*}
where the first step is the statement that $A^* \to B^*$ is $V$-adically \'etale, and the second step follows by \cite[Proposition 3.6.2]{BLM}.
In particular, \cite[Corollary 5.4.9]{BLM} implies that the maps
\begin{align*}
& R \colon \W_{r+1} M^* \to \W_r M^* \to N_r^*, \\
& F \colon \W_{r+1} M^* \to \W_r M^* \to N_r^*, \text{ and} \\
& V \colon \W_r M^* \to \W_{r+1} M^* \to N_{r+1}^*
\end{align*}
each extend uniquely to maps
\begin{align*}
& R \colon N_{r+1}^* \to N_r^*, \\
& F \colon N_{r+1}^* \to N_r^*, \\
& V \colon N_r^* \to N_{r+1}^*
\end{align*}
respectively.
Note that the uniqueness of the extensions implies that we have identities
\begin{align*}
R(b \otimes m) & = Rb \otimes Rm, \\
F(b \otimes m) & = Fb \otimes Fm, \text{ and} \\
V(Fb \otimes m) & = b \otimes V m
\end{align*}
for $b \in \W_r B^0$ and $m \in \W_r M^*$, since these are true for $b$ in the image of $\W_r A^0 \to \W_r B^0$.
\end{construction}

\begin{proposition} \label{prop:strict_tower}
The maps $R$, $F$, and $V$ of Construction \ref{con:strict_tower} endow the system $(N_r^*)_r$ with the structure of a strict Dieudonn\'e tower.

\begin{proof}
Properties (1), (3), (4), and (5) of \cite[Definition 2.6.1]{BLM} follow from the corresponding properties of $(M_r^*)_r$ and the uniqueness of the extensions of \cite[Corollary 5.4.9]{BLM}.  Properties (2) and (7) follow from base-changing the corresponding properties of $(M_r^*)_r$ along the \'etale (thus flat) map $\W_r A^0 \to \W_r B^0$.
Properties (6) and (8) follow from a similar argument if we are careful about semilinearity.
Namely, twisting the respective exact sequences for $(M^*_r)_r$ by various powers of Frobenius gives exact sequences of $\W_{r+1} A^0$-modules
$$
M_{r+1}^* \overset{F}{\to} F^*(M_r^*) \overset{d}{\to} F^*(M_r^{* + 1}) / p
$$
and
$$
(F^{r+1})^*(M_1^* \oplus M_1^{* - 1}) \overset{(V^r, dV^r)}{\longrightarrow} F^*(M_{r+1}^*) \overset{R}{\to} F^*(M_r^*),
$$
where the linearity of $d$ and $dV^r$ follows from the calculations
\begin{align*}
d(F(a) \cdot m) & = dF(a) \cdot m + F(a) \cdot dm \\
& = pFd \cdot m + F(a) \cdot dm \\
& \equiv F(a) \cdot dm \pmod p,
\end{align*}
and
\begin{align*}
dV^r(F^{r+1}(a) \cdot m) & = d(F(a) \cdot V^r m) \\
& = dF(a) \cdot V^r m + F(a) \cdot dV^r m \\
& = pFd(a) \cdot V^r m + F(a) \cdot dV^r m \\
& = F(a) \cdot dV^r m \in M^*_{r+1},
\end{align*}
each for $a \in \W_{r+1} A^0$.
The corresponding properties of $(N^*_r)_r$ follow by base-changing along the \'etale map $\W_r A^0 \to \W_r B^0$; this operation commutes with Frobenius pullbacks by \cite[Remark 5.4.2]{BLM}.
\end{proof}
\end{proposition}

We now claim that this strict Dieudonn\'e tower corresponds, under the equivalence of categories of \cite[Corollary 2.9.4]{BLM}, to the strictified tensor product $B^* \otimes_{A^*}^{\str} M^*$.

\begin{proposition} \label{V_etale} 
Suppose $A^* \to B^*$ is a $V$-adically \'etale morphism of strict Dieudonn\'e algebras, and suppose $M^*$ is a strict $A^*$-module.  Then the natural $B^*$-linear map
$$\rho' \colon B^* \otimes_{A^*} M^* \to \limarrow_r (\W_r(B^*) \otimes_{\W_r(A^*)} \W_r(M^*))$$
exhibits the latter as a strictification of the former.  In particular, this induces isomorphisms
$$\W_r(B^*) \otimes_{\W_r(A^*)} \W_r(M^*) \simeq \W_r (B^* \otimes_{A^*} ^{\str} M^*)$$
of $\W_r B^*$-modules for each $r$.
\begin{proof}
By the universal property of strictification, the given map $\rho'$ factors uniquely as
$$B^* \otimes_{A^*} M^* \overset{\rho}{\to} B^* \otimes_{A^*}^{\str} M^* \overset{g}{\to} \limarrow_r (\W_r(B^*) \otimes_{\W_r(A^*)} \W_r(M^*)).$$
We must show that $g$ is an isomorphism of strict Dieudonn\'e complexes.  We will construct its inverse by working in the equivalent category of strict Dieudonn\'e towers.

For each $r > 0$, the composite map
$$B^* \otimes_{A^*} M^* \overset{\rho}{\to} B^* \otimes_{A^*}^{\str} M^* \twoheadrightarrow \W_r(B^* \otimes_{A^*}^{\str} M^*)$$
factors as
$$B^* \otimes_{A^*} M^* \twoheadrightarrow \W_r(B^*) \otimes_{\W_r(A^*)} \W_r(M^*) \overset{f_r}{\to} \W_r(B^* \otimes_{A^*}^{\str} M^*),$$
by Lemma \ref{lem:Wr_tensor}.  It follows from the uniqueness part of \cite[Corollary 5.4.9]{BLM} that the maps $f_r$ are compatible with $R, F$, and $V$,
so they define a map of strict Dieudonn\'e towers
$$(f_r)_r \colon (\W_r(B^*) \otimes_{\W_r(A^*)} \W_r(M^*))_r \to (\W_r(B^* \otimes_{A^*}^{\str} M^*))_r.$$
Under the equivalence of categories between strict Dieudonn\'e complexes and strict Dieudonn\'e towers, the maps $g$ and $f_r$ respectively correspond to maps $g_r := \W_r(g)$ and $f := \limarrow_r f_r$ as in the following diagram:
\[
\begin{tikzcd}
                                                                                                                           & B^* \otimes_{A^*} M^* \arrow[ld, "\rho'"'] \arrow[rd, "\rho"] &                                                                              \\
\limarrow_n \left( \W_n(B^*) \otimes_{\W_n(A^*)} \W_n(M^*) \right) \arrow[d, two heads, "\pi'_r"'] \arrow[rr, "f", bend left=5] &                                             & B^* \otimes^{\str}_{A^*} M^* \arrow[d, two heads, "\pi_r"] \arrow[ll, "g", bend left=5] \\
\W_r(B^*) \otimes_{\W_r(A^*)} \W_r(M^*) \arrow[rr, "f_r", bend left=5]                                                       &                                             & \W_r(B^* \otimes^{\str}_{A^*} M^*) \arrow[ll, "g_r", bend left=5]             
\end{tikzcd}
\]
In the diagram above, we have
\begin{align*}
\rho' & = g \circ \rho \text{ and} \\
\pi_r \circ \rho & = f_r \circ \pi'_r \circ \rho'
\end{align*}
by construction.  It follows by passing to quotients (resp. limits) that 
\begin{align*}
\pi'_r \circ \rho' & = g_r \circ \pi_r \circ \rho \text{ and} \\
\rho & = f \circ \rho',
\end{align*}
so in particular we have
\begin{align*}
\rho & = f \circ g \circ \rho \text{ and } \\
\pi'_r \circ \rho' &  = g_r \circ f_r \circ \pi'_r \circ \rho'.
\end{align*}
By the universal properties of the maps $\rho$ and $\pi'_r \circ \rho'$, this implies that $f \circ g$ and $g_r \circ f_r$ are identity maps on the respective objects.  Since $(f_r)_r$ and $(g_r)_r$ correspond to $f$ and $g$ under the equivalence of categories, this proves that $f$ and $g$ are inverses.
\end{proof}
\end{proposition}

\begin{proof}[Proof of Proposition \ref{prop:dRW_is_qcoh}]
Set $M^* = \W\Omega^*_S \otimes^{\str}_{\W\Omega^*_R} \W\Omega^*_{R, \mathcal E}$.  To prove part (1), we will endow $M^*$ with maps $\iota^*_{r, M} \colon \dR(f_{\cris}^* \mathcal E(W_r(S), \gamma)) \to \W_r M^*$ making it a saturated de Rham--Witt complex of $(S, f_{\cris}^* \mathcal E)$.  Note that by Proposition \ref{V_etale} above and \cite[Corollary 5.3.5]{BLM}, we have isomorphisms
\begin{align*}
\W_r(M^*) & \simeq \W_r \Omega^*_S \otimes_{\W_r \Omega^*_R} \W_r\Omega^*_{R, \mathcal E} \text{ as dg-modules, and} \\
& \simeq W_r (S) \otimes_{W_r (R)} \W_r\Omega^*_{R, \mathcal E} \text{ as graded modules.}
\end{align*}
So once we prove that $M^*$ is a saturated de Rham--Witt complex associated to $(S, f_{\cris}^* \mathcal E)$, it will follow that the functor $S \mapsto \W_r \Omega^*_{S, f_{\cris}^* \mathcal E}$ on $\AffEt(X)$ is simply the quasicoherent sheaf $\widetilde{\W_r \Omega^*_{R, \mathcal E}}$ of $W_r \mathcal O$-modules, proving part (2).

We construct the maps $\iota^*_{r, M}$ as the composition
\begin{align*}
\dR(f_{\cris}^* \mathcal E(W_r(S), \gamma)) & \simeq \Omega^*_{W_r(S), \gamma} \otimes_{\Omega^*_{W_r(R), \gamma}} \dR(\mathcal E(W_r(R), \gamma)) \\
& \simeq \W_r \Omega^*_S \otimes_{\W_r \Omega^*_R} \dR(\mathcal E(W_r(R), \gamma)) \\
& \to \W_r \Omega^*_S \otimes_{\W_r \Omega^*_R} \W_r \Omega^*_{R, \mathcal E} = \W_r M^*,
\end{align*}
using the isomorphisms of Remark \ref{rmk:PD_dR_functorialities} and Remark \ref{rmk:PD_dR_dRW_etale_base_change} and the map
$$\iota^*_{r, R} \colon \dR(\mathcal E(W_r(R), \gamma)) \to \W_r \Omega^*_{R, \mathcal E}$$
which comes as part of the data of the universal de Rham--Witt module $\W\Omega^*_{R, \mathcal E}$.  Each map in this composition is a map of dg-$\Omega^*_{W_r(S), \gamma}$-modules, and is compatible with reduction maps and with Frobenius in degree 0.  Therefore the same is true of $\iota^*_{r, M}$, and thus the maps $\iota^*_{r, M}$ make $M^*$ a de Rham--Witt module over $(S, f_{\cris}^* \mathcal E)$.

Now let $N^* = (N^*, (\iota^*_{r, N})_r)$ be an arbitrary de Rham--Witt module over $(S, f_{\cris}^* \mathcal E)$.  We claim that there is a unique map $M^* \to N^*$ in $\dRWM_{S, f_{\cris}^* \mathcal E}$.  To prove this, first recall the de Rham--Witt module $f_* N^* = (N^*, (\iota_{r, f_* N}^*)_r)$ of Construction \ref{construction:dRWM_functoriality_in_R}.
By the universal property of $\W\Omega^*_{R, \mathcal E}$, there is a unique map $g \colon \W\Omega^*_{R, \mathcal E} \to f_* N^*$ in $\dRWM_{R, \mathcal E}$.  Since $N^*$ is a strict $\W\Omega^*_S$-module, $g$ factors uniquely as
$$
\W\Omega^*_{R, \mathcal E} \overset{a}{\to} \W\Omega^*_S \otimes_{\W\Omega^*_R}^{\str} \W\Omega^*_{R, \mathcal E} = M^* \overset{h}{\to} N^*,
$$
where $a = 1 \otimes \id$.  So it only remains to prove that $h \colon M^* \to N^*$ is in fact a map in $\dRWM_{S, f_{\cris}^* \mathcal E}$; that is, that $\iota^*_{r, N}$ coincides with the composition
$$
\dR(f_{\cris}^* \mathcal E(W_r(S), \gamma)) \overset{\iota^*_{r, M}}{\longrightarrow} \W_r M^* \overset{\W_r(h)}{\longrightarrow} \W_r N^*.
$$
To prove this, we chase the diagram:
\[
\begin{tikzcd}
\dR(\mathcal E(W_r(R), \gamma)) \arrow[r] \arrow[d, "{\iota^*_{r, R}}"] \arrow[rrd, "{\iota^*_{r, f_* N}}", bend left=40] & \dR(f_{\cris}^* \mathcal E(W_r(S), \gamma)) \arrow[rd, "{\iota^*_{r, N}}"] \arrow[d, "{\iota^*_{r, M}}"']          &          \\
{\W_r \Omega^*_{R, \mathcal E}} \arrow[r, "\W_r(a)"'] \arrow[rr, "\W_r(g)"', bend right=20]                           & \W_r M^* \arrow[r, "\W_r(h)"'] & \W_r N^*
\end{tikzcd}
\]
All regions except for the triangle at right commute by construction.  It follows that the triangle commutes on the image of $\dR(\mathcal E(W_r(R), \gamma)) \to \dR(f_{\cris}^* \mathcal E(W_r(S), \gamma))$.  But this image generates $\dR(f_{\cris}^* \mathcal E(W_r(S), \gamma))$ as a graded $\Omega^*_{W_r(S), \gamma}$-module, and the maps in the triangle are all maps of dg-$\Omega^*_{W_r(S), \gamma}$-modules, so the triangle commutes.
\end{proof}

\begin{remark} \label{rmk:dRW_sheaves}
As a result of Proposition \ref{prop:dRW_is_qcoh}, we can define $\W\Omega^*_{X, \mathcal E}$ and $\W_r\Omega^*_{X, \mathcal E}$ as sheaves on the \'etale site of $X$ even if $X$ is not affine; namely, they are the sheaves defined on affines by
\begin{align*}
\Spec R & \mapsto \W\Omega^*_{R, \mathcal E|_{\Spec R}} \text{ and} \\
\Spec R & \mapsto \W_r\Omega^*_{R, \mathcal E|_{\Spec R}},
\end{align*}
respectively.  By the proposition, these both exist provided that $X$ has an open cover by affines $\Spec R_i$ such that $(R_i, \mathcal E|_{\Spec R_i})$ admits a saturated de Rham--Witt complex; if so, then $\W_r\Omega^i_{X, \mathcal E}$ is a quasicoherent sheaf of $W_r \mathcal O_X$-modules for each $i$.
\end{remark}

\begin{remark}
As discussed in \textsection \ref{sec:future_directions}, we expect that it is also possible to define $\W\Omega^*_{X, \mathcal E}$ directly by a universal property analogous to Definition \ref{def:saturated_dRW}.  This will require the theory of (strict) Dieudonn\'e complexes and Dieudonn\'e towers valued in sheaves, which will be developed in \cite{joint}.
\end{remark}


\subsection{Compatibility with colimits}
\label{sec:compatibility_with_colimits}

\begin{para}
In this section, we will show that saturated de Rham--Witt complexes with coefficients behave well with respect to filtered colimits of rings, in a certain sense.  To make this precise, we introduce the following setup:
\end{para}

\begin{situation} \label{sit:crystal_and_colimit_of_rings}
Fix a $k$-algebra $R_0$ and a unit-root $F$-crystal $\mathcal E_0$ on $X_0 = \Spec R_0$.
Let $(X_i = \Spec R_i)_{i \in \mathcal I}$ be a diagram of affine $R_0$-schemes, indexed by a cofiltered category $\mathcal I$.
Let $X = \limarrow_i X_i$, which agrees with the spectrum of $R = \colim_i R_i$.
Then we have morphisms of affine schemes
\[
\begin{tikzcd}
X \arrow[rd, "g"'] \arrow[rr, "h_i"] &     & X_i \arrow[ld, "g_i"] \\
                                     & X_0 &                      
\end{tikzcd}
\]
for each $i \in \mathcal I$, and
$$
X_i \overset{f_{\alpha}}{\to} X_j
$$
for each $\alpha \colon i \to j$ in $\mathcal I$.
Pulling $\mathcal E_0$ back along the various maps yields unit-root $F$-crystals $\mathcal E_i := (g_i)^*_{\cris} \mathcal E_0$ on $X_i$ for each $i \in \mathcal I$, and similarly $\mathcal E := g^*_{\cris} \mathcal E_0$ on $X$.
\end{situation}

\begin{construction} \label{con:colimit_of_dRW_complexes_as_dRWM}
In Situation \ref{sit:crystal_and_colimit_of_rings}, suppose that $\W\Omega^*_{R_i, \mathcal E_i}$ exists for each $i$.  We wish to assemble these into a diagram indexed by $\mathcal I^{\op}$ and give their strictified colimit the structure of a de Rham--Witt module over $(R, \mathcal E)$.  The first step is to write down the transition maps $\W\Omega^*_{R_i, \mathcal E_i} \to \W\Omega^*_{R_j, \mathcal E_j}$ for each $\alpha \colon j \to i$ in $\mathcal I$.  To this end, note that the universal property of $\W\Omega^*_{R_i, \mathcal E_i}$ induces a map
$$\pi_{\alpha} \colon \W\Omega^*_{R_i, \mathcal E_i} \to f_{\alpha *} \W\Omega^*_{R_j, \mathcal E_j}$$
in $\dRWM_{R_i, \mathcal E_i}$, and recall (cf. Construction \ref{construction:dRWM_functoriality_in_R}) that the target
has the same underlying Dieudonn\'e complex as $\W\Omega^*_{R_j, \mathcal E_j}$.  So we can view $\pi_{\alpha}$ as a map $\W\Omega^*_{R_i, \mathcal E_i} \to \W\Omega^*_{R_j, \mathcal E_j}$ of strict Dieudonn\'e complexes.  The uniqueness part of the universal property of $\W\Omega^*_{R_i, \mathcal E_i}$ ensures that these maps compose correctly given morphisms $\ell \overset{\beta}{\to} j \overset{\alpha}{\to} i$ in $\mathcal I$, so that $(\W\Omega^*_{R_i, \mathcal E_i}, \pi_{\alpha})$ forms an $\mathcal I^{\op}$-shaped diagram in $\DC_{\str}$.

Let $M^*$ denote the colimit of this system; recall from Lemma \ref{lem:WSat_commutes_with_colimits} that this is the strictification of the colimit in $\DC$.  Our next task is to give $M^*$ the structure of a de Rham--Witt module over $(R, \mathcal E)$.  
Since each $\W\Omega^*_{R_i, \mathcal E_i}$ is a strict $\W\Omega^*_{R_i}$-module (compatibly with transition maps), their colimit in $\DC$ is a module over the Dieudonn\'e algebra colimit
$$\colim_{i, \DA} \W\Omega^*_{R_i}.$$
Lemma \ref{lem:WSat_module} then makes $M^*$ a module over the strictification
$$\W\Sat(\colim_{i, \DA} \W\Omega^*_{R_i}) = \colim_{i, \DA_{\str}} \W\Omega^*_{R_i}.$$
But this latter colimit equals $\W\Omega^*_R$, since $\W\Omega^*_{-}$ is a left adjoint (\cite[Definition 4.1.1]{BLM}) and thus commutes with arbitrary colimits.  So $M^*$ is a module over $\W\Omega^*_R$ in $\DC_{\str}$.

To finish making $M^*$ a de Rham--Witt module, we must endow it with a map
$$
\iota^*_r \colon \dR(\mathcal E(W_r(R), \gamma)) \to \W_r M^*
$$
for each $r > 0$, satisfying various compatibilities.  To construct $\iota^*_r$, we simply take the colimit of the corresponding maps
$$
\iota^*_{r, i} \colon \dR(\mathcal E_i(W_r(R_i), \gamma)) \to \W_r \Omega^*_{R_i, \mathcal E_i}
$$
as $i$ varies.  We identify the target of this map, $\colim_i \W_r \Omega^*_{R_i, \mathcal E_i}$, with $\W_r M^*$ using Proposition \ref{prop:TD_equivalence_compatible_with_filtered_colimits}; and we identify its source with $\dR(\mathcal E(W_r(R), \gamma))$ by the calculation
\begin{align*}
\colim_i \dR(\mathcal E_i(W_r(R_i), \gamma)) & \simeq \colim_i \left( \dR(\mathcal E_0(W_r(R_0), \gamma)) \otimes_{\Omega^*_{W_r(R_0), \gamma}} \Omega^*_{W_r(R_i), \gamma} \right) \\
& \simeq \dR(\mathcal E_0(W_r(R_0), \gamma)) \otimes_{\Omega^*_{W_r(R_0), \gamma}} \colim_i \left( \Omega^*_{W_r(R_i), \gamma} \right) \\
& \simeq \dR(\mathcal E_0(W_r(R_0), \gamma)) \otimes_{\Omega^*_{W_r(R_0), \gamma}} \Omega^*_{W_r(R), \gamma} \\
& \simeq \dR(\mathcal E(W_r(R), \gamma)).
\end{align*}
Since each $\iota^*_{r, i}$ is a morphism of dg-$\Omega^*_{W_r(R_i), \gamma}$-modules, their colimit is a morphism of dg-$\Omega^*_{W_r(R), \gamma}$-modules.  Finally, the compatibilities of the $\iota^*_r$ with quotient and Frobenius maps follow from the corresponding compatibilities of the $\iota^*_{r, i}$.
\end{construction}

\begin{proposition} \label{prop:colimit_of_dRW_complexes_is_dRW_complex}
In the situation above, if $\W\Omega^*_{R_i, \mathcal E_i}$ exists for each $i$, then the object
$$
M^* = \colim_{i, \DC_{\str}}(\W\Omega^*_{R_i, \mathcal E_i}) \in \dRWM_{R, \mathcal E}
$$
of Construction \ref{con:colimit_of_dRW_complexes_as_dRWM} is a saturated de Rham--Witt complex associated to $(R, \mathcal E)$.  In particular, we have
$$
\W_r \Omega^*_{R, \mathcal E} = \colim_i \W_r\Omega^*_{R_i, \mathcal E_i}
$$
for each $r$.
\end{proposition}
\begin{proof}
Let $N^*$ be any test object in $\dRWM_{R, \mathcal E}$; we must show there exists a unique map $M^* \to N^*$ in $\dRWM_{R, \mathcal E}$.  We first show existence.  For each $i$, pushing forward along $h_i \colon X \to X_i$ gives an object $h_{i *} N^*$ in $\dRWM_{R_i, \mathcal E_i}$.  The universal property of $\W\Omega^*_{R_i, \mathcal E_i}$ then provides a unique map
$$\psi_i \colon \W\Omega^*_{R_i, \mathcal E_i} \to h_{i *} N^*$$
in $\dRWM_{R_i, \mathcal E_i}$.  Forgetting the de Rham--Witt module structure for the moment, these are in particular maps of strict $\W\Omega^*_{R_i}$-modules $\W\Omega^*_{R_i, \mathcal E_i} \to N^*$, compatible as $i$ varies.  So passing to the strictified colimit gives a map
$$\psi \colon M^* \to N^*$$
of strict $\W\Omega^*_R$-modules.  To prove that this is a map of de Rham--Witt modules over $(R, \mathcal E)$, consider the following diagram:
\[
\begin{tikzcd}
\dR(\mathcal E_i(W_r(R_i), \gamma)) \arrow[r] \arrow[d, "{\iota^*_{r, i}}"] \arrow[rrd, "{\iota^*_{r, h_{i*} N}}", bend left=40] & \dR(\mathcal E(W_r(R), \gamma)) \arrow[d, "\iota^*_{r, M}"] \arrow[rd, "{\iota^*_{r, N}}"] &          \\
{\W_r \Omega^*_{R_i, \mathcal E_i}} \arrow[r] \arrow[rr, "\W_r(\psi_i)"', bend right=20]                                     & \W_r M^* \arrow[r, "\W_r(\psi)"]                                                  & \W_r N^*
\end{tikzcd}
\]
The various $\iota^*$ maps here are the ones giving $\W\Omega^*_{R_i, \mathcal E_i}$, $M^*$, $N^*$, and $h_{i *} N^*$ the structure of de Rham--Witt modules.  The square on the left commutes by construction of $\iota^*_{r, M}$, and the top and bottom regions commute by the constructions of $\iota^*_{r, h_{i*} N}$ and $\psi$ respectively.  The commutativity of the outer region is precisely the statement that $\psi_i \colon \W\Omega^*_{R_i, \mathcal E_i} \to h_{i*} N^*$ is a morphism in $\dRWM_{R_i, \mathcal E_i}$.  If this is the case for all $i$, then since
$$
\dR(\mathcal E(W_r(R), \gamma)) = \colim_i \dR(\mathcal E_i(W_r(R_i), \gamma)),
$$
it follows that the small triangular region commutes as well.  But this is the statement that $\psi \colon M^* \to N^*$ is a morphism in $\dRWM_{R, \mathcal E}$.

To complete the proof, we must show that $\psi$ is the unique map $M^* \to N^*$ in $\dRWM_{R, \mathcal E}$.  But given any such $\psi$, we can reconstruct the diagram above, constructing $\psi_i$ from $\psi$ instead of the other way around.  The commutativity of the small triangle implies that of the outer region, so the $\psi_i$ are morphisms of de Rham--Witt modules.  The universal property of $\W\Omega^*_{R_i, \mathcal E_i}$ forces these to agree with our earlier maps $\psi_i$, which proves the uniqueness of $\psi$.
\end{proof}


\section{Construction from a lift with Frobenius}
\label{ch:dRWLM}

\begin{para}
So far, we have only been able to construct our saturated de Rham--Witt complexes $\W\Omega^*_{R, \mathcal E}$ in the case of the trivial crystal $\mathcal E = \mathcal O$.  Our goal in this chapter is to construct $\W\Omega^*_{R, \mathcal E}$ for a general unit-root $F$-crystal $\mathcal E$, provided that $R$ comes equipped with a lift with Frobenius.
\end{para}

\begin{para}
Throughout this chapter, we fix $R$, $A$, $(A_r)_r$, and $\phi$ as in the Frobenius-lifted situation \ref{sit:lifted_situation}.  As usual, let $\mathcal E$ be a unit-root $F$-crystal on $\Spec R$.
\end{para}

\begin{para}
The main player in this chapter will be a variant of the category of de Rham--Witt modules, called the category $\dRWLM_{A, \mathcal E}$ of \emph{de Rham--Witt lift modules} over $(A, \mathcal E)$ (Definition \ref{def:dRW_lift_module}), which incorporates data coming from the lifts $A_r$ of $R$ rather than from the truncated Witt vectors $W_r(R)$.  The usefulness of this category comes from the fact that while $\dRWM_{R, \mathcal E}$ has no ``obvious'' initial object, $\dRWLM_{A, \mathcal E}$ does (Corollary \ref{cor:initial_dRWLM}), and the two categories turn out to be equivalent (Proposition \ref{prop:dRWM_dRWLM_equivalence}).  Thus, although the category $\dRWM_{R, \mathcal E}$ is more canonical (as it does not depend on a lift), $\dRWLM_{A, \mathcal E}$ is a useful setting in which to perform constructions.  The resulting construction can be viewed as a generalization of the formula $\W\Omega^*_R = \W\Sat(\widehat{\Omega}^*_A)$ of Construction \ref{con:BLM_lifted_construction}.
\end{para}


\subsection{de Rham--Witt lift modules}
\label{sec:dRWLM}

\begin{para}
Recall that for any $r > 0$, the PD-structure $\brackets$ on $(p) \subset A_r$ makes $\Spec A_r$ an object of $\Cris(\Spec R/{W_r})$.  Then by Example \ref{ex:PD_dR_DC_mod}, $\mathcal E$ determines PD-de Rham complexes
\begin{align*}
\dR(\mathcal E(A_r, \brackets)) & = (\mathcal E(A_r) \otimes_{A_r} \Omega^*_{A_r, \brackets}, \nabla) \text{ for each $r$, and} \\
\widehat{\dR}(\mathcal E(A, \brackets)) & = \limarrow_r \dR(\mathcal E(A_r, \brackets)).
\end{align*}
In fact, we have $\Omega^*_{A_r, \brackets} = \Omega^*_{A_r}$ by Lemma \ref{lem:PD_compatibility_for_free_when_I_equals_p} under our hypotheses.  Accordingly, we will usually suppress the PD-structure and simply write $\dR(\mathcal E(A_r))$ or $\widehat{\dR}(\mathcal E(A))$.
\end{para}

We begin with the following observation.  If $M^*$ is a strict $\W\Omega^*_R$-module, then Corollary \ref{cor:PD_dR_module_lift} gives $\W_r M^*$ the structure of a dg-$\Omega^*_{A_r}$-module, so in particular it is also a graded $A_r$-module.  We have the following easy analogue of Lemma \ref{lem:construct_graded_not_dg_map_non-sheafy}:

\begin{lemma} \label{lem:construct_graded_not_dg_map_lift}
Suppose we are given a $\W\Omega^*_R$-module $M^*$ in $\DC_{\str}$, equipped with an $A_r$-linear map $\lambda \colon \mathcal E(A_r) \to \W_r M^0$.  Then:
\begin{enumerate}
\item[(a)] The map $\lambda$ extends uniquely to a map of graded left $\Omega^*_{A_r}$-modules (not necessarily compatible with differentials)
$$
\lambda^* \colon \dR(\mathcal E(A_r)) = \mathcal E(A_r) \otimes_{A_r} \Omega^*_{A_r} \to \W_r M^*.
$$
\item[(b)] If we are instead given a compatible family of maps $\lambda_r \colon \mathcal E(A_r) \to \W_r M^0$ for all $r$, then the resulting maps $\lambda_r^*$ are also compatible.
\end{enumerate}
\end{lemma}
\begin{proof}
Omitted.
\end{proof}

\begin{mydef} \label{def:dRW_lift_module}
By a \emph{de Rham--Witt lift module for $(A, \mathcal E)$} we will mean a collection of the following data:  a left $\W\Omega^*_R$-module in $\DC_{\str}$, equipped with $A_r$-linear maps
$$
\lambda_r \colon \mathcal E(A_r) \to \W_r M^0
$$
for each $r > 0$, such that:
\begin{enumerate}

\item For each $r > 0$, the diagram
\[
\begin{tikzcd}
\mathcal E(A_r) \arrow[d, two heads] \arrow[r, "\lambda_r"] & \W_r M^0 \arrow[d, two heads] \\
\mathcal E(A_{r-1}) \arrow[r, "\lambda_{r-1}"]         & \W_{r-1} M^0           
\end{tikzcd}
\]
commutes, where the two vertical maps are the quotient maps.

\item For each $r > 0$, the diagram
\[
\begin{tikzcd}
\mathcal E(A_r) \arrow[d, "F"'] \arrow[r, "\lambda_r"] & \W_r M^0 \arrow[d, "F"] \\
\mathcal E(A_{r-1}) \arrow[r, "\lambda_{r-1}"]         & \W_{r-1} M^0           
\end{tikzcd}
\]
commutes, where the left vertical map is given by applying \eqref{eqn:F_on_sections} to the Frobenius map $\Spec A_{r-1} \to \Spec A_r$.

\item The maps $\lambda_r^*$ of Lemma \ref{lem:construct_graded_not_dg_map_lift} are maps of complexes
$$
(\mathcal E(A_r) \otimes_{A_r} \Omega^*_{A_r}, \nabla) \to (\W_r M^*, d).
$$
\end{enumerate}
A morphism of de Rham--Witt lift modules over $(A, \mathcal E)$ is a morphism $f \colon M^* \to N^*$ of strict $\W\Omega^*_R$-modules such that $\lambda_{r, N} = \W_r(f^0) \circ \lambda_{r, M}$ for each $r$.  We call the resulting category $\dRWLM_{A, \mathcal E}$.
\end{mydef}

\begin{remark} \label{rmk:redefine_dRWLM}
As in Remark \ref{rmk:redefine_dRWM}, we could equivalently demand the data of the extension $\lambda^*_r$, rather than only $\lambda_r = \lambda^0_r$.  The translation works exactly as before (with the result stated as Alternative Definition \ref{def:dRWL_module_again} below), except for the following.  Recall that in Definition \ref{def:dRW_module}, it was not possible to demand that the $\iota^*_r$ be compatible with Frobenius in all degrees, as their source $\dR(\mathcal E(W_r(R), \gamma))$ does not carry a divided Frobenius operator.  Here, the source of $\lambda^*_r$ does have a divided Frobenius (Example \ref{ex:PD_dR_DC_mod}), so we could demand that $\lambda^*_r$ be compatible with it in all degrees.  But in fact this comes for free, as the following lemma shows.
\end{remark}

\begin{lemma} \label{lem:dRWLM_F_compatibility_in_one_vs_all_degrees}
Let $(M^*, (\lambda_r)_r)$ be a de Rham--Witt lift module over $(A, \mathcal E)$.  Then for each $r$, the diagram
\[
\begin{tikzcd}
\dR(\mathcal E(A_r)) \arrow[r, "\lambda^*_r"] \arrow[d, "F"'] & \W_r M^* \arrow[d, "F"] \\
\dR(\mathcal E(A_{r-1})) \arrow[r, "\lambda^*_{r-1}"]         & \W_{r-1} M^*,           
\end{tikzcd}
\]
commutes, where the left vertical map is the divided Frobenius composed with the quotient map $\dR(\mathcal E(A_r)) \to \dR(\mathcal E(A_{r-1}))$.
\end{lemma}
\begin{proof}
Consider a simple tensor $e \otimes \omega \in \mathcal E(A_r) \otimes_{A_r} \Omega^i_{A_r} = \dR(\mathcal E(A_r))$.  We calculate:
\begin{align*}
\lambda^*_{r-1}(F(e \otimes \omega)) & = \lambda^*_{r-1}(\phi_{\mathcal E}(e) \otimes F \omega) \\
& = F \omega \cdot \lambda_{r-1}(\phi_{\mathcal E}(e)) \\
& = F \omega \cdot F(\lambda_r(e)) \\
& = F(\omega \cdot \lambda_r(e)) \\
& = F(\lambda^*_r(e \otimes \omega)),
\end{align*}
where the third equality comes from the given compatibility in degree $0$, and the fourth follows (after lifting $\omega$ to an element of $\widehat{\Omega}^*_A$ and $\lambda^*_r(e \otimes 1)$ to an element of $M^*$) from the compatibility of the multiplication map $\widehat{\Omega}^*_A \otimes M^* \to M^*$ with Frobenius.
\end{proof}

\begin{altdef} \label{def:dRWL_module_again}
A \emph{de Rham--Witt lift module over $(A, \mathcal E)$} is a collection of the following data:  a left $\W\Omega^*_R$-module $M^*$ in $\DC_{\str}$, equipped with maps
$$
\lambda^*_r \colon \dR(\mathcal E(A_r)) \to \W_r M^*
$$
of dg-$\Omega^*_{A_r}$-modules for each $r$, such that the following diagrams commute for all $r$:
\[
\begin{tikzcd}
\dR(\mathcal E(A_r)) \arrow[d, two heads] \arrow[r, "\lambda^*_r"] & \W_r M^* \arrow[d, two heads] \\
\dR(\mathcal E(A_{r-1})) \arrow[r, "\lambda^*_{r-1}"]              & \W_{r-1} M^*                 
\end{tikzcd}
\]
and
\[
\begin{tikzcd}
\dR(\mathcal E(A_r)) \arrow[d, "F"'] \arrow[r, "\lambda^*_r"] & \W_r M^* \arrow[d, "F"] \\
\dR(\mathcal E(A_{r-1})) \arrow[r, "\lambda^*_{r-1}"]              & \W_{r-1} M^*                 
\end{tikzcd}
\]
A morphism of de Rham--Witt lift modules over $(A, \mathcal E)$ is a morphism $f \colon M^* \to N^*$ of strict $\W\Omega^*_R$-modules such that $\lambda^*_{r, N} = \W_r(f) \circ \lambda^*_{r, M}$ for each $r$.
\end{altdef}

\begin{remark} \label{rmk:initial_dRWLM}
Unlike the situation in Remark \ref{rmk:no_obvious_initial_dRWM}, we will be able to directly construct an initial object in $\dRWLM_{A, \mathcal E}$.  Namely, recall from Example \ref{ex:PD_dR_DC_mod} that $\widehat{\dR}(\mathcal E(A))$ is a Dieudonn\'e complex, and moreover a $\widehat{\Omega}^*_A$-module in $\DC$.  By Lemma \ref{lem:WSat_module}, it follows that $\W\Sat(\widehat{\dR}(\mathcal E(A)))$ has the structure of a strict module over $\W\Sat(\widehat{\Omega}^*_A) = \W\Omega^*_R$.  We claim that this can be made into a de Rham--Witt lift module, and that it is initial.  In fact we will prove more (Proposition \ref{prop:dRWLM_equals_comma_category}):  the category of de Rham--Witt lift modules is \emph{equivalent} to the category of strict $\W\Omega^*_R$-modules equipped with a $\W\Omega^*_R$-module map from $\W\Sat(\widehat{\dR}(\mathcal E(A)))$; that is, the coslice category
$$
\W\Sat(\widehat{\dR}(\mathcal E(A))) / \W\Omega^*_R \lmodstr.
$$
This latter category has the obvious initial object $(\W\Sat(\widehat{\dR}(\mathcal E(A))), \id)$.
\end{remark}

\begin{lemma} \label{lemma:comma_category_to_dRWLM}
Suppose $M^*$ is a strict $\W\Omega^*_R$-module equipped with a $\W\Omega^*_R$-module map
$$
f \colon \W\Sat(\widehat{\dR}(\mathcal E(A))) \to M^*.
$$
\begin{enumerate}
\item For each $r$, the composition
\begin{align} \label{eqn:map_constructing_lambda}
\widehat{\dR}(\mathcal E(A)) \overset{\rho}{\to} \W\Sat(\widehat{\dR}(\mathcal E(A))) \overset{f}{\to} M^* \twoheadrightarrow \W_r M^*
\end{align}
factors uniquely as
$$
\widehat{\dR}(\mathcal E(A)) \to \dR(\mathcal E(A_r)) \overset{\lambda^*_r}{\to} \W_r M^*.
$$
\item The maps $\lambda^*_r$ give $M^*$ the structure of a de Rham--Witt lift module for $(A, \mathcal E)$.
\item The construction $(M^*, f) \mapsto (M^*, (\lambda^*_r)_r)$ defines a functor $\Theta$ from the coslice category $\W\Sat(\widehat{\dR}(\mathcal E(A))) / \W\Omega^*_R \lmodstr$ to $\dRWLM_{\mathcal E, A}$.
\end{enumerate}

\begin{proof}
To prove (1), simply note\footnote{This is well-known, but not quite trivial---it requires a few homological lemmas.  The interested reader can find a proof in \cite[\textsection 2.9]{thesis}.} that the natural map $\widehat{\dR}(\mathcal E(A)) \to \dR(\mathcal E(A_r))$ is the quotient by $p^r$, and the target $\W_r M^*$ is killed by $p^r$.

For (2), to show that the $\lambda^*_r$ satisfy the conditions of Alternative Definition \ref{def:dRWL_module_again}, note that even the maps \eqref{eqn:map_constructing_lambda} are maps of complexes, compatible with the quotient and Frobenius maps on each side.  It follows that the induced maps $\lambda^*_r$ are as well.

Finally, for (3), suppose we are given a morphism
\[
\begin{tikzcd}
                    & \W\Sat(\widehat{\dR}(\mathcal E(A))) \arrow[ld, "f_M"'] \arrow[rd, "f_N"] &     \\
M^* \arrow[rr, "g"] &                                                                           & N^*
\end{tikzcd}
\]
in the coslice category, and let $\lambda^*_{r, M}$ and $\lambda^*_{r, N}$ be the maps constructed in part (1).  We claim that $g$ is a morphism of de Rham--Witt lift modules from $(M^*, (\lambda^*_{r, M})_r)$ to $(N^*, (\lambda^*_{r, N})_r)$.  All there is to prove is that $\lambda^*_{r, N} = \W_r(g) \circ \lambda^*_{r, M}$ for each $r$.  This follows by chasing the diagram below, where the commutativity of the outer region follows from that of the inner regions plus the surjectivity of the indicated map on the left.
\begin{equation} \label{eqn:diagram_in_comma_category_lemma}
\begin{tikzcd}
                                                                                                                     &                                                                     &                                                                           & M^* \arrow[dd, "g"] \arrow[r, two heads] & \W_r M^* \arrow[dd, "\W_r(g)"] \\
\dR(\mathcal E(A_r)) \arrow[rrrru, "{\lambda^*_{r, M}}", bend left=25] \arrow[rrrrd, "{\lambda^*_{r, N}}"', bend right=25] & \widehat{\dR}(\mathcal E(A)) \arrow[r, "\rho"] \arrow[l, two heads] & \W\Sat(\widehat{\dR}(\mathcal E(A))) \arrow[ru, "f_M"] \arrow[rd, "f_N"'] &                                          &                                \\
                                                                                                                     &                                                                     &                                                                           & N^* \arrow[r, two heads]                 & \W_r N^*                      
\end{tikzcd}
\end{equation}
\end{proof}
\end{lemma}

\begin{proposition} \label{prop:dRWLM_equals_comma_category}
The functor $\Theta \colon \W\Sat(\widehat{\dR}(\mathcal E(A))) / \W\Omega^*_R \lmodstr \to \dRWLM_{\mathcal E, A}$ of Lemma \ref{lemma:comma_category_to_dRWLM} is an equivalence of categories.
\end{proposition}
\begin{proof}
We will show first that it is essentially surjective and then that it is fully faithful.  Suppose $(M^*, (\lambda^*_r)_r)$ is a de Rham--Witt lift module over $(A, \mathcal E)$, as in Alternative Definition \ref{def:dRWL_module_again}.  In particular, the maps $\lambda^*_r \colon \dR(\mathcal E(A_r)) \to \W_r M^*$ are maps of dg-$\Omega^*_{A_r}$-modules, compatible with the quotient and Frobenius maps on both sides as $r$ varies.  Passing to the limit, we get a map $\lambda^* \colon \widehat{\dR}(\mathcal E(A)) \to M^*$ of dg-$\widehat{\Omega}^*_A$-modules, which is again compatible with $F$.  By Lemma \ref{lem:description_of_algebras_and_modules_in_DC_and_DCstr}, this is a map in $\widehat{\Omega}^*_A \lmodDC$.  Thus by Lemma \ref{lem:WSat_module}, its strictification
$$
\W\Sat(\lambda^*) \colon \W\Sat(\widehat{\dR}(\mathcal E(A))) \to M^*
$$
is a map in $\W\Omega^*_R \lmodstr$.
Thus we have constructed an object $(M^*, \W\Sat(\lambda^*))$ in the coslice category $\W\Sat(\widehat{\dR}(\mathcal E(A))) / \W\Omega^*_R \lmodstr$, and one can check that the functor $\Theta$
sends this to $(M^*, (\lambda^*_r)_r)$.

To prove that $\Theta$ is fully faithful, note that it acts as the identity on the underlying strict $\W\Omega^*_R$-modules and the morphisms thereof, and the morphisms on both sides are morphisms $M^* \to N^*$ of strict $\W\Omega^*_R$-modules satisfying some extra conditions.  Namely, a morphism $g \colon M^* \to N^*$ in $\W\Omega^*_R \lmodstr$ is a morphism of de Rham--Witt lift modules if we have $\lambda^*_{r, N} = \W_r(g) \circ \lambda^*_{r, M}$ for each $r$, and it is a morphism in the coslice category if $f_N = g \circ f_M$, where $f_M$ and $f_N$ are the given maps from $\W\Sat(\widehat{\dR}(\mathcal E(A)))$.  Thus it suffices to show that these conditions are equivalent to each other.  We showed one implication already in Lemma \ref{lemma:comma_category_to_dRWLM}, and the reverse implication follows by chasing the same diagram \eqref{eqn:diagram_in_comma_category_lemma} in reverse:  if the outer region commutes, then the two maps $\widehat{\dR}(\mathcal E(A)) \to \W_r N^*$ agree for all $r$, which implies the two maps $\widehat{\dR}(\mathcal E(A)) \to N^*$ agree, and thus $f_N = g \circ f_M$ by the universal property of $\rho$.
\end{proof}

\begin{corollary} \label{cor:initial_dRWLM}
The category $\dRWLM_{A, \mathcal E}$ admits an initial object, whose underlying strict $\W\Omega^*_R$-module is $\W\Sat(\widehat{\dR}(\mathcal E(A)))$.
\end{corollary}
\begin{proof}
The coslice category $\W\Sat(\widehat{\dR}(\mathcal E(A))) / \W\Omega^*_R \lmodstr$ clearly has the initial object
$$(\W\Sat(\widehat{\dR}(\mathcal E(A))), \id).$$
The equivalence $\Theta$ sends this to an initial object of $\dRWLM_{A, \mathcal E}$ with the same underlying strict $\W\Omega^*_R$-module.
\end{proof}


\subsection{Comparing $\dRWM_{R, \mathcal E}$ and $\dRWLM_{A, \mathcal E}$}
\label{sec:comparison_dRWM_to_dRWLM}

We maintain the setup and notation of the previous section.

\begin{construction} \label{constr:dRWM_to_dRWLM}
Let $(M^*, (\iota^*_r)_r)$ be a de Rham--Witt module for $(R, \mathcal E)$, as in Alternative Definition \ref{def:dRW_module_again}.  We can give the $\W\Omega^*_R$-module $M^*$ the structure of a de Rham--Witt lift module over $(A, \mathcal E)$ as follows.  For each $r$, let
$$
h_r \colon (\Spec R \hookrightarrow \Spec W_r(R), \gamma) \to (\Spec R \hookrightarrow \Spec A_r, \brackets)
$$
be the PD-morphism which is induced by the canonical $\delta$-ring map $A \to W(R)$, and let
$$
\theta^*_r \colon \dR(\mathcal E(A_r)) \to \dR(\mathcal E(W_r(R), \gamma))
$$
be the map of dg-$\Omega^*_{A_r}$-modules induced by functoriality.
Then let $\lambda^*_r$ be the composition
$$
\lambda^*_r \colon \dR(\mathcal E(A_r)) \overset{\theta^*_r}{\to} \dR(\mathcal E(W_r(R), \gamma)) \overset{\iota^*_r}{\to} \W_r M^*,
$$
\end{construction}

\begin{proposition} \label{prop:dRWM_dRWLM_functor}
In the situation of Construction \ref{constr:dRWM_to_dRWLM}, $(M^*, (\lambda^*_r)_r)$ is a de Rham--Witt lift module over $(A, \mathcal E)$.
Moreover, the map sending $(M^*, (\iota^*_r)_r)$ to $(M^*, (\lambda^*_r)_r)$ and acting as the identity on morphisms is a functor $\Psi_{A, \mathcal E} \colon \dRWM_{R, \mathcal E} \to \dRWLM_{A, \mathcal E}$.
\begin{proof}
We must show that each $\lambda^*_r$ is a map of dg-$\Omega^*_{A_r}$-modules, and that they are compatible with quotient and Frobenius maps as $r$ varies.  In view of Lemmas \ref{lem:construct_graded_not_dg_map_lift} and \ref{lem:dRWLM_F_compatibility_in_one_vs_all_degrees}, it suffices to prove these compatibilities in degree $0$.

For each $r$, $\theta^*_r$ is a map of dg-$\Omega^*_{A_r}$-modules by construction, while $\iota^*_r$ is by definition a map of dg-$\Omega^*_{W_r(R), \gamma}$-modules, and is in particular a map of dg-$\Omega^*_{A_r}$-modules via the natural map $\Omega^*_{A_r} \to \Omega^*_{W_r(R), \gamma}$.  Therefore each $\lambda^*_r$ is a map of dg-$\Omega^*_{A_r}$-modules.
\\
\\
The maps $\theta^0_r$ are compatible with quotient and Frobenius maps, by evaluating $\mathcal E$ on the commutative squares
\[
\begin{tikzcd}
A_r \arrow[r] \arrow[d, two heads] & W_r(R) \arrow[d, two heads] \\
A_{r-1} \arrow[r]                  & W_{r-1}(R)                
\end{tikzcd}
\text{\quad and \quad}
\begin{tikzcd}
A_r \arrow[r] \arrow[d, "F"'] & W_r(R) \arrow[d, "F"] \\
A_{r-1} \arrow[r]                  & W_{r-1}(R)                
\end{tikzcd}
\]
of PD-thickenings.  The $\iota^0_r$ are compatible with quotient and Frobenius maps by definition, so it follows that the $\lambda^0_r$ are as well.  This proves that $(M^*, (\lambda^*_r)_r)$ is a de Rham--Witt lift module over $(A, \mathcal E)$.

To prove that $\Psi_{A, \mathcal E}$ is a functor, suppose we are given a map $f \colon M^* \to N^*$ which defines a morphism $(M^*, (\iota^*_{r, M})_r) \to (N^*, (\iota^*_{r, N})_r)$ in $\dRWM_{R, \mathcal E}$.  We must show that the same map $f$ defines a morphism $(M^*, (\lambda^*_{r, M})_r) \to (N^*, (\lambda^*_{r, N})_r)$ in $\dRWLM_{A, \mathcal E}$; that is, we have $\lambda^*_{r, N} = \W_r(f) \circ \lambda^*_{r, M}$ for each $r$.  This follows from chasing the diagram
\begin{equation} \label{diagram:dRWM_dRWLM_morphisms}
\begin{tikzcd}
                                                                                                                                         &                                                                                        & \W_r M^* \arrow[dd, "\W_r(f)"] \\
\dR(\mathcal E(A_r)) \arrow[r, "\theta^*_r"] \arrow[rru, "{\lambda^*_{r, M}}", bend left] \arrow[rrd, "{\lambda^*_{r, N}}"', bend right] & \dR(\mathcal E(W_r(R), \gamma)) \arrow[ru, "{\iota^*_{r, M}}"] \arrow[rd, "{\iota^*_{r, N}}"'] &                                \\
                                                                                                                                         &                                                                                        & \W_r N^*,                      
\end{tikzcd}
\end{equation}
where the commutativity of all inner regions implies that of the outer region.
\end{proof}
\end{proposition}

In fact we can prove more:

\begin{proposition} \label{prop:dRWM_dRWLM_equivalence}
The functor $\Psi_{A, \mathcal E} \colon \dRWM_{R, \mathcal E} \to \dRWLM_{A, \mathcal E}$ is an equivalence of categories.
\end{proposition}
\begin{proof}
We will show that $\Psi_{A, \mathcal E}$ is essentially surjective and fully faithful.  For essential surjectivity, suppose $(M^*, (\lambda^*_r)_r)$ is an object of $\dRWLM_{A, \mathcal E}$.  In particular, each $\lambda^*_r$ is a morphism $\dR(\mathcal E(A_r)) \to \W_r M^*$ of dg-$\Omega^*_{A_r}$-modules.  But Corollary \ref{cor:PD_dR_module_lift} shows that the dg-$\Omega^*_{A_r}$-module structure of $\W_r M^*$ factors through its dg-$\Omega^*_{W_r(R), \gamma}$-module structure, so we can factor $\lambda^*_r$ as
\[
\dR(\mathcal E(A_r)) \to \Omega^*_{W_r(R), \gamma} \otimes_{\Omega^*_{A_r}} \dR(\mathcal E(A_r)) \simeq \dR(\mathcal E(W_r(R), \gamma)) \overset{\iota^*_r}{\to} \W_r M^*,
\]
where the isomorphism comes from Remark \ref{rmk:PD_dR_functorialities}, and $\iota^*_r$ is a morphism of dg-$\Omega^*_{W_r(R), \gamma}$-modules.  The compatibilities of the $\iota^*_r$ with quotient and Frobenius maps follow from the same compatibilities of $\lambda^*_r$.  Thus $(M^*, (\iota^*_r)_r)$ is a de Rham--Witt module over $(R, \mathcal E)$, and by construction we have $\Psi_{A, \mathcal E}(M^*, (\iota^*_r)_r) = (M^*, (\lambda^*_r)_r)$.

Since $\Psi_{A, \mathcal E}$ acts as the identity on morphisms, it is clearly faithful.  To prove that it is full, we must only chase diagram \ref{diagram:dRWM_dRWLM_morphisms} in reverse:  recall that a morphism $f \colon M^* \to N^*$ in $\W\Omega^*_R \lmodstr$ defines a morphism in $\dRWM_{R, \mathcal E}$ if and only if the small triangle commutes, and it defines a morphism in $\dRWLM_{A, \mathcal E}$ if and only if the outer region commutes.  The universal property of
$$
\Omega^*_{W_r(R), \gamma} \otimes_{\Omega^*_{A_r}} \dR(\mathcal E(A_r)) \simeq \dR(\mathcal E(W_r(R), \gamma))
$$
implies that these two conditions are in fact equivalent.
\end{proof}

\begin{corollary} \label{cor:lifted_construction_as_dRWM}
Let $(M^*, (\lambda_r)_r)$ denote the initial object $\W\Sat(\limarrow_r \dR(\mathcal E(A_r)))$ of $\dRWLM_{A, \mathcal E}$.  Then $M^*$ can be endowed with maps $(\iota_r)_r$ making it an initial object of $\dRWM_{R, \mathcal E}$; that is, a saturated de Rham--Witt complex associated to $(R, \mathcal E)$.
\end{corollary}


\section{Proofs of main results}
\label{ch:main_results}

In this chapter, we prove our three main results:  the existence of our saturated de Rham--Witt complexes in general, and comparisons (for $X/k$ smooth) to crystalline cohomology and the classical de Rham--Witt complexes of \cite{etesse}.
The first two sections, comprising the general construction and comparison to cohomology, are inspired by \cite[\textsection 5]{ogus}, which does the same for the saturated de Rham--Witt complex of Bhatt--Lurie--Mathew.
Working with nontrivial coefficient crystals will require us to make a few changes, one of which (the use of PD-de Rham complexes rather than naive de Rham complexes in \textsection \ref{sec:cohomology_comparison}) turns out to be a significant simplification.

In the final section, which compares our de Rham--Witt complexes to the classical ones, the idea is to work on a pro-\'etale cover trivializing the $F$-crystal $(\mathcal E, \phi_{\mathcal E})$.  Such a cover exists by Remark \ref{rmk:proetale_trivialization}, and the results of sections \ref{sec:etale_localization} and \ref{sec:compatibility_with_colimits} allow us to express our saturated de Rham--Witt complexes on the cover in terms of those on the base scheme.


\subsection{General construction}
\label{sec:general_construction}

\begin{para} \label{para:general_construction_setup}
Throughout this section, we will work in the Frobenius-embedded situation \ref{sit:embedded_situation}, with two modifications:  we assume that $R$ is reduced (which does no harm in view of Corollary \ref{cor:R_Rred_same_dRW}), and we allow $A$ to be an arbitrary $p$-torsionfree $W$-algebra instead of a smooth one.  We let $(\mathcal E, \phi_{\mathcal E})$ be a unit-root $F$-crystal on $\Cris(\Spec R/W)$ as usual.

Additionally, let $\widetilde B$ be the quotient of the $p$-adically completed PD-envelope $B$ by the $p$-adic closure of the $p$-power torsion ideal, and $\widetilde B_r = \widetilde B/p^r \widetilde B$ for each $r > 0$.
Finally, we fix the following notation for the various affine schemes of interest to us:
\begin{align*}
X & = \Spec R, \\
Y_r & = \Spec A_r, \\
D_r & = \Spec B_r, \text{ and} \\
\widetilde{D_r} & = \Spec \widetilde{B_r}.
\end{align*}
\end{para}

\begin{remark} \label{rmk:lift_with_Frobenius_finite_type}
If $R$ is generated as a $k$-algebra by elements $\{x_i\}_{i \in I}$, we may choose $A$ to be the polynomial algebra
$$A = W[\{t_i\}_{i \in I}],$$
with the map $A \twoheadrightarrow R$ induced by $t_i \mapsto x_i$ and the Frobenius lift given by
$$
\phi \left(a \prod_i t_i^{e_i} \right) = \sigma(a) \prod_i t_i^{pe_i}.
$$
For future reference, we note that this relativizes as follows.  Given a morphism $f \colon R \to R'$ lying over a map $k \to k'$ of fields,
suppose $R$ is generated over $k$ by elements $\{x_i\}_{i \in I}$ and $R'$ is generated over $k'$ by elements $\{y_j\}_{j \in J}$.  Then consider the diagram
\[
\begin{tikzcd}
W(k)[\{t_i\}_{i \in I}] \arrow[r, hook] \arrow[d, two heads] & W(k')[\{t_i\}_{i \in I}, \{u_j\}_{j \in J}] \arrow[d, two heads] \\
R \arrow[r, "f"]   & R'
\end{tikzcd}
\]
where the left vertical map is as before, the right one sends $t_i$ to $f(x_i)$ and $u_j$ to $y_j$, and both polynomial algebras are equipped with Frobenius lifts as before.

In particular, if $R$ is a finite-type $k$-algebra (as we will assume later), then we can lift it to a polynomial algebra $A$ on finitely many generators, which is in particular smooth over $W$.  If $R \to R'$ is a morphism of finite-type algebras over $k$ and $k'$ respectively, then we can lift it to a morphism of polynomial algebras $A, A'$ on finitely many generators over $W(k)$ and $W(k')$ respectively, which are again smooth.  Note also that we are free to $p$-adically complete the polynomial algebras $A$ and $A'$ and work with the smooth $p$-adic formal schemes $\Spf A$ and $\Spf A'$.
\end{remark}

\begin{remark}
The natural map $R = A/I \to B/\overline{I}$ is an isomorphism.  The analogous statement $R = A_r/I_r \overset{\sim}{\to} B_r/\overline{I_r}$ for uncompleted PD-envelopes follows from \cite[Remark 3.20.4]{B-O}, since $(p) \subset W(k)$ is principal, and the case of the completed PD-envelope follows by quotienting by $p^r$ (cf. \cite[Remark 3.20.8]{B-O}) and passing to the limit.
\end{remark}

\begin{lemma} \label{lem:ogus_5.1}
The kernel of $B \to \widetilde B$ is a sub-PD-ideal of $(\overline I, \gamma)$.  In particular, the quotient map
$$
B \twoheadrightarrow B/\overline I = R
$$
factors through $\widetilde B$, and $\gamma$ induces PD-structures (which we will abusively also call $\gamma$) on the ideals $\ker(\widetilde B \to R) \subset \widetilde B$ and $\ker(\widetilde B_r \to R) \subset \widetilde B_r$ for each $r$.
\end{lemma}
\begin{proof}
Although our hypotheses are slightly different, the proof of \cite[Lemma 5.1]{ogus} goes through without change.
\end{proof}

\begin{remark} \label{rmk:two_D_tilde_objects}
Lemma \ref{lem:ogus_5.1} shows that for each $r > 0$, we have an object $(X \hookrightarrow \widetilde D_r, \gamma)$ of $\Cris(X/W)$.  Note that the PD-structure $\brackets$ of \ref{para:notation_PD_structures} also makes $\widetilde D_r$ into a PD-thickening of $\widetilde D_1$, leading to an object $\widetilde D_r = (\widetilde D_1 \hookrightarrow \widetilde D_r, \brackets)$ of $\Cris(\widetilde D_1/W_r)$.  Both of these objects will be useful for us; for future reference, we will let $f$ denote the natural closed embedding $X \hookrightarrow \widetilde D_1$, and $f^{\sharp} \colon \widetilde B_1 \to R$ the corresponding surjective ring map.
\end{remark}

\begin{remark} \label{rmk:E_is_a_pullback}
We will ultimately construct our saturated de Rham--Witt complex $\W\Omega^*_{R, \mathcal E}$ by applying Corollary \ref{cor:lifted_construction_as_dRWM} to the $p$-torsionfree lifting $\widetilde B$ of the $\F_p$-algebra $\widetilde B_1$, and then pulling back along $f$.  But before we can do any of this, we must specify what unit-root $F$-crystal on $\widetilde D_1$ we are working with.  This is done in the following three lemmas; the goal is to find a unit-root $F$-crystal $(\mathcal F, \phi_{\mathcal F})$ on $\Cris(\widetilde D_1/W_r)$ such that $f^*_{\cris}(\mathcal F) \simeq \mathcal E$.
\end{remark}

\begin{lemma} \label{lem:ker_fsharp_nilpotent}
For every $x \in \ker f^{\sharp}$, we have $x^p = 0$.  In particular, $f^{\sharp}$ expresses $R$ as the reduction of $\widetilde B_1$.
\end{lemma}
\begin{proof}
Since $\ker f^{\sharp}$ carries a PD-structure $\gamma$, we have for all $x \in \ker f^{\sharp}$ that $x^p = p! \gamma_p(x)$, which vanishes because $\widetilde B_1$ is an $\F_p$-algebra.  So $f^{\sharp}$ is surjective, it kills only nilpotents, and it kills all nilpotents because $R$ is reduced.
\end{proof}

\begin{lemma} \label{lem:construction_of_g}
There exists a unique morphism $g \colon \widetilde D_1 \to X$ such that $f \circ g$ is the absolute Frobenius endomorphism of $\widetilde D_1$.  Moreover, $g \circ f$ is the absolute Frobenius of $X$.
\end{lemma}
\begin{proof}
In ring-theoretic notation, this says that $\Frob_{\widetilde B_1}$ factors uniquely as
$$
\widetilde B_1 \overset{f^{\sharp}}{\to} R \overset{g^{\sharp}}{\to} \widetilde B_1,
$$
and moreover that $f^{\sharp} \circ g^{\sharp} = \Frob_R$.  
For the first statement, we need only observe that $f^{\sharp}$ is surjective and that $\Frob_{\widetilde B_1}$ annihilates $\ker f^{\sharp}$ by Lemma \ref{lem:ker_fsharp_nilpotent}.  For the second statement, we calculate that
\begin{align*}
f^{\sharp} \circ g^{\sharp} \circ f^{\sharp} & = f^{\sharp} \circ \Frob_{\widetilde B_1} \\
& = \Frob_R \circ f^{\sharp},
\end{align*}
which implies $f^{\sharp} \circ g^{\sharp} = \Frob_R$ because $f^{\sharp}$ is surjective.
\end{proof}

\begin{lemma} \label{lem:E_is_a_pullback_of_F}
Let $\mathcal F = g_{\cris}^* \mathcal E$, equipped with the Frobenius endomorphism
$$
\phi_{\mathcal F} = g_{\cris}^* \phi_{\mathcal E} \colon F_{X, \cris}^* \mathcal F \to \mathcal F.
$$
There exists an isomorphism of $F$-crystals $f_{\cris}^*(\mathcal F, \phi_{\mathcal F}) \overset{\sim}{\to} (\mathcal E, \phi_{\mathcal E})$.
\end{lemma}
\begin{proof}
By definition, we have
$$
f_{\cris}^* (\mathcal F, \phi_{\mathcal F}) = f_{\cris}^* g_{\cris}^* (\mathcal E, \phi_{\mathcal E}) = F_{X, \cris}^* (\mathcal E, \phi_{\mathcal E}),
$$
so we must give an isomorphism of $F$-crystals
$$
F_{X, \cris}^* (\mathcal E, \phi_{\mathcal E}) \simeq (\mathcal E, \phi_{\mathcal E}).
$$
But since $\mathcal E$ is a unit-root $F$-crystal, $\phi_{\mathcal E}$ itself is an isomorphism of crystals $F_{X, \cris}^* \mathcal E \to \mathcal E$, and it is moreover an isomorphism of $F$-crystals because the diagram
\[
\begin{tikzcd}
F_{X, \cris}^* F_{X, \cris}^* \mathcal E \arrow[r, "F_{X, \cris}^* \phi_{\mathcal E}"] \arrow[d, "F_{X, \cris}^* \phi_{\mathcal E}"'] & F_{X, \cris}^* \mathcal E \arrow[d, "\phi_{\mathcal E}"] \\
F_{X, \cris}^* \mathcal E \arrow[r, "\phi_{\mathcal E}"']                                                                                 & \mathcal E                                               
\end{tikzcd}
\]
commutes.
\end{proof}

\begin{theorem} \label{thm:general_construction}
Let $R$ be a reduced $k$-algebra, and let $(\mathcal E, \phi_{\mathcal E})$ be a unit-root $F$-crystal on $\Cris(\Spec R/W)$.
Then there exists a saturated de Rham--Witt complex $\W\Omega^*_{R, \mathcal E}$ associated to $(R, \mathcal E)$.  Moreover, it is isomorphic as a strict $\W\Omega^*_R$-module to the object
$$
\W\Sat(\widehat{\dR}(\mathcal F(\widetilde B)))
$$
of Corollary \ref{cor:initial_dRWLM}, viewed as a $\W\Omega^*_R$-module via the isomorphism $\W\Omega^*_{\widetilde B_1} \simeq \W\Omega^*_R$ of Remark \ref{rmk:BLM_nilpotent_invariance}.
\end{theorem}
\begin{proof}
Putting together Lemmas \ref{lem:E_is_a_pullback_of_F} and \ref{lem:ker_fsharp_nilpotent} and Propositions \ref{prop:dRWM_equiv_reduction} and \ref{prop:dRWM_dRWLM_equivalence}, we have equivalences of categories
$$
\dRWM_{R, \mathcal E} \simeq \dRWM_{R, f_{\cris}^* \mathcal F} \overset{\sim}{\to} \dRWM_{\widetilde B_1, \mathcal F} \overset{\sim}{\to} \dRWLM_{\widetilde B, \mathcal F},
$$
each preserving the underlying strict $\W\Omega^*_R$-modules.  The object we have specified is an initial object in the category on the right, so it corresponds to an initial object of the category on the left with the same underlying strict $\W\Omega^*_R$-module.
\end{proof}

\begin{corollary} \label{cor:general_construction_non-reduced}
Let $R$ be any $k$-algebra and $(\mathcal E, \phi_{\mathcal E})$ a unit-root $F$-crystal on $\Cris(\Spec R/W)$.  Then there exists a saturated de Rham--Witt complex $\W\Omega^*_{R, \mathcal E}$ associated to $(R, \mathcal E)$.
\end{corollary}
\begin{proof}
Combine Theorem \ref{thm:general_construction} with Corollary \ref{cor:R_Rred_same_dRW}.
\end{proof}


\subsection{Comparison to crystalline cohomology}
\label{sec:cohomology_comparison}

\begin{para}
Our goal in this section is to compare the saturated de Rham--Witt complex $\W\Omega^*_{X, \mathcal E}$ of Remark \ref{rmk:dRW_sheaves} to the cohomology of the $F$-crystal $\mathcal E$.  To prove this, we will first need a slightly different recipe for $\W\Omega^*_{R, \mathcal E}$ when $R$ is a smooth $k$-algebra, in terms of a de Rham complex of $\mathcal E$ rather than our auxiliary $F$-crystal $\mathcal F$.
\end{para}

\begin{para} \label{para:setup_for_smooth_affine_construction}
We maintain the setup and notation of \ref{para:general_construction_setup},
and assume moreover that $R$ is a smooth (and in particular finite-type) $k$-algebra.  In light of Remark \ref{rmk:lift_with_Frobenius_finite_type}, we can and do take $A$ to be smooth over $W$.  It follows by \cite[Corollary 3.35]{B-O} that the completed PD-envelope $B$ is $p$-torsionfree, so that $\widetilde B = B$.
\end{para}

\begin{remark} \label{rmk:three_dR_complexes_of_E}
The comparison will pass through three de Rham complexes associated to the $F$-crystals $\mathcal E$ and $\mathcal F$ (with $\mathcal E \simeq f^*_{\cris} \mathcal F$), which we now recall.  First, we have already discussed the completed PD-de Rham complex
$$
\widehat{\dR}(\mathcal F(B, \brackets)) = \limarrow_r \left( \mathcal F(B_r) \otimes_{B_r} \Omega^*_{B_r, \brackets} \right)
$$
associated to $\mathcal F$ on the formal PD-thickening $(D_1 \hookrightarrow D_r, \brackets)_r$ of $D_1 = \Spec B_1$.  (The reader may recall that we have $\Omega^*_{B_r, \brackets} = \Omega^*_{B_r}$ by Lemma \ref{lem:PD_compatibility_for_free_when_I_equals_p}; however, we will leave the PD-structure in the notation to minimize confusion.)  Closely related to this (cf. Remark \ref{rmk:two_D_tilde_objects}) is
$$
\widehat{\dR}(\mathcal E(B, \gamma)) = \limarrow_r \left( \mathcal E(B_r) \otimes_{B_r} \Omega^*_{B_r, \gamma} \right),
$$
which is associated to $\mathcal E$ on the formal PD-thickening $(X \hookrightarrow D_r, \gamma)_r$ of $X = \Spec R$.  Third, recall from Example \ref{ex:classical_dR_DC_mod} the completed de Rham complex
$$
\widehat{\dR}(\mathcal E(X \hookrightarrow Y_{\bullet})) = \limarrow_r \left( \mathcal E(B_r) \otimes_{A_r} \Omega^*_{A_r} \right),
$$
which is constructed from the PD-envelopes of the smooth embeddings $X \hookrightarrow Y_r$.

These three objects are respectively modules over $\widehat{\Omega}^*_{B, \brackets}$, $\widehat{\Omega}^*_{B, \gamma}$, and $\limarrow_r (B_r \otimes_{A_r} \Omega^*_{A_r})$ in $\DC$.  In fact the latter two are isomorphic:  \cite[\nopp 0, 3.1.6]{illusie} gives an isomorphism of the underlying algebras, and tensoring this with $\id_{\mathcal E}$ gives an isomorphism of the modules (see e.g. \cite[Proposition 2.8.5]{thesis}).
Moreover, the compatibility of the PD-structures $\gamma$ (on $I$) and $\brackets$ (on $(p) \subset I$) gives us a quotient map
$$
\widehat{\Omega}^*_{B, \brackets} \to \widehat{\Omega}^*_{B, \gamma},
$$
which has $p$-torsion kernel by Proposition \ref{prop:completed_dR_PD_only_kills_p-tors}.  This carries over to our chosen coefficients as follows.
\end{remark}

\begin{construction} \label{con:comparison_map_dR_F_to_dR_E}
Consider the PD-morphism
$$
h = \id_{D_r} \colon (X \hookrightarrow D_r, \gamma) \to (D_1 \hookrightarrow D_r, \brackets)
$$
of PD-embeddings over $f \colon X \hookrightarrow D_1$.  Remark \ref{rmk:PD_dR_functorialities} associates to this an isomorphism
$$
\Omega^*_{B_r, \gamma} \otimes_{\Omega^*_{B_r, \brackets}} \dR(\mathcal F(B_r, \brackets)) \simeq \dR(f^*_{\cris} \mathcal F(B_r, \gamma)) \simeq \dR(\mathcal E(B_r, \gamma)).
$$
of dg-$\Omega^*_{B_r, \gamma}$-modules.  Let
$$
\pi_r \colon \dR(\mathcal F(B_r, \brackets)) \to \dR(\mathcal E(B_r, \gamma))
$$
denote the resulting base change map, and
$$
\pi \colon \widehat{\dR}(\mathcal F(B, \brackets)) \to \widehat{\dR}(\mathcal E(B, \gamma))
$$
the map of limits.
\end{construction}

\begin{lemma} \label{lem:comparison_map_dR_F_to_dR_E_in_DC_mod}
The map $\pi$ is a morphism of $\widehat{\Omega}^*_{B, \brackets}$-modules in $\DC$.
\end{lemma}
\begin{proof}
It is clearly a map of dg-modules over $\widehat{\Omega}^*_{B, \brackets}$, so in light of Lemma \ref{lem:description_of_algebras_and_modules_in_DC_and_DCstr}, it suffices to prove that it is compatible with $F$.  But the Frobenius endomorphisms of $\dR(\mathcal F(B_r, \brackets))$ and $\dR(\mathcal E(B_r, \gamma))$ are defined by $\phi_{\mathcal F} \otimes F$ and $\phi_{\mathcal E} \otimes F$, respectively, where each $F$ is the divided Frobenius endomorphism of Proposition \ref{prop:F_on_PD_dR_I_not_p}.  The maps of tensor factors
\begin{align*}
\mathcal F(B_r) & \to \mathcal E(B_r) \text{ and} \\
\Omega^*_{B_r, \brackets} & \to \Omega^*_{B_r, \gamma}
\end{align*}
are both compatible with Frobenius, so $\pi_r$ and thus $\pi$ are as well.
\end{proof}

\begin{lemma} \label{lem:compare_dR_E_to_dR_F}
The strictification of $\pi$ is an isomorphism
$$
\W\Sat(\pi) \colon \W\Sat(\widehat{\dR}(\mathcal F(B, \brackets))) \stackrel{\sim}{\to} \W\Sat(\widehat{\dR}(\mathcal E(B, \gamma)))
$$
of strict modules over $\W\Omega^*_{B_1} = \W\Omega^*_R$.
\end{lemma}
\begin{proof}
By Lemma \ref{lem:WSat_module}, $\W\Sat(\pi)$ is a morphism of strict modules over $\W\Sat(\widehat{\Omega}^*_{B, \brackets})$.  But recall that $\W\Sat(\widehat{\Omega}^*_{B, \brackets})$ is isomorphic to $\W\Omega^*_{B_1}$ by Corollary \ref{cor:PD_dR_computes_dRW}, and thus also to $\W\Omega^*_R$ by Remark \ref{rmk:BLM_nilpotent_invariance}.  To show that $\W\Sat(\pi)$ is an isomorphism, first note that each $\pi_r$ is surjective with exact $p$-torsion kernel, since this is true of the quotient map
$$
\Omega^*_{B_r, \brackets} \to \Omega^*_{B_r, \gamma}
$$
(cf. Corollary \ref{cor:PD_dR_exact_p-torsion}) and we are tensoring this with the flat $B_r$-module $\mathcal F(B_r)$.  It follows by passage to the limit (and a Mittag-Leffler calculation analogous to that of Proposition \ref{prop:completed_dR_PD_only_kills_p-tors}) that $\pi$ itself is surjective with exact $p$-torsion kernel.  Thus $\Sat(\pi)$ is an isomorphism, as is $\W\Sat(\pi)$.
\end{proof}

\begin{corollary} \label{cor:dRW_construction_from_classical_dR_E}
There exist maps $\iota^*_r \colon \dR(\mathcal E(W_r(R), \gamma)) \to \W_r \Sat(\widehat{\dR}(\mathcal E(X \hookrightarrow Y_{\bullet})))$ making $\W\Sat(\widehat{\dR}(\mathcal E(X \hookrightarrow Y_{\bullet})))$ a saturated de Rham--Witt complex of $\mathcal E$ over $R$.
\end{corollary}
\begin{proof}
We showed in Theorem \ref{thm:general_construction} that $\W\Sat(\widehat{\dR}(\mathcal F(B, \brackets)))$ is a saturated de Rham--Witt complex, so we can simply transport this structure across the isomorphisms
$$
\W\Sat(\widehat{\dR}(\mathcal F(B, \brackets))) \simeq \W\Sat(\widehat{\dR}(\mathcal E(B, \gamma))) \simeq \W\Sat(\widehat{\dR}(\mathcal E(X \hookrightarrow Y_{\bullet}))),
$$
where the former is Lemma \ref{lem:compare_dR_E_to_dR_F} and the latter comes from strictifying the isomorphism discussed in Remark \ref{rmk:three_dR_complexes_of_E}.
\end{proof}

\begin{lemma} \label{lem:dR_E_quasi-Cartier}
The strictification map
$$
\rho \colon \widehat{\dR}(\mathcal E(X \hookrightarrow Y_{\bullet})) \to \W\Sat(\widehat{\dR}(\mathcal E(X \hookrightarrow Y_{\bullet})))
$$
and the induced map
$$
\rho_r \colon \widehat{\dR}(\mathcal E(X \hookrightarrow Y_{\bullet}))/p^r = \dR(\mathcal E(X \hookrightarrow Y_r)) \to \W_r\Sat(\widehat{\dR}(\mathcal E(X \hookrightarrow Y_{\bullet})))
$$
are quasi-isomorphisms.
\end{lemma}
\begin{proof}
This is analogous to \cite[Corollary 5.4]{ogus}, and the proof goes through without change, except that we must replace the reference to \cite[\nopp 8.20]{B-O} with its unit-root generalization, \cite[Corollary 7.3.6]{ogus_F-crystals}.
(Note that \cite[Theorem 1.8]{ogus} applies to the quasi-isomorphisms at both finite and infinite level.)
\end{proof}

\begin{proposition} \label{prop:functoriality_of_non-derived_part_of_comparison}
The morphisms of Corollary \ref{cor:dRW_construction_from_classical_dR_E} and Lemma \ref{lem:dR_E_quasi-Cartier} are functorial in the data $(\Spec R \hookrightarrow (\Spec A_r/W_r(k))_r, \phi_A)$.
That is, suppose we are given a commutative diagram
\[
\begin{tikzcd}
X' \arrow[r, hook] \arrow[d] & (Y'_r)_r \arrow[r] \arrow[d] & \Spec W_r(k') \ar[d] \\
X \arrow[r, hook]   & (Y_r)_r \arrow[r] & \Spec W_r(k),
\end{tikzcd}
\]
where each row is as in \ref{para:setup_for_smooth_affine_construction} and the maps $Y'_r \to Y_r$ are compatible with Frobenius lifts.  Suppose $\mathcal E$ is a unit-root $F$-crystal on $\Cris(X/W(k))$ and $\mathcal E'$ is its pullback to $\Cris(X'/W(k'))$.  Then the resulting diagrams
\[
\begin{tikzcd}
\dR(\mathcal E(X \hookrightarrow Y_r)) \arrow[r, "\rho_r"] \arrow[d]           & \W_r \Sat(\widehat{\dR}(\mathcal E(X \hookrightarrow Y_{\bullet})) \arrow[r, "\sim"] \arrow[d]  & \W_r \Omega^*_{R, \mathcal E} \arrow[d] \\
\dR(\mathcal E'(X' \hookrightarrow Y'_r)) \arrow[r, "\rho'_r"] & \W_r \Sat(\widehat{\dR}(\mathcal E'(X' \hookrightarrow Y'_{\bullet}))) \ar[r, "\sim"] & \W_r \Omega^*_{R', \mathcal E'}
\end{tikzcd}
\]
and
\[
\begin{tikzcd}
\widehat{\dR}(\mathcal E(X \hookrightarrow Y_{\bullet})) \arrow[r, "\rho"] \arrow[d]           & \W\Sat(\widehat{\dR}(\mathcal E(X \hookrightarrow Y_{\bullet}))) \arrow[r, "\sim"] \arrow[d]  & \W\Omega^*_{R, \mathcal E} \arrow[d] \\
\widehat{\dR}(\mathcal E'(X' \hookrightarrow Y'_{\bullet})) \arrow[r, "\rho'"] & \W\Sat(\widehat{\dR}(\mathcal E'(X' \hookrightarrow Y'_{\bullet}))) \ar[r, "\sim"] & \W\Omega^*_{R', \mathcal E'}
\end{tikzcd}
\]
commute, where the rightmost vertical maps are given by Remark \ref{rmk:functorialities_of_WOmega}.
\end{proposition}
\begin{proof}
We will prove the infinite-level statement, which immediately implies the mod-$p^r$ statement.
The compatibility of the left-hand square follows from the functoriality of the strictification map.  As for the right-hand square, this is a tedious exercise in unraveling the proofs of Corollary \ref{cor:dRW_construction_from_classical_dR_E} and its precursor Theorem \ref{thm:general_construction}, which we will describe now.
By the universal property of the saturated de Rham--Witt complex $\W\Omega^*_{R, \mathcal E}$, it suffices to prove that the given map
\begin{equation} \label{eqn:compatibility_map_of_dRW_construction_from_dR_E}
\W\Sat(\widehat{\dR}(\mathcal E(X \hookrightarrow Y_{\bullet}))) \to \W\Sat(\widehat{\dR}(\mathcal E'(X' \hookrightarrow Y'_{\bullet})))
\end{equation}
is a morphism in $\dRWM_{R, \mathcal E}$ once we use Corollary \ref{cor:dRW_construction_from_classical_dR_E}, Theorem \ref{thm:general_construction}, and Construction \ref{construction:dRWM_functoriality_in_R} to give its source and target this structure.

We are given maps $R \to R'$, $A_r \to A'_r$, and $W_r(k) \to W_r(k')$ of rings, which induce maps $h_r^{\sharp} \colon B_r \to B'_r$ of PD-envelopes for each $r$.  Applying Lemma \ref{lem:construction_of_g} to both PD-envelopes leads to a commutative diagram of rings
\[
\begin{tikzcd}
B_1 \arrow[r, "f^{\sharp}"] \arrow[d, "h_1^{\sharp}"] \arrow[rr, "\Frob_{B_1}", bend left] & R \arrow[r, "g^{\sharp}"] \arrow[d] & B_1 \arrow[d, "h_1^{\sharp}"] \\
B_1' \arrow[r, "f'^{\sharp}"] \arrow[rr, "\Frob_{B_1'}"', bend right]                      & R' \arrow[r, "g'^{\sharp}"]          & B_1',                         
\end{tikzcd}
\]
from which we see that $\mathcal F' = g_{\cris}'^* \mathcal E'$ is the pullback of
$\mathcal F = g_{\cris}^* \mathcal E$ along $h_1 \colon \Spec B_1' \to \Spec B_1$.  This gives us a commutative diagram
\[
\begin{tikzcd}
{\widehat{\dR}(\mathcal F(B, \brackets))} \arrow[d] \arrow[r, "\pi"] & {\widehat{\dR}(\mathcal E(B, \gamma))} \arrow[d] \arrow[r, "\sim"] & \widehat{\dR}(\mathcal E(X \hookrightarrow Y_{\bullet})) \arrow[d] \\
{\widehat{\dR}(\mathcal F'(B', \brackets))} \arrow[r, "\pi'"]        & {\widehat{\dR}(\mathcal E'(B', \gamma))} \arrow[r, "\sim"]          & \widehat{\dR}(\mathcal E'(X' \hookrightarrow Y'_{\bullet}))
\end{tikzcd}
\]
relating the maps of Remark \ref{rmk:three_dR_complexes_of_E} and Construction \ref{con:comparison_map_dR_F_to_dR_E}, where the vertical morphisms are the natural pullback maps.  As $\pi$ and $\pi'$ become isomorphisms after applying $\W\Sat$, we may thus identify the given map \eqref{eqn:compatibility_map_of_dRW_construction_from_dR_E} with the map
\begin{equation} \label{eqn:another_compatibility_map_of_dRW_construction_from_dR_E}
\W\Sat(\widehat{\dR}(\mathcal F(B, \brackets)) \to \W\Sat(\widehat{\dR}(\mathcal F'(B', \brackets)).
\end{equation}
By Corollary \ref{cor:initial_dRWLM}, these two objects are initial objects of $\dRWLM_{B, \mathcal F}$ and $\dRWLM_{B', \mathcal F'}$ respectively when endowed with the obvious maps
\begin{align*}
\lambda^*_r \colon \dR(\mathcal F(B_r, \brackets)) & \to \W_r \Sat(\widehat{\dR}(\mathcal F(B, \brackets)) \\
\lambda'^*_r \colon \dR(\mathcal F'(B'_r, \brackets)) & \to \W_r \Sat(\widehat{\dR}(\mathcal F'(B', \brackets)).
\end{align*}
Next, the proof of Theorem \ref{thm:general_construction} consists of turning these objects into initial objects of the categories $\dRWM_{R, \mathcal E}$ and $\dRWM_{R', \mathcal E'}$ by passing through various equivalences of categories.  Unwinding these functors, this amounts to factoring $\lambda^*_r$ and $\lambda'^*_r$ (uniquely) as the compositions across the rows of the following diagram:
\[
\begin{tikzcd}[column sep=small]
{\dR(\mathcal F(B_r, \brackets))} \arrow[d] \arrow[r] & {\dR(\mathcal F(W_r(B_1), \brackets))} \arrow[d] \arrow[r] & {\dR(\mathcal E(W_r(R), \gamma))} \arrow[d] \arrow[r, "\iota^*_r"] & {\W_r \Sat(\widehat{\dR}(\mathcal F(B, \brackets)))} \arrow[d] \\
{\dR(\mathcal F'(B_r', \brackets))} \arrow[r]         & {\dR(\mathcal F'(W_r(B_1'), \brackets))} \arrow[r]         & {\dR(\mathcal E'(W_r(R'), \gamma))} \arrow[r, "\iota'^*_r"]            & {\W_r \Sat(\widehat{\dR}(\mathcal F'(B', \brackets)))}
\end{tikzcd}
\]
The two leftmost squares and the outer rectangle of this diagram commute, so the uniqueness of the factorization implies that the rightmost square commutes as well.  This is precisely the compatibility of $\iota$ maps that we require in order for \eqref{eqn:another_compatibility_map_of_dRW_construction_from_dR_E} to be a morphism in $\dRWM_{R, \mathcal E}$, where we have implicitly applied the pushforward functor of Construction \ref{construction:dRWM_functoriality_in_R} to its target.
This map is also evidently compatible with module structures, so it is indeed a morphism of de Rham--Witt modules as claimed.
\end{proof}

\begin{para} \label{para:local_model_for_smooth_formal_embeddings}
We now leave the affine setting and assume only that $X/k$ is smooth and comes equipped with a closed embedding into a smooth $p$-adic formal scheme $\mathcal Y/\Spf W(k)$ which is equipped with a Frobenius lift $\phi_{\mathcal Y}$ (over $\sigma \colon W(k) \to W(k)$) of $\mathcal Y \times_{\Spf W} \Spec k$.
(By Remark \ref{rmk:lift_with_Frobenius_finite_type}, such an embedding exists for any smooth \emph{affine} $X/k$, and this relativizes for morphisms $X' \to X$ of smooth affine schemes over $\Spec k' \to \Spec k$.)  For each $r > 0$, we let $Y_r$ denote the scheme $\mathcal Y \times_{\Spf W} \Spec W_r(k)$.
Note that locally on $\mathcal Y$, any such embedding has the form
$$
\Spec R \hookrightarrow \Spf A,
$$
where $R$ is a smooth $k$-algebra and $A$ is a smooth and $p$-adically complete $W$-algebra equipped with a Frobenius lift $\phi_A$.  (In fact there exists a basis of affine open subschemes $\Spec R \subset X$ admitting such an embedding:  given an open neighborhood $U$ of $x \subset X$, one can choose an open formal subscheme $V \subset \mathcal Y$ whose preimage equals $U$, and one can then localize further within $V$.)  Similarly, given a map
\[
\begin{tikzcd}
X' \arrow[r, hook] \arrow[d] & \mathcal Y' \arrow[d] \arrow[r] & \Spf W' \arrow[d] \\
X \arrow[r, hook]   & \mathcal Y \arrow[r] & \Spf W
\end{tikzcd}
\]
of two such embeddings (which we always assume is compatible with Frobenius lifts), we can localize first on $\mathcal Y$ and then on $\mathcal Y'$ to put it in the form
\[
\begin{tikzcd}
\Spec R' \arrow[r, hook] \arrow[d] & \Spf A' \arrow[d] \arrow[r] & \Spf W' \arrow[d] \\
\Spec R \arrow[r, hook]   & \Spf A \arrow[r] & \Spf W.
\end{tikzcd}
\]
\end{para}

\begin{proposition} \label{prop:comparison_map_embeddable_case}
Suppose $X \hookrightarrow \mathcal Y/W$ is as in \ref{para:local_model_for_smooth_formal_embeddings}, and let $\mathcal E$ be a unit-root F-crystal on $\Cris(X/W)$.  Then we have quasi-isomorphisms
\begin{align*}
\dR(\mathcal E_{X \hookrightarrow Y_r}) & \to \W_r \Omega^*_{X, \mathcal E} \\
\widehat{\dR}(\mathcal E_{X \hookrightarrow \mathcal Y}) & \to \W\Omega^*_{X, \mathcal E}
\end{align*}
of Zariski sheaves (of $W_r$-modules and $W$-modules respectively) on $X$.
These are both functorial in the data $(X \hookrightarrow \mathcal Y/W, \phi_{\mathcal Y})$ in the sense that given compatible morphisms $f \colon X' \to X$ and $g \colon \mathcal Y' \to \mathcal Y$ over $\Spf W' \to \Spf W$, the diagram
\[
\begin{tikzcd}
\dR(\mathcal E_{X \hookrightarrow Y_r}) \arrow[r] \arrow[d]           & {\mathcal W_r \Omega^*_{X, \mathcal E}} \arrow[d]          \\
f_* \dR((f^*_{\cris} \mathcal E)_{X' \hookrightarrow Y'_r}) \arrow[r] & {f_* (\mathcal W_r \Omega^*_{X', f^*_{\cris} \mathcal E})}
\end{tikzcd}
\]
commutes for each $r$, and similarly when passing to the limit.
\end{proposition}
\begin{proof}
To construct the isomorphism
$$
\dR(\mathcal E_{X \hookrightarrow Y_r}) \simeq \W_r \Omega^*_{X, \mathcal E},
$$
we work on affine open subschemes $\Spec R \subset X$ on which $X \hookrightarrow \mathcal Y$ restricts to an affine embedding $\Spec R \hookrightarrow \Spf A$.  On each such open subscheme, Corollary \ref{cor:dRW_construction_from_classical_dR_E} and Lemma \ref{lem:dR_E_quasi-Cartier} define a quasi-isomorphism
$$
\dR(\mathcal E(\Spec R \hookrightarrow \Spec A_r)) \to \W_r \Omega^*_{R, \mathcal E|_{\Spec R}}.
$$
By Proposition \ref{prop:functoriality_of_non-derived_part_of_comparison}, these glue to a map of complexes of sheaves
$$
\dR(\mathcal E_{X \hookrightarrow Y_r}) \to \W_r \Omega^*_{X, \mathcal E},
$$
which we claim is also a quasi-isomorphism.  Indeed, the cone of this map is a complex of sheaves which is exact when evaluated on a basis of affines in $X$; thus it is exact on stalks and therefore also at the level of sheaves.
This yields the isomorphism at finite levels.  We then obtain the infinite-level isomorphism by applying $R\limarrow_r$.  Note here that $(\dR(\mathcal E_{X \hookrightarrow Y_r}))_r$ and $(\W_r \Omega^*_{X, \mathcal E})_r$ are towers of quasicoherent sheaves with surjective transition maps, so they satisfy
\begin{align*}
R\limarrow_r \dR(\mathcal E_{X \hookrightarrow Y_r}) & = \limarrow_r \dR(\mathcal E_{X \hookrightarrow Y_r}) = \widehat{\dR}(\mathcal E_{X \hookrightarrow \mathcal Y}) \text{ and} \\
R\limarrow_r \W_r \Omega^*_{X, \mathcal E} & = \limarrow_r \W_r \Omega^*_{X, \mathcal E} = \W\Omega^*_{X, \mathcal E}
\end{align*}
by \cite[tag \texttt{0BKS}]{stacks}.
Functoriality in the data $(X \hookrightarrow \mathcal Y/W, \phi_{\mathcal Y})$ is guaranteed by Proposition \ref{prop:functoriality_of_non-derived_part_of_comparison}.
\end{proof}

\begin{para}
In order to prove Proposition \ref{prop:comparison_map_embeddable_case} without assuming the existence of a smooth embedding with Frobenius, we will need the method of cohomological descent; cf. e.g. \cite[\textsection V, 3.4]{berthelot}.
Our setup is as follows.
\end{para}

\begin{construction} \label{con:hypercover_setup}
Let $X$ be a smooth $k$-scheme and $(\mathcal E, \phi_{\mathcal E})$ a unit-root $F$-crystal on $X$.
Choose an affine open cover $(U_i)_{i \in I}$ of $X$.
For each $i$, fix a smooth formal embedding $U_i \hookrightarrow \mathcal Y_i/W$ with Frobenius lift, as in Remark \ref{rmk:lift_with_Frobenius_finite_type}.
For every nonempty finite subset $J \subset I$, we get a smooth formal embedding
\begin{equation} \label{eqn:U_J_to_Y_J}
U_J := \bigcap_{i \in J} U_i \hookrightarrow \prod_{i \in J} \mathcal Y_i =: \mathcal Y_J,
\end{equation}
where the product is taken over $\Spf W$ and $\mathcal Y_J$ carries a Frobenius lift defined component-wise.
Then we can assemble the open subschemes $U_J$ into a hypercover $X_{\bullet}$ of $X$, with
$$
X_n = \coprod_{|J| = n} U_J,
$$
and \eqref{eqn:U_J_to_Y_J} defines a closed embedding of the semi-simplicial scheme $X_{\bullet}$ into a smooth formal semi-simplicial scheme $\mathcal Y_{\bullet}$ equipped with Frobenius lift.
\end{construction}

\begin{remark}
To prevent any confusion between the semi-simplicial embedding $X_{\bullet} \hookrightarrow \mathcal Y_{\bullet}$ of Construction \ref{con:hypercover_setup} and the $p$-adic towers of lifts $X \hookrightarrow Y_{\bullet}$, we will refrain from using the latter notation for the remainder of this section.  When necessary, we will denote the $p$-adic tower associated to the semi-simplicial formal scheme $\mathcal Y_{\bullet}$ as $(\mathcal Y_{\bullet, r})_r$.
\end{remark}

\begin{theorem} \label{thm:dRW_computes_Ru_*_E}
Let $X$ be a smooth $k$-scheme and $\mathcal E$ a unit-root F-crystal on $\Cris(X/W)$.  Then there are canonical isomorphisms
$$
Ru_{X/W_r *} \mathcal E \simeq \W_r\Omega^*_{X, \mathcal E}
$$
in $D(X, W_r)$ for each $r > 0$, and
$$
Ru_{X/W *} \mathcal E \simeq \W\Omega^*_{X, \mathcal E}
$$
in $D(X, W)$.
\end{theorem}
\begin{proof}
We will prove the finite-level statement; the infinite-level statement follows as in Proposition \ref{prop:comparison_map_embeddable_case}, using the isomorphism
$$
Ru_{X/W *} \mathcal E \simeq R\limarrow_r Ru_{X/W_r *} \mathcal E
$$
of \cite[\nopp 7.22.2]{B-O}.
Choose an affine open cover $\{U_i\}$, and let $X_{\bullet}$ be the resulting hypercover, equipped with a smooth formal embedding $X_{\bullet} \hookrightarrow \mathcal Y_{\bullet}$ with Frobenius lift as in Construction \ref{con:hypercover_setup}.
Viewing $\mathcal E|_{X_{\bullet}}$ as a sheaf in the topos $(X_{\bullet}/W_r)_{\cris}$ (cf. \cite[\textsection V, 3.4.1]{berthelot}), the embedding $X_{\bullet} \hookrightarrow \mathcal Y_{\bullet}$ induces a quasi-isomorphism
$$
\dR(\mathcal E_{X_{\bullet} \hookrightarrow \mathcal Y_{\bullet, r}}) \to \W_r \Omega^*_{X_{\bullet}, \mathcal E|_{X_{\bullet}}}
$$
of sheaves of $W_r$-modules on $X_{\bullet}$ by Proposition \ref{prop:comparison_map_embeddable_case}.
This becomes an isomorphism in the derived category $D(X_{\bullet}, W_r)$.
From \cite[Theorem 7.1(2)]{B-O}, we also have an isomorphism
$$
Ru_{X_{\bullet}/W_r *} \mathcal E|_{X_{\bullet}} \simeq \dR(\mathcal E_{X_{\bullet} \hookrightarrow \mathcal Y_{\bullet, r}})
$$
in $D(X_{\bullet}, W_r)$;
composing this with the above yields an isomorphism
\begin{equation} \label{eqn:cohomology_comparison_on_hypercover}
Ru_{X_{\bullet}/W_r *} \mathcal E|_{X_{\bullet}} \simeq \W_r \Omega^*_{X_{\bullet}, \mathcal E|_{X_{\bullet}}}
\end{equation}
in $D(X_{\bullet}, W_r)$.
Now let $\pi_{\zar}$ and $\pi_{\cris}$ denote the canonical morphisms of topoi
\begin{align*}
\pi_{\zar} \colon X_{\bullet, \zar} & \to X_{\zar} \\
\pi_{\cris} \colon (X_{\bullet}/W_r)_{\cris} & \to (X/W_r)_{\cris},
\end{align*}
whose pullback functors are the usual pullbacks of sheaves.
Combining \eqref{eqn:cohomology_comparison_on_hypercover} with the theorem of cohomological descent (\cite[\textsection V, Th\'eor\`eme 3.4.8]{berthelot}), we obtain isomorphisms
\begin{align*}
Ru_{X/W_r *} \mathcal E & \simeq Ru_{X/W_r *} R \pi_{\cris *} (\pi^*_{\cris} \mathcal E) \\
& \simeq R\pi_{\zar *} Ru_{X_{\bullet}/W_r *}(\mathcal E|_{X_{\bullet}}) \\
& \simeq R\pi_{\zar *} \W_r \Omega^*_{X_{\bullet}, \mathcal E|_{X_{\bullet}}} \\
& \simeq R\pi_{\zar *} \pi^*_{\zar} \W_r \Omega^*_{X, \mathcal E} \\
& \simeq \W_r \Omega^*_{X, \mathcal E}
\end{align*}
in $D(X_{\zar}, W_r)$.

We claim that this isomorphism does not depend on the choice of the open cover $\{U_i\}$ and the embeddings with Frobenius $(U_i \hookrightarrow \mathcal Y_i/W, \phi_{\mathcal Y_i})$.  Indeed, if we are given two choices $\{U_i \hookrightarrow \mathcal Y_i/W, \phi_{\mathcal Y_i}\}$ and $\{U'_j \hookrightarrow \mathcal Y'_j/W, \phi_{\mathcal Y'_j}\}$ of such data with one factoring through the other, the functorialities of Proposition \ref{prop:comparison_map_embeddable_case} and \cite[Theorem 7.1(2)]{B-O} ensure that the resulting isomorphisms agree.  In general, given only $(U_i \hookrightarrow \mathcal Y_i/W, \phi_{\mathcal Y_i})$ and $\{U'_j \hookrightarrow \mathcal Y'_j/W, \phi_{\mathcal Y'_j}\}$, we may reduce to the case above by forming a common refinement of the open covers and taking the product of the embeddings and Frobenius lifts.
\end{proof}

\begin{corollary} \label{cor:cohomology_comparison}
Under the same hypotheses as Theorem \ref{thm:dRW_computes_Ru_*_E}, there are canonical isomorphisms
$$
\mathbb H^i(X_{\zar}, \W_r \Omega^*_{X, \mathcal E}) \simeq H^i((X/W_r)_{\cris}, \mathcal E)
$$
for each $r > 0$, and
$$
\mathbb H^i(X_{\zar}, \W\Omega^*_{X, \mathcal E}) \simeq H^i((X/W)_{\cris}, \mathcal E).
$$
\end{corollary}
\begin{proof}
Apply $R^i \Gamma(X_{\zar}, -)$ to each part of Theorem \ref{thm:dRW_computes_Ru_*_E}.
\end{proof}

\begin{remark} \label{rmk:compatibilities_of_cohomology_comparison}
The isomorphisms of Theorem \ref{thm:dRW_computes_Ru_*_E} enjoy various compatibilities.  We will state them
at the level of $Ru_{X/W_r *} \mathcal E$; of course, all of the compatibilities readily pass to the limit and to cohomology.
\end{remark}

\begin{proposition} \label{prop:compatibilities_of_cohomology_comparison}
Fix $r$, and let
$$
\tau_{X, \mathcal E} \colon Ru_{X/W_r *} \mathcal E \simeq \W_r \Omega^*_{X, \mathcal E}
$$
denote the isomorphism of Theorem \ref{thm:dRW_computes_Ru_*_E} in $D(X, W_r)$.  Then:

\begin{enumerate}
\item The isomorphism
$$
\tau_{X, \mathcal O} \colon Ru_{X/W_r *} \mathcal O_{X/S} \simeq \W_r \Omega^*_{X, \mathcal O_{X/S}} = \W_r \Omega^*_X
$$
is compatible with multiplicative structures.

\item $\tau_{X, \mathcal E}$ is compatible with the module structures of the two sides over $Ru_{X/W_r *} \mathcal O_{X/W}$ resp. $\W_r \Omega^*_X$ (via the isomorphism $\tau_{X, \mathcal O}$).

\item $\tau_{X, \mathcal E}$ is functorial in the $F$-crystal $\mathcal E$.

\item $\tau_{X, \mathcal E}$ is functorial in $X$ in the sense that given a morphism $f \colon X' \to X$ over $\Spec k' \to \Spec k$, the square
\[
\begin{tikzcd}
Ru_{X/W_r *} \mathcal E \arrow[rr, "\sim", "{\tau_{X, \mathcal E}}"'] \arrow[d]                             &  & {\W_r \Omega^*_{X, \mathcal E}} \arrow[d]         \\
Rf_* Ru_{X'/W_r *} (f^*_{\cris} \mathcal E) \arrow[rr, "\sim", "{Rf_* (\tau_{X', f^*_{\cris} \mathcal E})}"'] &  & {Rf_* \W_r \Omega^*_{X', f^*_{\cris} \mathcal E}}
\end{tikzcd}
\]
commutes, where the left vertical arrow is the functoriality map of \cite[0, 3.2.5(a)]{illusie}, and the right one is the derived version of the morphism defined on affines by Remark \ref{rmk:functorialities_of_WOmega}.

\item (Compatibility with Frobenius) The diagram
\[
\begin{tikzcd}
Ru_{X/W_r *} \mathcal E \arrow[d] \arrow[rr, "\sim", "\tau_{X, \mathcal E}"'] &  & \W_r \Omega^*_{X, \mathcal E} \arrow[dd, "\W_r(\alpha_F)"] \\
R(F_X)_* Ru_{X/W_r *} (F_X)^*_{\cris} \mathcal E \arrow[d, "\sim", "\phi_{\mathcal E}"'] &  &  \\
R(F_X)_* Ru_{X/W_r *} \mathcal E \arrow[rr, "\sim", "R(F_X)_* \tau_{X, \mathcal E}"'] &  & R(F_X)_* \W_r \Omega^*_{X, \mathcal E}
\end{tikzcd}
\]
commutes, where the upper left vertical morphism is the functoriality map of \cite[0, 3.2.5(a)]{illusie} and $\alpha_F$ is the map defined by $p^i F$ in degree $i$.

\end{enumerate}
\end{proposition}
\begin{proof}
The first three compatibilities follow from stepping through all of the quasi-isomorphisms we composed, for any choice of open cover and embeddings with Frobenius.
Statement (4) follows from choosing affine open covers and embeddings $\{(U_i \hookrightarrow \mathcal Y_i/W, \phi_{\mathcal Y_i})\}$ of $X$ and $\{(U'_j \hookrightarrow \mathcal Y'_j/W', \phi_{\mathcal Y'_j})\}$ of $X'$ as in \ref{para:local_model_for_smooth_formal_embeddings}, and then tracing through the proof of Theorem \ref{thm:dRW_computes_Ru_*_E} and using the functorialities of Proposition \ref{prop:comparison_map_embeddable_case} and \cite[Theorem 7.1(2)]{B-O}.

As for (5), by applying part (4) to the absolute Frobenius morphism $F_X \colon X \to X$ (over $F_k$) and part (3) to $\phi_{\mathcal E}$, we see that the diagram
\[
\begin{tikzcd}
Ru_{X/W_r *} \mathcal E \arrow[d] \arrow[rrr, "\sim", "\tau_{X, \mathcal E}"'] & & & \W_r \Omega^*_{X, \mathcal E} \arrow[d] \\
R(F_X)_* Ru_{X/W_r *} (F_X)^*_{\cris} \mathcal E \arrow[d, "\sim", "\phi_{\mathcal E}"'] \arrow[rrr, "\sim", "R(F_X)_* \tau_{X, (F_X)^*_{\cris} \mathcal E}"'] & & & R(F_X)_* \W_r \Omega^*_{X, (F_X)^*_{\cris} \mathcal E} \arrow[d] \\
R(F_X)_* Ru_{X/W_r *} \mathcal E \arrow[rrr, "\sim", "R(F_X)_* \tau_{X, \mathcal E}"'] & & & R(F_X)_* \W_r \Omega^*_{X, \mathcal E}
\end{tikzcd}
\]
commutes, where the right-hand vertical arrows are induced by applying Remark \ref{rmk:functorialities_of_WOmega} to $F_X$ and $\phi_{\mathcal E}$ respectively.  On affine opens $\Spec R \subset X$, the composition of these two morphisms is given by $\W_r$ of the de Rham--Witt module map
$$
\W\Omega^*_{R, \mathcal E} \to (F_{\Spec R})_* \phi_{\mathcal E}^* \W\Omega^*_{R, \mathcal E},
$$
which equals $\alpha_F$ by Proposition \ref{prop:alpha_F_on_WOmega_R,E}.
\end{proof}


\subsection{Comparison to the classical de Rham--Witt complex with coefficients}
\label{sec:comparison_to_etesse}

\begin{para}
Fix, as usual, a $k$-algebra $R$ and a unit-root $F$-crystal $\mathcal E$ on $\Spec R$.  Then \'Etesse (\cite[\textsection II, D\'efinition 1.1.7]{etesse}) defines the \emph{classical de Rham--Witt complex with coefficients}
$$
W\Omega^*_{R, \mathcal E} = \limarrow_r W_r \Omega^*_{R, \mathcal E},
$$
where
$$
W_r \Omega^*_{R, \mathcal E} = \dR(\mathcal E(W_r(R), \gamma)) \otimes_{\Omega^*_{W_r(R), \gamma}} W_r \Omega^*_R.
$$
Recall from Theorem \ref{thm:BLM_comparison} that Bhatt--Lurie--Mathew constructs maps
\begin{align*}
\zeta_r \colon W_r \Omega^*_R & \to \W_r \Omega^*_R \text{ for each } r, \text{ and} \\
\zeta \colon W\Omega^*_R & \to \W\Omega^*_R,
\end{align*}
for any $R$, which are isomorphisms if $R$ is regular Noetherian.
Our goal in this section is to prove an analogous result in the case of unit-root coefficients.  Namely, given a unit-root $F$-crystal $\mathcal E = (\mathcal E, \phi_{\mathcal E})$ on $\Cris(\Spec R/W(k))$, we will construct maps
\begin{align*}
\zeta_{r, \mathcal E} \colon W_r \Omega^*_{R, \mathcal E} & \to \W_r \Omega^*_{R, \mathcal E} \text{ for each } r, \text{ and} \\
\zeta_{\mathcal E} \colon W\Omega^*_{R, \mathcal E} & \to \W\Omega^*_{R, \mathcal E},
\end{align*}
and we will show that these maps are isomorphisms if $R$ is a smooth $k$-algebra.
\end{para}

\begin{remark}
In view of \'Etesse's comparison (\cite[\textsection II, Th\'eor\`eme 2.1.1]{etesse}) between $Ru_{X/W_r *} \mathcal E$ and the classical de Rham--Witt complex $W_r \Omega^*_{X, \mathcal E}$, this gives an alternative proof of Theorem \ref{thm:dRW_computes_Ru_*_E}.
\end{remark}

The proof of the comparison will rely on the following classical fact about unit-root $F$-crystals on a smooth $k$-scheme:

\begin{remark} \label{rmk:proetale_trivialization}
For $X/k$ smooth, it follows from a theorem of Katz (\cite[Proposition 4.1.1]{katz}) that unit-root $F$-crystals on $X$ can be classified in terms of \'etale $\Z_p$-local systems.\footnote{The connection between this and Katz's original proposition is well-known to the experts but poorly documented in the literature.  The interested reader may consult \cite[\textsection 2.5]{thesis} for an exposition.}
It follows that every unit-root $F$-crystal $(\mathcal E, \phi_{\mathcal E})$ on $\Cris(X/W(k))$ admits a pro-\'etale trivialization in the following sense:  there exists an affine \'etale cover
$$
(\Spec S_{i, 1} \to X)_{i \in I}
$$
trivializing the restriction of $(\mathcal E, \phi_{\mathcal E})$ to $\Cris(X/W_1(k))$, and each $\Spec S_{i, 1}$ admits a tower of \emph{finite} \'etale covers $\Spec S_{i, r}$ (for $r \geq 1$) trivializing the restriction of $(\mathcal E, \phi_{\mathcal E})$ to $\Cris(X/W_r(k))$.  Thus, setting $S_{i, \infty} = \colim_r S_{i, r}$, the pullback of $(\mathcal E, \phi_{\mathcal E})$ to the affine scheme $\Spec S_{i, \infty}$ is isomorphic to a finite direct power of the trivial crystal
$$
(\mathcal O_{\Spec S_{i, \infty}/W}, F).
$$
If $X$ is quasicompact, then we can take the index set $I$ to be finite, and then the cover
$$
\coprod_i \Spec S_{i, \infty} \to X
$$
trivializing $(\mathcal E, \phi_{\mathcal E})$ is a pro-\'etale affine cover in the sense of \cite[Definition 4.2.1]{BS}.
\end{remark}

\begin{construction} \label{con:map:from_etesse}
Given a $k$-algebra $R$ and a unit-root $F$-crystal $(\mathcal E, \phi_{\mathcal E})$ on $\Spec R$, the map
$$
\iota_r^* \colon \dR(\mathcal E(W_r(R), \gamma)) \to \W_r \Omega^*_{R, \mathcal E}
$$
of dg-$\Omega^*_{W_r(R), \gamma}$-modules induces a map
$$
\zeta_{r, \mathcal E} \colon W_r \Omega^*_{R, \mathcal E} := \dR(\mathcal E(W_r(R), \gamma)) \otimes_{\Omega^*_{W_r(R), \gamma}} W_r \Omega^*_R \to \W_r \Omega^*_{R, \mathcal E}
$$
of dg-$W_r \Omega^*_R$-modules, where the target is a dg-$W_r \Omega^*_R$-module by Lemma \ref{lem:Vr+dVr_submodule} and the comparison map $\zeta_r$.  These maps are compatible as $r$ varies, so passing to the limit yields a map
$$
\zeta_{\mathcal E} \colon W\Omega^*_{R, \mathcal E} \to \W\Omega^*_{R, \mathcal E}
$$
of dg-$W\Omega^*_R$-modules.
\end{construction}

\begin{lemma} \label{lem:etale_base_change_of_comparison_maps}
Suppose we have an \'etale map $f \colon \Spec S \to \Spec R$ and a unit-root $F$-crystal $\mathcal E$ on $\Cris(\Spec R/W)$.  Then the map
$$
\zeta_{r, f^*_{\cris} \mathcal E} \colon W_r \Omega^*_{S, f^*_{\cris} \mathcal E} \to \W_r \Omega^*_{S, f^*_{\cris} \mathcal E}
$$
is the base change of
$$
\zeta_{r, \mathcal E} \colon W_r \Omega^*_{R, \mathcal E} \to \W_r \Omega^*_{R, \mathcal E}
$$
along $W_r(R) \to W_r(S)$, where we have identified the source and target objects using Remark \ref{rmk:etale_base_change_of_etesse} and Proposition \ref{prop:dRW_is_qcoh} respectively.
\end{lemma}
\begin{proof}
We must show that the right-hand square of the diagram
\[
\begin{tikzcd}
\dR(\mathcal E(W_r(R), \gamma)) \arrow[r, two heads]   \ar[d]          & \dR(\mathcal E(W_r(R), \gamma)) \otimes_{\Omega^*_{W_r(R), \gamma}} W_r \Omega^*_R \arrow[d] \arrow[r, "\zeta_{r, \mathcal E}"]   & \W_r \Omega^*_{R, \mathcal E} \arrow[d]   \\
\dR(f^*_{\cris} \mathcal E(W_r(S), \gamma)) \arrow[r, two heads] & \dR(f^*_{\cris} \mathcal E(W_r(S), \gamma)) \otimes_{\Omega^*_{W_r(S), \gamma}} W_r \Omega^*_S \arrow[r, "\zeta_{r, f^*_{\cris} \mathcal E}"] & \W_r \Omega^*_{S, f^*_{\cris} \mathcal E},
\end{tikzcd}
\]
commutes, where the vertical maps are induced by Remarks \ref{rmk:PD_dR_functorialities} and \ref{rmk:functorialities_of_WOmega}, and are all base change maps along $W_r(R) \to W_r(S)$.
In fact the right vertical map defines a morphism
$$
\W\Omega^*_{R, \mathcal E} \to f_* \W\Omega^*_{S, f^*_{\cris} \mathcal E}
$$
in $\dRWM_{R, \mathcal E}$; that is, the maps $\iota^*_{r, \mathcal E}$ are compatible with $\iota^*_{r, f^*_{\cris} \mathcal E}$.  This is precisely the commutativity of the outer rectangle, which implies that of the square in question.
\end{proof}

\begin{lemma} \label{lem:classical_to_saturated_comparisons_agree_for_O}
For the trivial $F$-crystal $(\mathcal O, F)$, we have equalities of comparison maps
\begin{align*}
\zeta_{r, \mathcal O} & = \zeta_r \colon W_r \Omega^*_R \to \W_r \Omega^*_R \text{ and} \\
\zeta_{\mathcal O} & = \zeta \colon W\Omega^*_R \to \W\Omega^*_R,
\end{align*}
where we have identified $W\Omega^*_{R, \mathcal O}$ with $W\Omega^*_R$ and $\W\Omega^*_{R, \mathcal O}$ with $\W\Omega^*_R$, using Proposition \ref{prop:dRW_with_trivial_coefficients} for the latter.
\end{lemma}
\begin{proof}
We will prove the finite-level statement, which of course implies the statement on limits.
By the construction of $\zeta_{r, \mathcal O}$, it suffices to prove that the diagram
\[
\begin{tikzcd}
{\Omega^*_{W_r(R), \gamma}} \arrow[d, two heads] \arrow[rd] &                 \\
W_r \Omega^*_R \arrow[r, "\zeta_r"]                       & \W_r \Omega^*_R
\end{tikzcd}
\]
of differential graded algebras commutes.  But this diagram commutes on the image of $W_r(R) = \Omega^0_{W_r(R), \gamma}$, which generates the PD-de Rham complex as a differential graded algebra.
\end{proof}

\begin{theorem} \label{thm:etesse_comparison}
Let $R$ be a smooth $k$-algebra, and let $(\mathcal E, \phi_{\mathcal E})$ be a unit-root $F$-crystal on $\Cris(\Spec R/W(k))$.
Then the comparison maps $\zeta_{r, \mathcal E}$ and $\zeta_{\mathcal E}$ are isomorphisms.
\end{theorem}
\begin{proof}
Remark \ref{rmk:proetale_trivialization} allows us to choose an \'etale cover $\Spec S_1 \to R$ and finite \'etale covers
$$
\cdots \to \Spec S_2 \to \Spec S_1
$$
such that for each $n$, the restriction of $(\mathcal E, \phi_{\mathcal E})$ to $\Cris(\Spec R/W_n)$ is trivialized on $\Spec S_n$.  Let $S_{\infty}$ denote the colimit of the $R$-algebras $S_n$, and let $f_n \colon \Spec S_n \to \Spec R$ and $f_{\infty} \colon \Spec S_{\infty} \to \Spec R$ denote the various covers.

For each $r$ and $n$, Lemma \ref{lem:etale_base_change_of_comparison_maps} tells us that the maps
$$
\zeta_{r, f^*_{n, \cris} \mathcal E} \colon W_r \Omega^*_{S_n, f^*_{n, \cris} \mathcal E} \to \W_r \Omega^*_{S_n, f^*_{n, \cris} \mathcal E}
$$
are given by base-changing
$$
\zeta_{r, \mathcal E} \colon W_r \Omega^*_{R, \mathcal E} \to \W_r \Omega^*_{R, \mathcal E}
$$
along $W_r(R) \to W_r(S_n)$.  Passing to colimits as $n \to \infty$ and using Proposition \ref{prop:colimit_of_dRW_complexes_is_dRW_complex}, it follows that
$$
\zeta_{r, f^*_{\infty, \cris} \mathcal E} \colon W_r \Omega^*_{S_{\infty}, f^*_{\infty, \cris} \mathcal E} \to \W_r \Omega^*_{S_{\infty}, f^*_{\infty, \cris} \mathcal E}
$$
is the base-change of $\zeta_{r, \mathcal E}$ along $W_r(R) \to W_r(S_{\infty})$.

Since the $F$-crystal $f^*_{\infty, \cris} \mathcal E$ is trivial, it follows from Lemma \ref{lem:classical_to_saturated_comparisons_agree_for_O} and taking finite direct sums that $\zeta_{r, f^*_{\infty, \cris} \mathcal E}$ is an isomorphism.  But the pro-\'etale cover $\Spec S_{\infty} \to \Spec R$ is faithfully flat and we have
$$
\zeta_{r, f^*_{\infty, \cris} \mathcal E} = \zeta_{r, \mathcal E} \otimes_{W_r(R)} W_r(S_{\infty}),
$$
so faithfully flat descent tells us that $\zeta_{r, \mathcal E}$ must also be an isomorphism.  It follows by passage to the limit that $\zeta_{\mathcal E}$ is as well.
\end{proof}

\begin{corollary} \label{cor:etesse_comparison_non-affine}
Suppose $X$ is a $k$-scheme and $(\mathcal E, \phi_{\mathcal E})$ is a unit-root $F$-crystal on $\Cris(X/W)$.  Then we have compatible maps
\begin{align*}
W_r \Omega^*_{X, \mathcal E} & \to \W_r \Omega^*_{X, \mathcal E} \text{ for each } r, \text{ and} \\
W\Omega^*_{X, \mathcal E} & \to \W\Omega^*_{X, \mathcal E},
\end{align*}
which are isomorphisms if $X/k$ is smooth.
\end{corollary}
\begin{proof}
The maps are defined on affines by Construction \ref{con:map:from_etesse}, and they glue by Lemma \ref{lem:etale_base_change_of_comparison_maps}.  If $X/k$ is smooth, then they are isomorphisms by working affine-locally and applying Theorem \ref{thm:etesse_comparison}.
\end{proof}

\begin{remark}
Suppose $X/k$ is smooth.  Since the classical de Rham--Witt complex $W\Omega^*_{X, \mathcal E}$ is defined without reference to a Frobenius endomorphism on $\mathcal E$, it follows from Corollary \ref{cor:etesse_comparison_non-affine} that $\W\Omega^*_{X, \mathcal E}$ is determined as a complex (and even as a dg-$\W\Omega^*_X$-module) by the crystal $\mathcal E$ without its Frobenius endomorphism.
Moreover, the finite-level comparison implies that $\W_r \Omega^*_{X, \mathcal E}$ is determined by the restriction of $\mathcal E$ to $\Cris(X/W_r)$.
We are not aware of a direct proof of either of these observations, or whether they hold true for $X/k$ not smooth.
\end{remark}

\newpage

\printbibliography

\end{document}